\documentclass[11pt]{amsart}

\usepackage{lmodern} 
\usepackage{euscript}
\usepackage{amsmath, amsthm, amssymb, amsfonts}
\usepackage{mathrsfs}
\usepackage{longtable}
\usepackage{mathtools}

\usepackage{tikz}
\usetikzlibrary{decorations.markings,decorations.pathmorphing}
\usepackage{tikz-cd}
\usetikzlibrary{arrows}
\tikzset{
  closed/.style = {decoration = {markings, mark = at position 0.5 with { \node[transform shape, xscale = .8, yscale=.4] {/}; } }, postaction = {decorate} }
}

\usepackage{caption}

\usepackage[all,arc,2cell]{xy}
\usepackage{enumerate}
\usepackage{todonotes}

\usepackage{bbm} 
\usepackage{hyperref}  
\hypersetup{%
  bookmarksnumbered=true,%
  colorlinks=true,%
  linkcolor=blue,%
  citecolor=blue,%
  filecolor=blue,%
  menucolor=blue,%
  urlcolor=blue,%
  pdfnewwindow=true,%
  pdfstartview=FitBH}

\theoremstyle{plain} 
\newtheorem{thm}{Theorem}[section]
\newtheorem{cor}[thm]{Corollary}
\newtheorem{prop}[thm]{Proposition}
\newtheorem{lem}[thm]{Lemma}

\theoremstyle{definition}
\newtheorem{defn}[thm]{Definition}
\newtheorem{exmp}[thm]{Example}
\newtheorem{notn}[thm]{Notation}

\theoremstyle{remark}
\newtheorem{rmk}[thm]{Remark}

\newcommand{\BA}{{\mathbb{A}}}

\newcommand{\BC}{{\mathbb{C}}}

\newcommand{\BG}{{\mathbb{G}}}
\newcommand{\BH}{{\mathbb{H}}}

\newcommand{\BP}{{\mathbb{P}}}
\newcommand{\BQ}{{\mathbb{Q}}}

\newcommand{\BZ}{{\mathbb{Z}}}

\newcommand{\CC}{{\mathcal C}}
\newcommand{\CD}{{\mathcal D}}
\newcommand{\CE}{{\mathcal E}}
\newcommand{\CF}{{\mathcal F}}
\newcommand{\CG}{{\mathcal G}}
\newcommand{\CH}{{\mathcal H}}
\newcommand{\CI}{{\mathcal I}}
\newcommand{\CJ}{{\mathcal J}}
\newcommand{\CK}{{\mathcal K}}
\newcommand{\CL}{{\mathcal L}}
\newcommand{\CM}{{\mathcal M}}
\newcommand{\CN}{{\mathcal N}}
\newcommand{\CO}{{\mathcal O}}
\newcommand{\CP}{{\mathcal P}}
\newcommand{\CQ}{{\mathcal Q}}
\newcommand{\CR}{{\mathcal R}}

\newcommand{\CU}{{\mathcal U}}
\newcommand{\CV}{{\mathcal V}}

\newcommand{\CY}{{\mathcal Y}}

\newcommand{\Fj}{{\mathfrak{j}}}

\newcommand{\FF}{{\mathfrak{F}}}
\newcommand{\Chow}{\mathrm{CH}}
\newcommand{\td}{\mathrm{td}}
\newcommand{\ch}{\mathrm{ch}}
\newcommand{\Corr}{\mathrm{Corr}}
\newcommand{\Coh}{\mathrm{Coh}}

\newcommand*\dd{\mathop{}\!\mathrm{d}}

\DeclareFontFamily{OT1}{rsfs}{}
\DeclareFontShape{OT1}{rsfs}{n}{it}{<-> rsfs10}{}
\DeclareMathAlphabet{\curly}{OT1}{rsfs}{n}{it}
\renewcommand\hom{\curly H\!om}
\def\endo{\curly E\!nd}

\newcommand\Ext{\operatorname{Ext}}
\newcommand\Hom{\operatorname{Hom}}

\newcommand\Spec{\operatorname{Spec}}
\newcommand\Aut{\operatorname{Aut}}
\newcommand\Isom{\operatorname{Isom}}

\DeclareMathOperator{\Jac}{Jac}
\DeclareMathOperator{\Pic}{Pic}
\DeclareMathOperator{\Prym}{Prym}
\DeclareMathOperator{\Nm}{Nm}
\DeclareMathOperator{\Fix}{Fix}
\DeclareMathOperator{\Sing}{Sing}
\DeclareMathOperator{\Br}{Br}

\newcommand{\et}{\mathrm{\acute{e}t}}
\newcommand{\inte}{\mathrm{int}}
\newcommand{\pl}{\mathrm{pl}}

\def\id{\mathrm{id}}
\def\codim{\mathrm{codim}}
\def\Supp{\mathrm{Supp}}
\def\AJ{\mathrm{AJ}}
\def\AP{\mathrm{AP}}

\def\ra{\rightarrow}
\def\ot{\otimes}

\def\ul{\underline}
\def\ol{\overline}
\def\ul{\underline}

\def\tilC{\widetilde{\CC}}
\def\p23{p_{23}}
\def\q23{q_{23}}
\newcommand{\EM}{\EuScript{M}}
\newcommand{\ER}{\EuScript{R}}
\newcommand{\EC}{\EuScript{C}}
\newcommand{\EJ}{\EuScript{J}}
\newcommand{\EP}{\EuScript{P}}

\newcommand{\SX}{\mathscr{X}}

\DeclareFontFamily{U}{wncy}{}
\DeclareFontShape{U}{wncy}{m}{n}{<->wncyr10}{}
\DeclareSymbolFont{mcy}{U}{wncy}{m}{n}
\DeclareMathSymbol{\Sha}{\mathord}{mcy}{"58} 

\begin{document}
\title[Autoduality of compactified Pryms]{Autoduality of compactified Pryms for \'etale double covers of curves with planar singularities}
\date{\today}

\author[H. Yu]{Huishi Yu}
\address{Peking University}
\email{yuhuishi@pku.edu.cn}

\begin{abstract}
We construct a Poincar\'e sheaf on the compactified Prym variety associated with an \'etale double cover of integral curves with planar singularities, and prove that the associated Fourier--Mukai transform is an autoequivalence of its derived category. 

As an application, we prove the motivic decomposition conjecture of Corti--Hanamura for the Laza--Sacc\`a--Voisin fibration, and construct a multiplicative motivic perverse filtration lifting the cohomological one. 

\end{abstract}

\maketitle

\setcounter{tocdepth}{1} 

\tableofcontents
\setcounter{section}{-1}

\section{Introduction}
Let $A$ be an abelian variety over an algebraically closed field $k$, and let $\hat{A}$ be its dual. The Poincar\'e line bundle $\CP$ on $A \times \hat{A}$ is the universal family of algebraically trivial line bundles on $A$. In the seminal work \cite{Mukai81}, Mukai proved that the Fourier--Mukai functor with kernel $\CP$ is an equivalence between the bounded derived categories $D^b(A)$ and~$D^b(\hat{A})$. 

In particular, this applies when $A$ is either the Jacobian of a smooth curve or the Prym variety of a finite cover of smooth curves, and it is natural to extend the theory to curves with singularities. In \cite{Ari10b}, Arinkin extended the Fourier--Mukai transform to compactified Jacobians of integral projective curves with planar singularities. Melo--Rapagnetta--Viviani \cite{MRV19a, MRV19b} later generalized this result to fine compactified Jacobians of reduced projective curves with planar singularities. 
Based on the previous results on the (fine) compactified Jacobians, Franco--Hanson--Ruano \cite{FHR22} extended the Fourier--Mukai transform to compactified Prym varieties of flat and ramified covers $\widetilde{C} \ra C$, with $C$ being smooth curves and $\widetilde{C}$ being projective reduced curves with planar singularities. A similar result was also proved by Groechenig--Shen \cite{GS22} independently. 

The aim of this paper is to extend the Fourier--Mukai transform to the compactified Prym varieties for \emph{\'etale double covers of projective integral curves with planar singularities}. 

\subsection{Main results}
Throughout this article, let $k$ be an algebraically closed field of characteristic 0. All schemes are over $k$ and are implicitly assumed to be locally of finite type and separated. 

Let $\CC \ra B$ be a projective flat family of integral curves with planar singularities, over an integral scheme $B$. Let $\pi: \widetilde{\CC} \ra \CC$ be an \emph{\'etale} double cover such that the geometric fibers of $\widetilde{\CC} \ra B$ are integral curves, and let $\iota \colon \widetilde{\CC} \to \widetilde{\CC}$ be the associated involution. 

Let $J_{\widetilde{\CC}}$ be the relative Jacobian and $\ol{J}_{\widetilde{\CC}}$ the relative compactified Jacobian. We denote by $P = \Prym(\widetilde{\CC}/\CC)$ the relative Prym variety. Let $\ol{P}$ be the relative compactified Prym variety, which is a natural compactification of $P$. Let $\ol{\CP}$ be the normalized Poincar\'e sheaf on $\ol{J}_{\widetilde{\CC}} \times_B \ol{J}_{\widetilde{\CC}}$. 

\begin{thm}[{Theorem~\ref{thm: def of G bar}, Proposition~\ref{prop: uniqueness, flat}}]\label{thm A}
    There exists a coherent sheaf $\ol{\CG}$ on $\ol{P} \times_B \ol{P}$ that satisfies the following properties:
    \begin{enumerate}
        \item $\ol{\CG}$ is trivialized along the zero sections of both factors, i.e., 
        \[\ol{\CG}|_{\ol{P}\times_B 0} \cong \ol{\CG}|_{0 \times_B \ol{P}} \cong \CO_{\ol{P}};\]
        \item $((1-\iota^{*}) \times \id_{\ol{P}})^{*}(\ol{\CG}|_{P \times_B \ol{P}}) \cong \ol{\CP}|_{J_{\widetilde{\CC}}\times_B \ol{P}}$;
        \item $\ol{\CG}$ is flat over both projections $p_i \colon \ol{P}\times_B \ol{P} \ra \ol{P}$ for $i=1,2$, and the restriction $\ol{\CG}|_{\ol{P}_b \times \{F\}}$ is a maximal Cohen--Macaulay sheaf for every $b\in B$ and $F \in \ol{P}_b$;
        \item $\ol{\CG}$ is invariant under the permutation of factors.
    \end{enumerate}
\end{thm}

The sheaf $\ol{\CG}$ is called the \emph{normalized Poincar\'e sheaf} on $\ol{P} \times_B \ol{P}$. It is an analogue of the classical Poincar\'e line bundle. 

\begin{thm}[{Theorem~\ref{thm: equivalence}}]\label{thm B}
    The Fourier--Mukai functor
    \[\Phi_{\ol{\CG}} \colon D^b(\ol{P}) \ra D^b(\ol{P}),\quad F \mapsto Rp_{2,*}(p_1^{*}F \ot^L \ol{\CG})\]
    is an equivalence of categories. 
\end{thm}

Let $\Pic_{(\ol{P}/B)}^{=}$ be the moduli space of relative torsion-free rank 1 sheaves on $\ol{P}$. Then $\ol{\CG}$ induces a morphism
\[\ol{\rho} \colon \ol{P} \ra \Pic_{(\ol{P}/B)}^{=}.\]
We denote by $\Pic_{(\ol{P}/B)}^0$ the identity component of $\Pic_{(\ol{P}/B)}$, and by $\ol{\Pic}_{(\ol{P}/B)}^0$ its scheme-theoretic closure in $\Pic_{(\ol{P}/B)}^{=}$. 

\begin{thm}[{Theorem~\ref{thm: autoduality}}]\label{thm: C}
    The morphism $\ol{\rho} \colon \ol{P} \xrightarrow{\sim} \ol{\Pic}_{(\ol{P}/B)}^0$ is an isomorphism of $B$-schemes. 
\end{thm}

After the construction of the normalized Poincar\'e sheaf, the proofs of Theorems~\ref{thm B} and \ref{thm: C} are similar to the proof for compactified Jacobians (\cite{Arinkin10}, \cite{Ari10b}). 

\begin{rmk}\label{rmk: universal Prym}
    Although the proofs of Theorems~\ref{thm A}, \ref{thm B}, and \ref{thm: C} are written in the scheme-theoretic setting, under the additional assumption that $\ol{P}$ is nonsingular, the arguments are intended to be carried out on the universal compactified Prym stack $\ol{\EP} \to \ER_g$ introduced in Section~\ref{subsec: universal Prym}. Therefore, these theorems hold for arbitrary families.     
    We briefly explain here how to construct the Poincar\'e sheaf on $\ol{\EP}$. After pulling back to a smooth presentation $U\to \mathcal \ER_g$ and working on each irreducible component of $U$, we are in the scheme-theoretic situation considered in Sections \ref{sec: construction} and \ref{sec_properties}, and hence there exists a normalized Poincar\'e sheaf. 
    Once rigidifications along the zero sections are fixed, the normalized Poincar\'e sheaves are uniquely isomorphic by an isomorphism compatible with the rigidifications; see Remark~\ref{rmk: uniqueness of CG}. Thus it carries natural descent data and descends to a coherent sheaf on $\ol{\EP} \times_{\ER_g} \ol{\EP}$, which is the desired Poincar\'e sheaf on $\ol{\EP}$. 
\end{rmk}

\subsection{Applications}
In \cite{MSY23}, Maulik--Shen--Yin studied the dualizable abelian fibrations that satisfy the Fourier vanishing condition (FV). They proved that such fibrations have many good properties, and their main application is that the relative compactified Jacobian $\ol{J}_{\CC}$ is a dualizable abelian fibration that satisfies the Fourier vanishing condition (FV), provided that the curves are of planar singularities and the total space $\ol{J}_{\CC}$ is nonsingular. 

In this paper, we give other examples of dualizable abelian fibrations that satisfy (FV). 

\begin{thm}[{Theorem~\ref{thm: Prym is dualizable}}]\label{thm: compactified Prym is dualizable abelian fibration s.t. FV}
    Let $\CC \ra B$ be a projective flat family of integral curves with planar singularities over an irreducible base $B$. 
    Let $\pi: \widetilde{\CC} \ra \CC$ be an \'etale double cover such that the geometric fibers of $\widetilde{\CC} \ra B$ are integral. If the total space of the relative compactified Prym variety $\ol{P}$ is nonsingular, then $\ol{P} \ra B$ is a self-dualizable abelian fibration satisfying (FV).
\end{thm}

Theorem~\ref{thm: compactified Prym is dualizable abelian fibration s.t. FV} can be applied to the compactification of relative intermediate Jacobian fibration associated with a general cubic 4-fold constructed by Laza--Sacc\`a--Voisin. For a general cubic 4-fold $X$, we denote by $U\subset B= (\BP^5)^{\vee}$ the locus of smooth hyperplane sections of $X$, and $\CJ_U \ra U$ the associated intermediate Jacobian fibration. Laza--Sacc\`a--Voisin \cite{LSV17} constructed a smooth projective compactification $\ol{\CJ}$ of $\CJ_U$, which is a hyper-K\"ahler variety of OG10 type and admits a Lagrangian fibration $\ol{\CJ} \ra B$ that extends $\CJ_U \ra U$. We call $\ol{\CJ} \ra B$ an \emph{LSV fibration}. Since $\ol{\CJ}$ arises as the descent of a nonsingular relative compactified Prym variety, Theorem~\ref{thm: compactified Prym is dualizable abelian fibration s.t. FV}, together with a descent argument, yields the following.

\begin{thm}[{Theorem~\ref{thm: LSV is dualizable}}]\label{thm: LSV is dualizable abelian fibration s.t. FV}
    The LSV fibration $\ol{\CJ} \to B$ is a self-dualizable abelian fibration satisfying (FV).
\end{thm}

The proofs of Theorems~\ref{thm: compactified Prym is dualizable abelian fibration s.t. FV} and \ref{thm: LSV is dualizable abelian fibration s.t. FV} parallel the arguments of \cite[Theorem 0.2]{MSY23}. The main input is Arinkin's dimension bound for the compactified Prym variety. 

Using the properties of dualizable abelian fibrations that satisfy (FV)~(\cite[Theorem 2.6, Corollary 2.7]{MSY23}), we immediately get the following result.
\begin{cor}\label{cor: motivic multiplicativity}
    Let $\pi\colon M \ra B$ be one of the following: the compactified Prym fibration as in Theorem~\ref{thm: compactified Prym is dualizable abelian fibration s.t. FV}, an LSV fibration, or an \'etale Shafarevich--Tate twist of an LSV fibration. Then the following results hold:
    \begin{enumerate}
         \item (Corti--Hanamura conjecture) The decomposition of $R\pi_{*}\BQ_{M}$ into semisimple perverse sheaves admits a motivic lifting;
        \item (Multiplicative perverse filtration) There exists a motivic perverse filtration for $\pi$ that is multiplicative.
    \end{enumerate}
\end{cor}

\subsection{Outline of the paper}
We briefly outline the contents of this paper. In Section~\ref{sec: Preliminaries}, we recall the definition of the universal Prym stack for \'etale double covers. In Section~\ref{sec: construction}, we prove the existence part of Theorem~\ref{thm A}. For this, we construct the normalized Poincar\'e sheaf on the relative compactified Prym by descending the Poincar\'e sheaf on relative compactified Jacobian constructed by Arinkin. Then in Section~\ref{sec_properties}, we study the properties of the Poincar\'e sheaf and finish the proof of Theorem~\ref{thm A}. In Section~\ref{sec_Autoequivalence}, we prove Theorem~\ref{thm B}. The key step is to prove Arinkin's dimension bound. In Section~\ref{sec_autoduality}, we deduce Theorem~\ref{thm: C} from Theorem~\ref{thm B}. In Section~\ref{sec_application}, we apply our main results to prove Theorems~\ref{thm: compactified Prym is dualizable abelian fibration s.t. FV}, \ref{thm: LSV is dualizable abelian fibration s.t. FV}, and Corollary~\ref{cor: motivic multiplicativity}.

\subsection{Acknowledgements}
The author is deeply grateful to Liang Xiao and Qizheng Yin for their invaluable guidance and helpful discussions, and especially to Qizheng Yin for suggesting this project. 
The author would also like to thank Hanfei Guo, Jia Choon Lee, Ziwei Lu, Lu Qi, Junliang Shen, Shuting Shen, and Feinuo Zhang for helpful conversations. 

\section{Preliminaries}\label{sec: Preliminaries}
Throughout this article, we work over an algebraically closed field $k$ of characteristic zero. 
Let $g$ be the arithmetic genus of the fiber curves $\CC \to B$. 
When $g=1$, our assumption that the geometric fibers of $\widetilde{\CC} \to B$ are integral curves forces the fibers of $\CC \to B$ to be smooth. Hence, the associated relative compactified Prym variety equals the relative Prym variety, which is a zero-dimensional abelian scheme over $B$. Consequently, Theorems~\ref{thm A}, \ref{thm B}, and \ref{thm: C} trivially hold. Therefore, for the rest of this article we assume that $g > 1$. 

\subsection{Universal Prym stack for \'etale double covers}\label{subsec: universal Prym}
To prove Theorems~\ref{thm A}, \ref{thm B}, and \ref{thm: C} for arbitrary families, we use the universal family of \'etale double covers and work with stacks. We now briefly define the universal Prym stack and summarize some of its properties. 

We denote by $\EM_g^{\inte,\pl}$ the moduli stack of integral projective curves $C$ of genus $g$ with planar singularities. For any scheme $B$, the objects of the groupoid $\EM_g^{\inte,\pl}(B)$ are morphisms 
\[f: X \to B,\] 
where $X$ is an \emph{algebraic space} and $f$ is a flat, proper morphism of finite presentation whose fibers are integral curves of genus $g$ with planar singularities. By \cite[Section~3]{Arinkin10}, $\EM_g^{\inte,\pl}$ is a smooth algebraic stack of finite type, of dimension $3g-3$. 

Let $\ol{\EJ}_{g}$ be the universal compactified Jacobian stack over $\EM_g^{\inte,\pl}$ defined as in \cite[Definition 7]{Arinkin10}. Then it is a smooth algebraic stack, and the natural forgetful map $\ol{\EJ}_{g} \to \EM_g^{\inte,\pl}$ is representable by algebraic spaces (\cite[Propositions~3.1, 3.3]{Melo19}). 

Since $g > 1$, for any $f \colon X \to B$ as above, the relative dualizing sheaf $\omega_{X/B}$ is relatively ample. Therefore, $X$ is a scheme and $f$ is locally projective (\cite[Tag~0D32]{stacks-project}). According to \cite[Theorem~8.5]{AK80}, the map $\ol{\EJ}_{g} \to \EM_g^{\inte,\pl}$ is representable by schemes.

\begin{defn}
    Let $\ER_g$ be the prestack over $\mathrm{Sch}_k$ whose objects over a scheme $B$ are diagrams 
    \[\widetilde{C}\ra C \ra B,\]
    where $C \ra B$ is an object in $\EM_g^{\inte,\pl}(B)$, the morphism $\widetilde{C} \ra C$ is an \'etale double cover, and the geometric fibers of $\widetilde{C} \ra B$ are integral. Morphisms from $(\widetilde{C}'\ra C' \ra B')$ to $(\widetilde{C}\ra C \ra B)$
    are given by Cartesian diagrams 
    \begin{equation*}
            \begin{tikzcd}
            \widetilde{C}' \arrow[r, "\widetilde{\alpha}"] \arrow[d] 
                \arrow[dr, phantom, "\square"]
            & \widetilde{C} \arrow[d]  \\
            C' \arrow[r, "\alpha"]  \arrow[d] \arrow[dr, phantom, "\square"]                 
            & C \arrow[d]\\    
            B' \arrow[r, "f"]  
            & B.          
            \end{tikzcd}
    \end{equation*}
\end{defn}   

Consider the natural forgetful map 
\[F \colon \ER_g \ra \EM_g^{\inte,\pl}, \quad (\widetilde{C},C) \mapsto C.\]

\begin{prop}
    The morphism $F$ is representable by algebraic spaces, \'etale, and locally of finite type. Consequently, $\ER_g$ is a smooth algebraic stack, locally of finite type. 
\end{prop}

\begin{proof}
    Let $B$ be any scheme and $p\colon B \ra \EM_g^{\inte,\pl}$ be any morphism.
    Let $\CC_B \ra B$ be the family of curves corresponding to $p$. The relative Picard functor $\Pic_{(\CC_B/B)}^0$ is an algebraic space, locally of finite type over $B$ (\cite[Tags~0D2C, 0DMB]{stacks-project}). 
    Let $\Pic_{(\CC_B/B)}^0[2]$ be the 2-torsion subgroup of $\Pic_{(\CC_B/B)}^0$. Since $k$ has characteristic zero, the multiplication-by-2 map on $\Pic_{(\CC_B/B)}^0$ is \'etale. Thus, the morphism $\Pic_{(\CC_B/B)}^0[2] \ra B$ is \'etale and locally of finite type. 
    
    Since geometric irreducibility of the fibers is an open condition on the base, the fiber product $\ER_g \times_{\EM_g^{\inte,\pl}} B$ is represented by an open subspace of $\Pic_{(\CC_B/B)}^0[2]$. Therefore, $F$ is representable by algebraic spaces, \'etale, and locally of finite type. Because $\EM_g^{\inte,\pl}$ is a smooth algebraic stack, it follows from \cite[Tag~05UN]{stacks-project} that $\ER_g$ is also a smooth algebraic stack, locally of finite type. 
\end{proof}

Denote by $\pi \colon \widetilde{\EC} \ra \EC$ the universal \'etale double cover over $\ER_g$. The family $\widetilde{\EC} \to \ER_g$ gives a morphism $\ER_g \to \EM_{2g-1}^{\inte,\pl}$. We denote by $\ol{\EJ} = \ER_g \times_{\EM_{2g-1}^{\inte,\pl}} \ol{\EJ}_{2g-1}$ the relative compactified Jacobian of $\widetilde{\EC}$ over $\ER_g$. Let $\iota: \widetilde{\EC} \ra \widetilde{\EC}$ be the involution associated to the \'etale double cover $\pi: \widetilde{\EC} \ra \EC$. This induces an involution map 
\begin{equation*}
    \begin{aligned}
        \tau \colon \ol{\EJ} & \ra \ol{\EJ}\\
        \CF & \longmapsto \iota^{*}\CF^{\vee},
    \end{aligned}
\end{equation*}
where $\CF^{\vee}=\hom_{\CO_{\widetilde{\EC}}}(\CF,\CO_{\widetilde{\EC}})$ denotes the dual sheaf of $\CF$. The \emph{relative compactified Prym stack} $\ol{\EP}$ is defined to be the irreducible component of the fixed locus of $\tau$ that contains the zero section:
\begin{equation*}
    \ol{\EP} \coloneqq \Fix(\tau)_0.
\end{equation*}
This makes $\ol{\EP}$ a closed substack of $\ol{\EJ}$. Thus $\ol{\EP}$ is an algebraic stack and the natural morphism 
\[\ol{p} \colon \ol{\EP} \ra \ER_g\]
is representable by schemes. 

\begin{prop}\label{prop: compact Prym is sm}
    The main properties of $\ol{\EP}$ are summarized as follows.
    \begin{enumerate}
        \item The morphism $\ol{p} \colon \ol{\EP} \ra \ER_g$ is a locally projective morphism with integral fibers of dimension $g-1$.
        \item The stack $\ol{\EP}$ is smooth. Consequently, the morphism $\ol{p}$ is a flat, local complete intersection morphism.        
    \end{enumerate}
\end{prop}

\begin{proof}
    Property (1) follows from the corresponding properties of relative compactified Prym fibrations (see \cite[Remark~4.5, Proposition~4.10]{LSV17}). 

    To prove (2), it suffices to check the assertion after base changing to a smooth presentation. Let $U \to \ER_g$ be a smooth presentation. Since $\ER_g$ is a smooth algebraic stack, $U$ is a smooth scheme. Let $\pi_U\colon \widetilde{\EC}_U \to \EC_U$ be the base change of $\pi$ to $U$, which is an \'etale double cover between schemes. Consider the 2-fiber product $\ol{\Prym}(\widetilde{\EC}_U/\EC_U) = U \times_{\ER_g} \ol{\EP}$. The natural projection induces a smooth presentation 
    \[\ol{\Prym}(\widetilde{\EC}_U/\EC_U) \to \ol{\EP}.\]
    Since the composition $U \to \ER_g \to \EM_g^{\inte,\pl}$ is smooth, the family of curves $\EC_U \to U$ is a locally versal family. Then, according to \cite[Theorem~4.20]{LSV17}, the scheme $\ol{\Prym}(\widetilde{\EC}_U/\EC_U)$ is smooth. Consequently, the morphism $\ol{p}_U \colon \ol{\Prym}(\widetilde{\EC}_U/\EC_U) \to U$ is flat by miracle flatness, and is a local complete intersection morphism by \cite[Tag~069M]{stacks-project}. 
\end{proof}

\subsection{A technical proposition}
\begin{notn}
    Let $f \colon X \ra B$ be a flat projective morphism whose geometric fibers are curves. Let $\CE$ be a $B$-flat coherent sheaf on $X$. Then the \emph{determinant of cohomology of $\CE$ with respect to $f$} is defined to be 
    \[\CD_f(\CE) \coloneq \det Rf_{*}(\CE).\]
\end{notn}

The determinant of cohomology enjoys good properties that will be frequently used in the calculations, such as functoriality, compatibility with base change, the projection formula and additivity with respect to short exact sequences; see \cite[Proposition 44]{Esteves01} for details. 

We now prove a technical proposition, which generalizes \cite[Proposition 4.4]{GS22} and will be used multiple times in Section~\ref{sec: construction}. 

\begin{prop}\label{prop: important calculation}
    Let $X$ and $Y$ be finite-type integral $B$-schemes with $X$ quasi-projective. 
    Consider a family of \'etale double covers of projective integral curves $\pi \colon \widetilde{\CC} \to \CC$ over the base $B$. Let~$\varpi \colon \widetilde{\CC} \times_B X\times_B Y \ra \CC \times_B X\times_B Y$ be the base change of $\pi$. Let $p_{ij}$ and $q_{ij}$ be the projections from $\widetilde{\CC} \times_B X \times_B Y$ and $\CC \times_B X \times_B Y$ to their $(i,j)$-th factors, respectively. 
    Then for line bundles $\CM$ and $\CN$ on $\widetilde{\CC} \times_B X \times_B Y$, and a line bundle $K$ on $\CC\times_B X$, we have an isomorphism:
    \begin{equation}\label{eq: important calculation}
        \begin{aligned}
            \CD_{\p23}(\CM \ot \varpi^{*}q_{12}^{*}K) \ot \CD_{\p23}(\CM)^{-1} 
            & \ot \CD_{\p23}(\CN \ot \varpi^{*}q_{12}^{*}K)^{-1} \ot \CD_{\p23}(\CN) \\
            \cong{} \CD_{\q23}(\det \varpi_{*}\CM \ot q_{12}^{*}K) \ot \CD_{\q23} (\det \varpi_{*}\CM)^{-1} 
            & \ot \CD_{\q23}(\det \varpi_{*}\CN \ot q_{12}^{*}K)^{-1} \ot \CD_{\q23} (\det \varpi_{*}\CN).
        \end{aligned}
    \end{equation}
\end{prop}

\begin{proof}
    We first assume that $K \cong \CO(D-E)$ with $D$ and $E$ effective Cartier divisors on $\CC\times_B X$ that are finite flat over $X$. By the projection formula, we have an isomorphism
    \begin{equation}\label{sec1.1 eq 0}
        \CD_{\p23}(\CM \ot \varpi^{*}q_{12}^{*}K) \ot \CD_{\p23}(\CM)^{-1} 
        \cong \CD_{q_{23}}(\varpi_{*}\CM \ot q_{12}^{*}K) \ot \CD_{q_{23}}(\varpi_{*}\CM)^{-1}.
    \end{equation}
    We will compare this with $\CD_{\q23}(\det \varpi_{*}\CM \ot q_{12}^{*}K) \ot \CD_{\q23} (\det \varpi_{*}\CM)^{-1}$. 
    
    Consider the following exact sequences
    \begin{equation}\label{sec1.1 eq 1}
        0 \ra \varpi_{*}\CM (-E_Y) \ra \varpi_{*}\CM \ra (\varpi_{*}\CM)|_{E_Y} \ra 0,
    \end{equation}
    \begin{equation}\label{sec1.1 eq 2}
        0 \ra \varpi_{*}\CM(-E_Y) \ra \varpi_{*}\CM(D_Y-E_Y) \ra (\varpi_{*}\CM(D_Y-E_Y))|_{D_Y} \ra 0,
    \end{equation}
    where $D_Y$ and $E_Y$ denote the pullbacks of $D$ and $E$ to $\CC \times_B X \times_B Y$, respectively.
    Since the determinant of cohomology is additive (\cite[Proposition 44]{Esteves01}), we have
    \begin{equation}\label{sec1.1 eq 3}
        \CD_{q_{23}}(\varpi_{*}\CM) \cong \CD_{q_{23}}(\varpi_{*}\CM(-E_Y)) \ot \CD_{q_{23}}((\varpi_{*}\CM)|_{E_Y}),
    \end{equation}
    \begin{equation}\label{sec1.1 eq 4}
        \CD_{q_{23}}(\varpi_{*}\CM \ot q_{12}^{*}K) \cong \CD_{q_{23}}(\varpi_{*}\CM(-E_Y)) \ot \CD_{q_{23}}((\varpi_{*}\CM(D_Y-E_Y))|_{D_Y}).
    \end{equation}
    Combining \eqref{sec1.1 eq 3} and \eqref{sec1.1 eq 4}, we obtain
    \begin{equation}\label{sec1.1 eq 5}
        \CD_{q_{23}}(\varpi_{*}\CM \ot q_{12}^{*}K) \ot \CD_{q_{23}}(\varpi_{*}\CM)^{-1} \cong \CD_{q_{23}}((\varpi_{*}\CM(D_Y-E_Y))|_{D_Y}) \ot \CD_{q_{23}}((\varpi_{*}\CM)|_{E_Y})^{-1}.
    \end{equation}

    By considering the exact sequences \eqref{sec1.1 eq 1} and \eqref{sec1.1 eq 2} with all $\varpi_{*}\CM$ replaced by $\det(\varpi_{*}\CM)$, we get 
    \begin{equation}\label{sec1.1 eq 6}
    \begin{aligned}
        \CD_{q_{23}}(\det(\varpi_{*}\CM) \ot q_{12}^{*}K) \ot \CD_{q_{23}}(\det \varpi_{*}\CM)^{-1} \cong{}& \CD_{q_{23}}(\det(\varpi_{*}\CM)(D_Y-E_Y)|_{D_Y})\\
        &\ot \CD_{q_{23}}(\det(\varpi_{*}\CM)|_{E_Y})^{-1}.
    \end{aligned}
    \end{equation}

    Let $i \colon D\times_B Y \hookrightarrow \CC \times_B X \times_B Y $ be the closed embedding and $s_{23}= q_{23} \circ i$. By assumption $s_{23}$ is finite and flat. Since $\varpi$ is an \'etale double cover, the sheaf $\varpi_{*}\CM$ is locally free of rank~$2$. 
    Therefore we have
    \begin{align*}
        \CD_{q_{23}}((\varpi_{*}\CM(D_Y-E_Y)) |_{D_Y}) 
        & \cong \det Rq_{23,*}i_{*}i^{*}(\varpi_{*}\CM \ot q_{12}^{*}K)\\
        & \cong \det {s_{23}}_{*}(i^{*}\varpi^{*}\CM \ot i^{*}q_{12}^{*}K)\\
        & \cong \det {s_{23}}_{*}(i^{*}\det\varpi_{*}\CM \ot i^{*}q_{12}^{*}K) \ot \det {s_{23}}_{*}(i^{*}q_{12}^{*}K)\\
        & \cong \CD_{q_{23}}((\det \varpi_{*}\CM(D_Y-E_Y))|_{D_Y}) \ot \CD_{q_{23}}(q_{12}^{*}K|_{D_Y}).
    \end{align*}
    Here the third isomorphism follows from the Grothendieck--Riemann--Roch theorem for the finite flat morphism $s_{23}$. A similar argument shows that 
    \begin{equation*}
        \CD_{q_{23}}((\varpi_{*}\CM) |_{E_Y}) 
         \cong \CD_{q_{23}}((\det \varpi_{*}\CM)|_{E_Y}) \ot \CD_{q_{23}}(\CO_{\CC\times_B X \times_B Y}|_{E_Y}).
    \end{equation*}
    Combining these two calculations with isomorphisms \eqref{sec1.1 eq 0}, \eqref{sec1.1 eq 5}, and \eqref{sec1.1 eq 6}, we get
    \begin{align*}
        \CD_{\p23}(\CM \ot \varpi^{*}q_{12}^{*}K) \ot \CD_{\p23}(\CM)^{-1} \cong{}&
        \CD_{\q23}(\det \varpi_{*}\CM \ot q_{12}^{*}K) \ot \CD_{\q23} (\det \varpi_{*}\CM)^{-1}\\ 
         & \ot \CD_{q_{23}}(q_{12}^{*}K|_{D_Y}) 
        \ot \CD_{q_{23}}(\CO_{\CC\times_B X \times_B Y}|_{E_Y})^{-1}.
    \end{align*}
    Applying the same argument and replacing $\CM$ by $\CN$, we obtain the desired isomorphism \eqref{eq: important calculation}.

    Now it remains to prove that $K$ can be written as the difference of two effective Cartier divisors that are finite flat over $X$. Indeed, this follows from the following lemma.
\end{proof}

\begin{lem}\label{lem: line bundle as difference of flat divisors}
    Let $S$ be a quasi-projective integral scheme over $k$, and let $f \colon C_S \ra S$ be a flat projective family of integral curves with arithmetic genus $g$. Then every line bundle $\CV$ on $\CC_S$ can be written as $\CV \cong \CO_{\CC_S}(D-E)$, where $D$ and $E$ are effective Cartier divisors on $C_S$ that are finite flat over $S$. 
\end{lem}
\begin{proof}
    Since $C_S$ is quasi-projective over $k$, we may choose an ample line bundle $\CN$ on $C_S$. 
    By \cite[Tag 01VS]{stacks-project}, there exists $d_0 \geq 1$ such that for all $d \geq d_0$, there is an immersion~$i \colon C_S \hookrightarrow \BP_S^n$ over $S$ for some $n\geq 0$ such that $\CN^{\ot d}\cong i^{*}\CO_{\BP_S^n}(1)$. 
    Since $S$ is also quasi-projective, the Segre embedding gives us an immersion $C_S \hookrightarrow \mathbb{P}^N_k$, where $N$ depends on $d$. 

    Since S is of finite type, a Bertini-type dimension counting shows that there exists an integer $m_0$ depending on $N$ such that for every $m\geq m_0$ a general hypersurface $H$ of degree $m$ in $\BP_k^N$ intersects each fiber of $C_S \ra S$ properly and the intersection lies in the regular locus of $C_S$. The intersection number of $H$ with each fiber is a constant. Thus $C_S \cdot H$ defines a Cartier divisor on $C_S$ that is finite flat over $S$. This divisor corresponds to the line bundle 
    \begin{equation*}
        \CN^{\ot dm} \ot f^{*}\CO_S(m).
    \end{equation*}
    Replacing $m$ by $m+1$ in the above equation and taking difference, we see that the conclusion holds for
    \[\CN^{\ot d} \ot f^{*}\CO_S(1).\]
    Then replacing $d$ by $d+1$ and taking difference, we see that the conclusion holds for $\CN$.

    Now take an integer $s$ such that $\CV \ot \CN^{\ot s}$ is still ample on $C_S$. The same argument shows that the conclusion holds for $\CV \ot \CN^{\ot s}$, and thus $\CV$.     
\end{proof}

\section{Poincar\'e sheaf on the relative compactified Prym variety}\label{sec: construction}

\subsection{Notation and setup}\label{subsec1.1: notations}
As in the case of compactified Jacobians (see \cite[Remark (2) after Theorem~C]{Ari10b}), to prove Theorems~\ref{thm A}, \ref{thm B}, and \ref{thm: C} for arbitrary families, it suffices to prove them for the universal family. 

For the sake of exposition and to avoid the cumbersome notation associated with algebraic stacks, we will carry out the proofs in the setting of schemes. However, as explained in Remark~\ref{rmk: universal Prym}, the entire argument can be carried out over the universal family using stack-theoretic language. 

Throughout the rest of this article, we work in the following setting. Let $B$ be an irreducible~$k$-scheme. Let $\CC \ra B$ be a flat and projective morphism whose geometric fibers are integral curves of arithmetic genus $g > 1$ with planar singularities. Let $\pi: \widetilde{\CC} \ra \CC$ be an \'etale double cover such that the geometric fibers of $\widetilde{\CC} \ra B$ are integral curves of arithmetic genus~$2g-1$. 

We denote by $\ol{J}_{\CC}$ (resp.~$\ol{J}_{\widetilde{\CC}}$) the relative compactified Jacobian $\ol{\Jac}(\CC/B)$ (resp.~$\ol{\Jac}(\widetilde{\CC}/B)$) that parameterizes relative torsion-free, rank 1 degree 0 sheaves on $\CC$ (resp.~on $\widetilde{\CC}$). They are projective $B$-schemes, flat over $B$, and the geometric fibers are integral and local complete intersection schemes; see \cite[Theorem 8.5]{AK80} and \cite[Theorem 9]{AIK77}. 

We assume that $\ol{J}_{\CC}$ is nonsingular. In particular, this implies that $B$ is nonsingular. According to \cite[Theorem 4.20 (3)]{LSV17}, the relative compactified Prym variety $\ol{\Prym}(\widetilde{\CC}/\CC)$ is nonsingular and $\ol{\Prym}(\widetilde{\CC}/\CC) \ra B$ is flat of relative dimension $g-1$.

Let $J_{\widetilde{\CC}} = \Jac(\widetilde{\CC}/B)$ (resp.~$J_{\CC} = \Jac(\CC/B)$ ) be the relative generalized Jacobian, which is an open dense subset of $\ol{J}_{\widetilde{\CC}}$ (resp.~$\ol{J}_{\CC}$) and forms a nonsingular commutative group scheme over $B$. We define the Norm map 
\begin{equation*}
    \begin{aligned}
        \Nm \colon \Jac(\widetilde{\CC}/B)& \ra \Jac(\CC/B)\\
        L & \longmapsto \det(\pi_{*}L)\ot \det(\pi_{*}\CO_{\widetilde{\CC}})^{-1}. 
    \end{aligned}
\end{equation*}
Then the \emph{relative generalized Prym} of $\widetilde{\CC}$ over $\CC$ is defined to be the irreducible component of $\ker(\Nm)$ containing the zero section: 
\begin{equation}\label{eq: Prym= ker Nm}
    \Prym(\widetilde{\CC}/\CC) \coloneqq \ker(\Nm)_0.
\end{equation}
This is a commutative group scheme over $B$, smooth of relative dimension $g-1$. It is easy to check that 
\begin{equation}\label{Prym=Im}
    \Prym(\widetilde{\CC}/\CC) = \Fix(-\iota^{*})_0 = \mathrm{Im} (1-\iota^{*}) \subset J_{\widetilde{\CC}}.
\end{equation}
Then $\Prym(\widetilde{\CC}/\CC)$ is an open dense subset in $\ol{\Prym}(\widetilde{\CC}/\CC)$, and in particular
\begin{equation*}
    \ol{\Prym}(\widetilde{\CC}/\CC) = \ol{\mathrm{Im} (1-\iota^{*})}.
\end{equation*}
For simplicity, we set $P=\Prym(\widetilde{\CC}/\CC)$ and $\ol{P} = \ol{\Prym}(\widetilde{\CC}/\CC)$. 

\subsection{The Severi inequality}
For a proper curve $C$ over $k$, let $\widetilde{g}(C)$ be the geometric genus of $C$. For a closed point $b\in B$, define $\delta(b)$ to be the dimension of the maximal affine subgroup of $P_b$. According to the discussion following \cite[equation (4.2)]{LSV17}, we get $\delta(b) = g - \widetilde{g}(\CC_b)$. This defines an upper semi-continuous function 
\[\delta \colon B \to \BZ_{\geq 0}, \quad b \mapsto \delta(\CC_b).
\]
For a closed subvariety $Z \subset B$, we define $\delta_Z$ to be the minimal value of the function $\delta$ on $Z$. 

\begin{lem}\label{Severi inequality}
    The compactified Prym fibration $\ol{\Prym}(\widetilde{\CC}/\CC) \ra B$ satisfies the \emph{Severi inequality}
    \begin{equation}
        \mathrm{codim}_B (Z) \geq \delta_Z
    \end{equation}
    for any irreducible subset $Z\subset B$.
\end{lem}
\begin{proof}
    Let $\delta'$ be the $\delta$-function associated to commutative group scheme $J_{\CC} \ra B$. Then we have $\delta'(b) = g -\widetilde{g}(\CC_b) = \delta(b)$ holds for every closed point $b\in B$. 
    Since $\ol{J}_{\CC}$ is nonsingular by assumption, \cite[Lemma 4.1]{MaulikShen20} shows that 
    $\mathrm{codim}_B (Z) \geq \delta'_Z = \delta_Z$. 
\end{proof}

Consider the stratification of $B$ by locally closed subsets $B^{(\widetilde{g})}$, where $B^{(\widetilde{g})}$ parameterizes points $b\in B$ for which $\widetilde{g}(\CC_b) = \widetilde{g}$. The Severi inequality then implies that 
\[\codim (B^{(\widetilde{g})}) \geq g-\widetilde{g}.\]

\subsection{The Poincar\'e line bundle on Prym varieties of \'etale double covers of smooth curves}\label{subsection: over smooth curves}
In this section, all the arguments remain valid for families of smooth curves. For simplicity, we only consider smooth curves over the field $k$. 

Let $C$ be a proper smooth curve of genus $g$ over $k$, and $\widetilde{C} \ra C$ be an \'etale double cover. We denote by $J = \Jac(\widetilde{C})$ the Jacobian variety of $\widetilde{C}$, which is an abelian variety of dimension~$2g-1$. It has a canonical principal polarization 
\[\varphi_{\Theta} \colon J \xrightarrow{\sim} \hat{J},\quad x \mapsto T_x^{*}(\CO_J(\Theta)) \ot \CO_J(-\Theta),\]
where $\Theta$ is the theta divisor of the Jacobian variety $J$, and $T_{x} \colon J \ra J$ is the translation by the point $x$. We denote by $\CP$ the pullback of the normalized universal sheaf on $J \times \hat{J}$ via the isomorphism $\id_J \times \varphi_{\Theta}$. This line bundle is known as the \emph{normalized Poincar\'e line bundle} on $J\times J$. For brevity, we refer to it simply as the Poincar\'e line bundle. Alternatively, it can be described as 
\[\CP = m^{*}\CO_J(\Theta) \ot p_1^{*} \CO_J(-\Theta) \ot p_2^{*} \CO_J(-\Theta),\] 
where $m \colon J \times J \ra J$ is the multiplication map, and $p_i$ for $i=1,2$ denotes the projection onto the $i$-th factor. 

We denote by $P = \Prym(\widetilde{C}/C)$ the Prym variety. Mumford \cite{Mumford74} proved that $P$ also has a canonical principal polarization $\varphi_{\Xi} \colon P \xrightarrow{\sim} \hat{P}$, where $\Xi$ is the theta divisor on the Prym variety~$P$. Moreover, the following relation holds:
\begin{equation}\label{eq: relation between polarization}
    \varphi_{\Theta|_P} = \varphi_{\Xi} \circ [2],
\end{equation}
where $[2]$ is the multiplication-by-2 map on $P$. Let $\CG$ be the normalized Poincar\'e line bundle on $P\times P$. In the rest of this section, we study the relationship between the sheaves $\CP$ and $\CG$. 

Let $\CR$ be the restriction of $\CP$ to $J\times P$, and let $f$ be the homomorphism $1-\iota^{*}\colon J \ra P$, which is faithfully flat.

\begin{lem}\label{lem: smooth case equivariant}
    We have an isomorphism $\CR \cong (f\times \id_{P})^{*}\CG$ of line bundles on $J \times P$.
\end{lem}

Although this result is likely known to experts, we are unaware of an explicit reference in the literature. Therefore, we include a proof here for completeness. The key ingredient of the proof is to compare the polarization using \eqref{eq: relation between polarization}.
\begin{proof}
    According to the universal property of $\CG$, there exists a unique morphism $\widetilde{f} \colon J \ra P$ such that $\CR = (\widetilde{f}\times 1)^{*}\CG$. To show that $\widetilde{f} = f$, it suffices to verify that 
    \begin{equation}\label{eq: fiberwise equal}
        \CR|_{[M] \times P} \cong \CG|_{[M\ot \iota^{*}M^{-1}]\times P}
    \end{equation}
    for every line bundle $M$ on $\widetilde{C}$. 
    
    We have
    \[\CR|_{[M] \times P} \cong \CP|_{[M] \times P} \cong  (T_{[M]}^{*}\CO_J(\Theta) \ot \CO_J(-\Theta))|_P.\]
    Take $L\in \Jac(C)$ and $N\in \Prym(\widetilde{C}/C)$ such that 
    \[M \cong \pi^{*}L \ot N,\]
    we get 
    \[\CR|_{[M] \times P} \cong T_{[N]}^{*}((T_{[\pi^{*}L]}^{*}\CO_J(\Theta))|_P) \ot \CO_J(-\Theta)|_P.\]
    A similar calculation shows that 
    \[\CG|_{[M\ot \iota^{*}M^{-1}]\times P} \cong T_{2[N]}^{*}\CO_P(\Xi) \ot \CO_P(-\Xi).\]
    
    From \eqref{eq: relation between polarization} we get 
    \[T_{[N]}^{*}(\CO_J(\Theta)|_P) \ot \CO_J(-\Theta)|_P \cong T_{2[N]}^{*}\CO_P(\Xi) \ot \CO_P(-\Xi).\]
    Thus 
    \[\CR|_{[M] \times P} \ot \CG^{-1}|_{[M\ot \iota^{*}M^{-1}]\times P} \cong T_{[N]}^{*}((T_{[\pi^{*}L]}^{*}\CO_J(\Theta))|_P) \ot T_{[N]}^{*}(\CO_J(-\Theta))|_P).\]
    Then to prove \eqref{eq: fiberwise equal}, we only need to check
    \[\CR_{[\pi^{*}L]\times P} \cong (T_{[\pi^{*}L]}^{*}\CO_J(\Theta) \ot \CO_J(-\Theta))|_P\]
    is trivial for every $L \in \Jac(C)$. 

    Notice that the Poincar\'e line bundle can also be written as 
    \begin{equation}
    \CP \cong \CD_{\p23}(p_{12} ^{*}\CU \ot p_{13}^{*}\CU) \ot 
    \CD_{\p23}(p_{12}^{*}\CU)^{-1} \ot 
    \CD_{\p23}(p_{13}^{*}\CU)^{-1} \ot 
    \CD_{\p23}(\CO_{\widetilde{C}\times J_{\widetilde{C}} \times J_{\widetilde{C}}}),
    \end{equation}
    where $p_{ij}$ is the projection onto the $(i,j)$-th factors of $\widetilde{C} \times {J}_{\widetilde{C}} \times {J}_{\widetilde{C}}$, and $\CU$ is a universal sheaf on $\widetilde{C} \times {J}_{\widetilde{C}}$ (see, e.g., \cite[Remark 2.4]{Esteves99}). Also, according to \eqref{eq: Prym= ker Nm}, 
    $\det (\pi \times \id_P)_{*}\CU'$ is fiberwise isomorphic to $\det (\pi \times \id_P)_{*}\CO$ over $P$. Then the seesaw principle yields that there exists a line bundle $M$ on $P$ such that 
    \begin{equation}\label{lem1.5 eq 1}
        \det (\pi \times \id_P)_{*}\CU' \cong \det (\pi \times \id_P)_{*}\CO \ot q_2^{*}M.
    \end{equation}

    Therefore we get 
    \begin{align*}
        \CR|_{[\pi^{*}L] \times P} 
        \cong{}&  \CD_{p_2}((\pi \times \id_P)^{*}q_1^{*}L \ot \CU') \ot \CD_{p_2}(\CU')^{-1} \ot \CD_{p_2}((\pi \times \id_P)^{*}q_1^{*}L )^{=1} \ot \CD_{p_2}(\CO) \\
        \cong{}&  \CD_{q_2}(q_1^{*}L \ot \det (\pi \times \id_P)_{*}\CU') \ot \CD_{q_2}(\det (\pi \times \id_P)_{*}\CU')^{-1} \\
        & \ot \CD_{q_2}(q_1^{*}L \ot \det (\pi \times \id_P)_{*}\CO)^{-1} \ot \CD_{q_2}(\det (\pi \times \id_P)_{*}\CO) \\
        \cong{}&  \CO,
    \end{align*}
    where $\CU'$ is the restriction of $\CU$ on $\widetilde{C}\times P$, and $p_i$ (resp.~$q_i$) is the projection onto the $i$-th factor of $\widetilde{C} \times P$ (resp.~$C \times P$). The first isomorphism holds by the base change property of determinant of cohomology. The second holds by Proposition~\ref{prop: important calculation}. The third isomorphism follows from \eqref{lem1.5 eq 1} and the projection formula of the determinant of cohomology \cite[Proposition 44 (3)]{Esteves01}. 
\end{proof}

\subsection{Construction of the Poincar\'e line bundle on $P\times_B {P}$ under assumptions}
In this section, we assume that the families of curves $\widetilde{\CC} \ra B$ and $\CC \ra B$ satisfy the following two assumptions: 
\begin{enumerate}
    \item [(A1)] $\widetilde{\CC} \ra B$ admits a section whose image lies in the smooth locus of the fibers;
    \item [(A2)] $B$ is quasi-projective. 
\end{enumerate}
Assumption (A1) implies that $\CC \ra B$ also admits a section that lies in the smooth locus. It guarantees the existence of universal sheaves over relative Jacobians. Assumption (A2) is needed in the calculation; see Proposition~\ref{prop: important calculation}. We will explain how to remove these assumptions in Section~\ref{subsec: no section}.

We first briefly recall the construction of the Poincar\'e line bundle $\CP$ on $J_{\widetilde{\CC}}\times_B {J}_{\widetilde{\CC}}$. The existence of section allows us to consider the universal sheaf $\CU$ on $\widetilde{\CC} \times_B {J}_{\widetilde{\CC}}$ that is normalized in the sense that it is trivialized along the given section. Similarly, let $\CV$ be the normalized universal sheaf on $\CC \times_B {J}_{\CC}$. Consider the projections 
\begin{equation*}
    \begin{tikzcd}
        & \widetilde{\CC} \times_B J_{\widetilde{\CC}} \times_B J_{\widetilde{\CC}}  \arrow[dl,"p_{12}"'] \arrow[d,"p_{13}"] \arrow[dr,"p_{23}"]&  \\
        \widetilde{\CC} \times_B J_{\widetilde{\CC}} & \widetilde{\CC} \times_B J_{\widetilde{\CC}} & J_{\widetilde{\CC}} \times_B J_{\widetilde{\CC}}.
    \end{tikzcd}
\end{equation*}
The Poincar\'e line bundle on $J_{\tilC} \times_B {J}_{\tilC}$ is defined to be
\begin{equation}\label{eq: CP}
    \CP = \CD_{\p23}(p_{12} ^{*}\CU \ot p_{13}^{*}\CU) \ot 
    \CD_{\p23}(p_{12}^{*}\CU)^{-1} \ot 
    \CD_{\p23}(p_{13}^{*}\CU)^{-1} \ot 
    \CD_{\p23}(\CO_{\widetilde{\CC} \times_B {J}_{\widetilde{\CC}} \times_B {J}_{\widetilde{\CC}}}).
\end{equation}
It is easy to check that $\CP$ is normalized, being trivial along the zero sections of both factors; over the locus of smooth curves $B^{(g)}$, it restricts to the Poincar\'e line bundle defined in Section~\ref{subsection: over smooth curves} (see \cite[Remark 2.4]{Esteves99}). 

\begin{rmk}
    Since we are working on degree 0 compactified Jacobians, the construction of $\CP$ is independent of the choices of the universal sheaves and the sections of $\widetilde{\CC} \ra B$; see \cite[Proposition 3.1]{AddingtonDonovanMeachan16} or \cite[Section 3.6]{MSY23}. 
\end{rmk}

Let $\CR$ be the restriction of $\CP$ to ${J_{\widetilde{\CC}}\times_B P}$. Then by the base-change property of the determinant of cohomology (\cite[Proposition 44]{Esteves01}) we get 
\begin{equation}\label{eq: CR}
    \CR \cong \CD_{\p23}(p_{12} ^{*}\CU \ot p_{13}^{*}\CU') \ot 
    \CD_{\p23}(p_{12}^{*}\CU)^{-1} \ot 
    \CD_{\p23}(p_{13}^{*}\CU')^{-1} \ot 
    \CD_{\p23}(\CO_{\widetilde{\CC} \times_B {J}_{\widetilde{\CC}} \times_B P}), 
\end{equation}
where $p_{ij}$ is the projection onto the $(i,j)$-th factors of $\widetilde{\CC} \times_B {J}_{\widetilde{\CC}} \times_B P$, and $\CU'$ is the restriction of $\CU$ to $\widetilde{\CC}\times_B P$. 

We denote by $f$ the morphism
\begin{equation}\label{eq: def of f}
    f \coloneq 1-\iota^{*}: J_{\tilC} \ra P,
\end{equation}
which is obviously faithfully flat. Consider the morphism of $B$-group schemes induced by the \'etale double cover $\pi \colon \widetilde{\CC} \ra \CC$:  
\[\pi^{*} \colon J_{\CC} \ra J_{\widetilde{\CC}}, \quad F \mapsto \pi^{*}F. \]
Let $\Lambda \coloneqq \ker \pi^{*}$ be the subgroup scheme of $J_{\CC}$. It is generated by the 2-torsion line bundle 
\[\eta \coloneqq \det \pi_{*}\CO_{\widetilde{\CC}}.\]
Then the image $\pi^{*}J_{\CC}$ is a subgroup of $J_{\widetilde{\CC}}$ isomorphic to the quotient group scheme $J_{\CC} \big/ \Lambda$. 

\begin{lem}\label{lem: descent=equivariant}
    The group scheme $J_{\widetilde{\CC}}$ is a $\pi^{*}J_{\CC}$-torsor over $P$ in the fppf topology. 
\end{lem}
\begin{proof}
    The morphism $f \colon J_{\widetilde{\CC}} \ra P$ has scheme-theoretic kernel $\pi^{*}J_{\CC}$; see \cite[Lemma 2.8]{Sacca13b}. Since~$f$ is faithfully flat, it follows that 
    \[1 \ra \pi^{*}J_{\CC} \ra J_{\widetilde{\CC}} \ra P \ra 1\]
    is an exact sequence in $\mathrm{ShAb_{B}}$, the category of sheaves of abelian groups on $(B)_{\mathrm{fppf}}$. 
    The claim now follows from \cite[Proposition 4.31]{BenMooen12}.
\end{proof}

Consider the group action 
\[\sigma \colon \pi^{*}J_{\CC} \times_B J_{\widetilde{\CC}} \ra J_{\widetilde{\CC}}, \quad (\pi^{*}F,L) \mapsto \pi^{*}F \ot L,\]
which lifts the action of $\pi^{*}J_{\CC}$ on $J_{\widetilde{\CC}} \times_B P$ via $\sigma \times \id_{P}$. We now show that $\CR$ is equivariant with respect to this action. 

\begin{lem}\label{lem: equivariant}
    The sheaf $\CR$ is $\pi^{*}J_{\CC}$-equivariant. More precisely, there is an isomorphism 
    \[\alpha \colon (\sigma \times \id_{P})^{*}\CR  \xrightarrow{\sim} p_{23}^{*}\CR\]
    on $\pi^{*}J_{\CC} \times_B J_{\widetilde{\CC}}\times_B P$
    satisfying the cocycle condition 
    \begin{equation}\label{eq: cocycle condition}
        (m \times \id_{J_{\widetilde{\CC}}} \times \id_{P})^{*} \alpha = p_{234}^{*}\alpha \circ (\id_{\pi^{*}J_{\CC}}\times \sigma \times \id_{P})^{*}\alpha
    \end{equation}
    on $\pi^{*}J_{\CC} \times_B \pi^{*}J_{\CC} \times_B J_{\widetilde{\CC}}\times_B P$, where $m$ denotes the group multiplication on $\pi^{*}J_{\CC}$.
\end{lem}

\begin{proof}
    Consider the commutative diagram 
    \begin{equation*}
        \begin{tikzcd}
            J_{\CC} \times_B J_{\widetilde{\CC}}\times_B P 
            \arrow[rd, shift left, "\widetilde{p_{23}}"]
            \arrow[rd, shift right, "\widetilde{\sigma}\times \id_P"']
            \arrow[rr, "\pi^{*} \times \id_{J_{\widetilde{\CC}}\times P}"] 
            & & \pi^{*}J_{\CC} \times_B J_{\widetilde{\CC}}\times_B P
            \arrow[dl, shift left, "p_{23}"] \arrow[dl, shift right, "\sigma\times \id_{P}"']\\
             &J_{\widetilde{\CC}} \times_B P.&
        \end{tikzcd}
    \end{equation*} 
    We will first construct the descent data $\widetilde{\alpha}$ on $J_{\CC} \times_B J_{\widetilde{\CC}}\times_B P$, then show that it descends to a descent data $\alpha$ on $\pi^{*}J_{\CC} \times_B J_{\widetilde{\CC}}\times_B P$.
    
    \medskip
    \noindent {\bf Step 1.}
    We first prove that the sheaves $\widetilde{\p23}^{*}\CR$ and $(\widetilde{\sigma}\times \id_{P})^{*}\CR$ are isomorphic.     
    Using the base-change property \cite[Proposition 44 (1)]{Esteves01} and the projection formula of the determinant of cohomology, we obtain isomorphisms
    \begin{equation*}
        \widetilde{\p23}^{*}\CR 
        \cong \CD_{p_{234}}(p_{13}^{*}\CU \ot p_{14}^{*}\CU') \ot
        \CD_{p_{234}}(p_{13}^{*}\CU)^{-1} \ot \CD_{p_{234}}(p_{14}^{*}\CU')^{-1} 
        \ot \CD_{p_{234}}(\CO_{\widetilde{\CC} \times_B J_{\CC} \times_B J_{\widetilde{\CC}}\times_B P})
    \end{equation*}
    and 
    \begin{align*}
        (\widetilde{\sigma}\times \id_P)^{*}\CR 
        & \cong \CD_{p_{234}}(\varpi^{*} q_{12}^{*}\CV \ot p_{13}^{*}\CU \ot p_{14}^{*}\CU') 
        \ot \CD_{p_{234}}(\varpi^{*} q_{12}^{*}\CV \ot p_{13}^{*}\CU)^{-1} \\
        & \quad \quad \ot \CD_{p_{234}}(p_{14}^{*}\CU')^{-1} 
        \ot \CD_{p_{234}}(\CO_{\widetilde{\CC} \times_B J_{\CC} \times_B J_{\widetilde{\CC}}\times_B P}),
    \end{align*}
    where $\CV$ is the universal sheaf on $\CC\times_B J_{\CC}$, $\varpi = \pi \times \id_{J_{\CC} \times_B J_{\widetilde{\CC}} \times_B P}$, and $p_{\bullet}$ (resp.~$q_{\bullet}$) are projections from $\widetilde{\CC} \times_B J_{\CC} \times_B J_{\widetilde{\CC}} \times_B P$ (resp.~$\CC \times_B J_{\CC} \times_B J_{\widetilde{\CC}} \times_B P$) to corresponding factors. Together with Proposition~\ref{prop: important calculation}, 
    these two isomorphisms yield
    \begin{align*}
        &\quad  (\widetilde{\sigma}\times \id_P)^{*}\CR \ot \widetilde{\p23}^{*}\CR^{-1} \\
        &\cong \CD_{q_{234}}(q_{12}^{*}\CV \ot \det \varpi_{*}(p_{13}^{*}\CU \ot p_{14}^{*}\CU')) 
         \ot \CD_{q_{234}}(\det \varpi_{*}(p_{13}^{*}\CU \ot p_{14}^{*}\CU'))^{-1} \\
        & \quad \quad \ot \CD_{q_{234}}(q_{12}^{*}\CV \ot \det \varpi_{*}p_{13}^{*}\CU)^{-1} 
        \ot \CD_{q_{234}}(\det \varpi_{*}p_{13}^{*}\CU).
    \end{align*}
   
    Notice that \eqref{eq: Prym= ker Nm} together with \cite[Proposition 3.10]{HP12} implies that there exists a line bundle $M$ on $J_{\CC} \times_B J_{\widetilde{\CC}} \times_B P$ such that
    \[\det \varpi_{*}(p_{13}^{*}\CU \ot p_{14}^{*}\CU') \cong \det \varpi_{*}p_{13}^{*}\CU \ot q_{234}^{*}M.\]
    Using the projection formula of the determinant of cohomology, we obtain
    \[(\widetilde{\sigma}\times \id_P)^{*}\CR \ot \widetilde{\p23}^{*}\CR^{-1} \cong \CO_{J_{\CC} \times_B J_{\widetilde{\CC}}\times_B P}.\]

    \medskip
    \noindent {\bf Step 2.}
    Let $\widetilde{m}$ be the group multiplication on $J_{\CC}$. We now prove that there exists a unique element $\widetilde{\alpha} \in \mathrm{Isom}((\widetilde{\sigma}\times \id_P)^{*}\CR, \widetilde{\p23}^{*}\CR)$ that satisfies the cocycle condition 
    \begin{equation}\label{eq: cocycle condition for J_C}
        (\widetilde{m} \times \id_{J_{\widetilde{\CC}}} \times \id_{P})^{*} \widetilde{\alpha} = \widetilde{p}_{234}^{*}\widetilde{\alpha} \circ (\id_{J_{\CC}}\times \widetilde{\sigma} \times \id_{P})^{*}\widetilde{\alpha}.
    \end{equation}

    For simplicity, we set $F = (\id_{J_{\CC}} \times \widetilde{\sigma} \times \id_P)^{*} \widetilde{p_{23}}^{*}\CR$ and $G = (\widetilde{m}\times \id_{J_{\widetilde{\CC}}} \times \id_P)^{*}\widetilde{p_{23}}^{*}\CR$. Denote
    \begin{equation*}
        a(\widetilde{\alpha}) \coloneq \widetilde{p}_{234}^{*}\widetilde{\alpha} \circ (\id_{J_{\CC}}\times \widetilde{\sigma} \times \id_{P})^{*}\widetilde{\alpha} \circ [(\widetilde{m} \times \id_{J_{\widetilde{\CC}}} \times \id_{P})^{*} \widetilde{\alpha}]^{-1} \in \Aut(G).
    \end{equation*}
    For a fixed isomorphism $\widetilde{\alpha}' \in \mathrm{Isom}((\widetilde{\sigma}\times \id_P)^{*}\CR, \widetilde{\p23}^{*}\CR)$, we have a bijection 
    \[\Aut(\widetilde{p_{23}}^{*}\CR) \ra \mathrm{Isom}((\widetilde{\sigma}\times \id_P)^{*}\CR, \widetilde{\p23}^{*}\CR), \quad \beta \mapsto \beta \circ \widetilde{\alpha}'.\]
    Therefore, it suffices to show that there exists a unique $\beta \in \Aut(\widetilde{\p23}^{*}\CR)$ such that $a(\beta\circ \widetilde{\alpha}') = \id$. 
    
    Since $G$ is a line bundle, $\Aut(G)$ is a commutative group. Then a simple calculation yields that 
    \begin{equation}\label{eq: translation by beta}
        a(\beta\circ\widetilde{\alpha}') = a(\widetilde{\alpha}') \circ [\widetilde{p}_{234}^{*}\widetilde{\alpha}' \circ  (\id_{J_{\CC}} \times \widetilde{\sigma} \times \id_P)^{*}\beta  \circ (\widetilde{p}_{234}^{*}\widetilde{\alpha}')^{-1}] \circ [\widetilde{p_{234}}^{*}\beta \circ 
    (\widetilde{m}\times \id_{J_{\widetilde{\CC}}} \times \id_P)^{*}\beta^{-1}].
    \end{equation}   
    It remains to check 
    \begin{equation}\label{eq: pullback 1}
        (\widetilde{m}\times \id_{J_{\widetilde{\CC}}} \times \id_P)^{*}\beta = \widetilde{p_{234}}^{*}\beta, \quad \forall \beta \in \Aut(\widetilde{p_{23}}^{*}\CR),
    \end{equation}
    and 
    \begin{equation}\label{eq: pullback 2}
        \widetilde{p}_{234}^{*}\widetilde{\alpha}' \circ ((\id_{J_{\CC}} \times \widetilde{\sigma} \times \id_P)^{*}-)\circ (\widetilde{p}_{234}^{*}\widetilde{\alpha}')^{-1}  \colon \Aut(\widetilde{p_{23}}^{*}\CR) \ra \Aut(G)
    \end{equation}
    is an isomorphism.

    We set $X = J_{\CC} \times_B J_{\widetilde{\CC}} \times_B P$ and $\ol{X} = \ol{J}_{\CC} \times_B \ol{J}_{\widetilde{\CC}} \times_B \ol{P}$. 
    Let $p\colon X \ra B$ and $\ol{p} \colon \ol{X} \ra B$ be projection maps, then $\ol{p}$ is proper and flat with geometrically integral fibers. 
    We can regard an automorphism of $\widetilde{p_{23}}^{*}\CR$ as an element in $H^0(B, \CO_{B}^{\times})$ via the canonical group isomorphism
    \[H^0(B,\CO_B^{\times}) \cong H^0(\ol{X},\CO_{\ol{X}}^{\times}) \cong H^0(X, \CO_X^{\times}) \cong
    \Aut((\widetilde{p_{23}}^{*}\CR).\]
    The first isomorphism is induced by the canonical isomorphism $\CO_B \xrightarrow{\sim} \ol{p}_{*}\CO_{\ol{X}}$ by the Stein factorization. 
    The second isomorphism holds since the scheme $\ol{X}$ is Cohen--Macaulay and the complement of $X$ in $\ol{X}$ is of codimension at least $2$. 
    Set $Y \coloneq J_{\widetilde{\CC}}\times_B P$ and $\ol{Y} \coloneq \ol{J_{\widetilde{\CC}}}\times_B \ol{P}$. Let~$q\colon Y \to B$ and $\ol{q} \colon \ol{Y} \to B$ be the projection map. A similar argument shows that $\Aut(\CR)$ is canonically isomorphic to $H^0(B, \CO_{B}^{\times})$ and we have the commutative diagram
    \begin{equation*}
        \begin{tikzcd}
             \Aut(\CR) \arrow[rr,"\widetilde{{p_{23}}}^{*}"] \arrow[d, equal] & &\Aut(\widetilde{{p_{23}}}^{*}\CR) \arrow[d,equal]\\
             H^0(Y,\CO_Y^{\times}) \arrow[rr,"\widetilde{{p_{23}}}^{*}"]& & H^0(X,\CO_X^{\times})\\
            & H^0(B,\CO_B^{\times}). \arrow[ul, "{\cong}"',"q^{*}"] \arrow[ur,"p^{*}"',"\cong"]&
        \end{tikzcd}
    \end{equation*}
    This shows that $\widetilde{{p_{23}}}^{*} \colon \Aut(\CR) \xrightarrow{\sim} \Aut(\widetilde{p_{23}}^{*}\CR)$ is an isomorphism. 
    
    Now for any $\beta \in \Aut(\widetilde{p_{23}}^{*}\CR)$, set $\beta' = (\widetilde{{p_{23}}}^{*})^{-1}(\beta)$. Equation \eqref{eq: pullback 1} follows from
    \[(\widetilde{m}\times \id_{J_{\widetilde{\CC}}} \times \id_P)^{*}\beta = (\widetilde{m}\times \id_{J_{\widetilde{\CC}}} \times \id_P)^{*} \widetilde{{p_{23}}}^{*}\beta' =\widetilde{p_{234}}^{*} \widetilde{{p_{23}}}^{*}\beta' =\widetilde{p_{234}}^{*} \beta.\]
    Similarly, we can show that $(\id_{J_{\CC}} \times \widetilde{\sigma} \times \id_P)^{*} \colon \Aut(\widetilde{p_{23}}^{*}\CR) \ra \Aut(F)$ is an isomorphism and conclude \eqref{eq: pullback 2}.
    
    \medskip
    \noindent {\bf Step 3.} 
    We show that $\widetilde{\alpha}$ descends to an isomorphism $\alpha \colon (\sigma\times \id_P)^{*}\CR \xrightarrow{\sim}  p_{23}^{*}\CR$ satisfying the cocycle condition \eqref{eq: cocycle condition}. 
    Consider the action of $\Lambda$ on $J_{\CC} \times_B J_{\widetilde{\CC}}\times_B P$ by multiplication on the first factor. It follows that $(\widetilde{\sigma} \times \id_P)^{*} \CR$ and $\widetilde{p_{23}}^{*}\CR$ are $\Lambda$-equivariant sheaves, since they are pulled back from $\pi^{*}J_{\CC} \times_B J_{\widetilde{\CC}} \times_B P$. 
    Their $\Lambda$-equivariant structures are given by involution isomorphisms~$\mu_{\eta} \in \mathrm{Aut}((\widetilde{\sigma} \times \id_P)^{*} \CR)$ and $\lambda_{\eta} \in \mathrm{Aut}(\widetilde{p_{23}}^{*}\CR)$, respectively.

    In Section~\ref{subsection: over smooth curves}, we verified Lemma~\ref{lem: smooth case equivariant} over $B^{(g)}$. 
    This provides an isomorphism 
    \[\alpha^{\circ} \in \Hom(({\sigma}\times \id_P)^{*}\CR|_{B^{(g)}}, {p_{23}}^{*}\CR|_{B^{(g)}})\]
    that satisfies the cocycle condition. 
    Let the morphism $\widetilde{\alpha^{\circ}} \in \Hom((\widetilde{\sigma}\times \id_P)^{*}\CR|_{B^{(g)}} , \widetilde{p_{23}}^{*}\CR|_{B^{(g)}})$ be the pullback of $\alpha^{\circ}$. It then satisfies
    \begin{equation}\label{eq: Lambda-equivariant}
        \widetilde{\alpha^{\circ}} \circ \mu_{\eta} = \lambda_{\eta} \circ \widetilde{\alpha^{\circ}}.
    \end{equation}
    
    By the uniqueness result of Step~2, this pullback morphism $\widetilde{\alpha^{\circ}}$ is the unique isomorphism in $\Hom((\widetilde{\sigma}\times \id_P)^{*}\CR|_{B^{(g)}}, \widetilde{p_{23}}^{*}\CR|_{B^{(g)}})$ that satisfies the cocycle condition, and it thus equals the restriction of the map $\widetilde{\alpha}$ to $B^{(g)}$. The identity $\widetilde{\alpha} \circ \mu_{\eta} = \lambda_{\eta} \circ \widetilde{\alpha}$ then follows from \eqref{eq: Lambda-equivariant}, which implies that $\widetilde{\alpha}$ is $\Lambda$-equivariant. Consequently, the equivariant map $\widetilde{\alpha}$ descends to an isomorphism $\alpha \in \Hom(({\sigma}\times \id_P)^{*}\CR, {p_{23}}^{*}\CR)$ that extends $\alpha^{\circ}$. The cocycle condition \eqref{eq: cocycle condition} holds because it holds over $B^{(g)}$.
\end{proof}

\begin{thm}\label{thm: def G}
    There exists a line bundle $\CG$ on $P\times_B P$ satisfying the following conditions:
    \begin{enumerate}
        \item $\CR \cong (f \times \id_{P})^{*}\CG$;
        \item $\CG$ is trivialized along zero sections $0\times_B P$ and $P\times_B 0$.
    \end{enumerate}
\end{thm}

\begin{proof}
    Lemma~\ref{lem: descent=equivariant} shows that $J_{\widetilde{\CC}} \times_B P$ is a $\pi^{*}J_{\CC}$-torsor over $P \times_B P$. 
    The sheaf $\CR$ on $J_{\widetilde{\CC}} \times_B P$ is a $\pi^{*}J_{\CC}$-equivariant sheaf by Lemma~\ref{lem: equivariant}. Therefore, it descends to a sheaf $\CG'$ on $P\times_B P$. Twisting $\CG'$ by $p_1^{*}(\CG'|_{P\times_B 0})^{-1}$ gives the desired line bundle.
\end{proof}

From the proof, we see that when $B=B^{(g)}$, the above Poincar\'e line bundle coincides with the one defined in Section~\ref{subsection: over smooth curves} using the theta divisor.

\subsection{Construction of the Poincar\'e line bundle without assumptions}\label{subsec: no section}
In this section, we show that the Poincar\'e line bundle constructed above descends to the general case where the family of curves may not admit a section in its smooth locus, and the base $B$ is not necessarily quasi-projective.

We first recall the definition of the Poincar\'e line bundle $\CP$ on $J_{\widetilde{\CC}} \times_B J_{\widetilde{\CC}}$, which is a sheaf locally defined by \eqref{eq: CP}. 
The fact that these locally defined sheaves can be glued into an untwisted global sheaf relies heavily on our restriction to degree~$0$ compactified Jacobians; see \cite[Proposition 3.1]{AddingtonDonovanMeachan16} for the case of any degree.  

Let $\{V_i\}_{i\in I}$ be an \'etale cover of $B$ such that assumptions (A1) and (A2) hold after base-changing to $V_i$. Let $\CU_i$ be a universal sheaf on $\widetilde{\CC}_{V_i} \times_{V_i} J_{\widetilde{\CC}_{V_i}}$. 
We define $\CP_i$ on $J_{\widetilde{\CC}_{V_i}} \times_{V_i} J_{\widetilde{\CC}_{V_i}}$ as in \eqref{eq: CP}. 
On the overlap $V_{ij} = V_i \times_B V_j$, there exists a line bundle $\CL_{ij}$ on $J_{\widetilde{\CC}_{V_{ij}}}$, unique up to isomorphism, such that 
\[\CU_i|_{V_{ij}} \ot p_2^{*}\CL_{ij} \cong \CU_j|_{V_{ij}},\]
where $p_2$ is the projection morphism to the second factor $J_{\widetilde{\CC}_{V_{ij}}}$.

The projection formula of the determinant of cohomology (\cite[Proposition 44]{Esteves01}) induces an isomorphism
\[\phi_{\CL_{ij}} \colon \CP_i|_{V_{ij}} \xrightarrow{\sim} \CP_j|_{V_{ij}}.\]
Since $\CL_{ij}|_{V_{ijk}} \ot \CL_{jk}|_{V_{ijk}} \cong \CL_{ik}|_{V_{ijk}}$, the functorial property (\cite[Proposition 44]{Esteves01}) yields that 
\[ \phi_{\CL_{jk}}|_{V_{ijk}} \circ \phi_{\CL_{ij}}|_{V_{ijk}} = \phi_{\CL_{ik}}|_{V_{ijk}}.\]
Therefore, $(\{\CP_i\}_{i\in I} , \{\phi_{\CL_{ij}}\}_{i,j \in I} )$ defines 
a sheaf $\CP$ on ${J}_{\widetilde{\CC}} \times_B J_{\widetilde{\CC}}$. It is known that the sheaf $\CP$ is independent of the choices of the cover $\{V_i\}_{i\in I}$.

Let $\CR$ be the restriction of $\CP$ to ${J}_{\widetilde{\CC}} \times_B P$. 
Our goal is to show that Lemma~\ref{lem: equivariant} still holds, then Theorem~\ref{thm: def G} follows.
To apply the same proof as in Lemma~\ref{lem: equivariant}, we only need to find an isomorphism satisfying \eqref{eq: cocycle condition for J_C}, which follows directly from the first two steps of Lemma~\ref{lem: equivariant}. Indeed, Step~1 shows that \'etale locally over $V_i$ such an isomorphism exists, then the uniqueness proved in Step~2 guarantees that they define a global isomorphism, as we now explain. 

Let $\CR_i$ be the restriction of $\CP_i$ to $J_{\widetilde{\CC}_{V_i}} \times_{V_i} P_{V_i}$. Then $\phi_{\CL_{ij}}$ induces
\[\phi_{ij} \colon \CR_i|_{V_{ij}} \xrightarrow{\sim} \CR_j|_{V_{ij}},\]
and $(\{\CR_i\}_{i\in I} , \{\phi_{ij}\}_{i,j \in I} )$ defines $\CR$. 
By Steps~1 and 2 of Lemma~\ref{lem: equivariant}, there exists a unique isomorphism
\[\widetilde{\alpha_i} \colon (\widetilde{\sigma} \times \id_P)^{*} \CR_i \xrightarrow{\sim} \widetilde{p_{23}}^{*}\CR_i\]
that satisfies the cocycle condition \eqref{eq: cocycle condition for J_C}. 
It remains to prove that $\{\widetilde{\alpha_i}\}_{i\in I}$ defines an isomorphism $\widetilde{\alpha} \colon (\widetilde{\sigma} \times \id_P)^{*} \CR \xrightarrow{\sim} \widetilde{p_{23}}^{*}\CR$, in other words, the diagram
\begin{equation*}
    \begin{tikzcd}
        (\widetilde{\sigma} \times \id_P)^{*} \CR_i|_{V_{ij}} \arrow[r,"\widetilde{\alpha_i}|_{V_{ij}}"] \arrow[d, "(\widetilde{\sigma} \times \id_P)^{*}\phi_{ij}"',"{\rotatebox{270}{$\sim$}}"] & \widetilde{p_{23}}^{*}\CR_i|_{V_{ij}} \arrow[d, "\widetilde{p_{23}}^{*}\phi_{ij}", "{\rotatebox{270}{$\sim$}}"'] \\        
        (\widetilde{\sigma} \times \id_P)^{*} \CR_j|_{V_{ij}} \arrow[r,"\widetilde{\alpha_j}|_{V_{ij}}"]  & \widetilde{p_{23}}^{*}\CR_j|_{V_{ij}}
    \end{tikzcd}
\end{equation*}
commutes. By the uniqueness argument, it suffices to check that the map 
\[ (\widetilde{p_{23}}^{*}\phi_{ij})^{-1} \circ \widetilde{\alpha}_j|_{V_{ij}}\circ (\widetilde{\sigma}\times \id_P)^{*}\phi_{ij}\]
also satisfies the cocycle condition \eqref{eq: cocycle condition for J_C}, 
which is true by a simple calculation similar to \eqref{eq: translation by beta}. 

\subsection{Stratification of Prym varieties}
In this section, we show that $\CG$ extends to a maximal Cohen--Macaulay sheaf on $\ol{P} \times_B \ol{P}$. 

To start with, we need to recall the following properties of compactified Jacobians (cf. \cite{Cook98}).

Let $\Gamma$ be an integral projective curve over $k$. 
In \cite{Cook98}, Cook defined a \emph{local type} on $\Gamma$ to be a collection $\ul{M} = \{M_p\}_{p\in \Sing(\Gamma)}$ of isomorphism classes of torsion-free, rank 1 $\CO_{\Gamma,p}$-modules for~$p$ running through all the singularities of $\Gamma$. 

A torsion-free, rank 1 sheaf $F$ on $\Gamma$ is said to be \emph{of local type} $\{M_p\}_{p\in \Sing(\Gamma)}$ if $F_p$ is isomorphic to $M_p$ as $\CO_{\Gamma,p}$-modules for every singular point $p$. 

\begin{prop}(\cite[Lemma 5]{Cook98})\label{statification of Jac}
    For every local type $\ul{M}$, let $U_{\ul{M}} \subset \ol{\Jac}(\Gamma)$ be the 
    subset parameterizing torsion-free, rank 1 degree 0 sheaves of local type $\ul{M}$. 
    Then $U_{\ul{M}}$ is a nonempty locally closed subvariety.     
    Take $F \in U_{\ul{M}}$, then $\Jac(\Gamma)$ acts transitively on $U_{\ul{M}}$, with affine stabilizer $\ker\big(n^{*} \colon \Jac(\Gamma) \ra \Jac(\Gamma')\big)$, where $\Gamma' \coloneqq \Spec_{\CO_{\Gamma}}(\endo(F))$ and $n \colon \Gamma' \ra \Gamma$ is the partial normalization.
\end{prop}

From this proposition we see that the partial normalization $\Gamma'$ only depends on the local type, and we call $\Gamma'$ the \emph{partial normalization corresponding to} $\ul{M}$. 

Now let $\pi \colon \widetilde{\Gamma} \ra \Gamma$ be an \'etale double cover. 
Let $\{x_1,\dots,x_n\}$ be the singular points of $\Gamma$. For each $x_i$, write $\pi^{-1}(x_i) = \{p_i,q_i\}$, so that $\{p_1,\dots, p_n, q_1,\dots, q_n\}$ are the singular points of the curve $\widetilde{\Gamma}$. 
In particular, we have $\CO_{\Gamma,x_i} \cong \CO_{\widetilde{\Gamma},p_i} \cong \CO_{\widetilde{\Gamma},q_i}$.
Then Proposition~\ref{statification of Jac} has a parallel for compactified Prym varieties.

\begin{prop}\label{stratification of Prym}
    The compactified Prym variety $\ol{\Prym}(\widetilde{\Gamma}/\Gamma)$ is stratified by nonempty locally closed subvarieties
    \[V_{\ul{M}} \coloneqq \ol{\Prym}(\widetilde{\Gamma}/\Gamma) \cap U_{\ul{M}},\]
    where $\ul{M}$ runs through local types $\{M_{p_1}, \dots, M_{p_n}, M_{q_1}, \dots, M_{q_n}\}$ in which $M_{p_i}$ is isomorphic to $M_{q_i}^{\vee}$ as $\CO_{\Gamma,x_i}$-modules. 
    Furthermore, let $\widetilde{n} \colon \widetilde{\Gamma}' \ra \widetilde{\Gamma}$ be the partial normalization of~$\widetilde{\Gamma}$ corresponding to $\ul{M}$, and $n\colon \Gamma' \ra \Gamma$ be the partial normalization of $\Gamma$ corresponding to $\ul{N}=\{N_{x_1}, \dots, N_{x_n}\}$ where $N_{x_i} \cong M_{p_i} \cong M_{q_i}^{\vee}$.      
    Then $\Prym(\widetilde{\Gamma}/\Gamma)$ acts transitively on $V_{\ul{M}}$, with affine stabilizer $\ker\big(\widetilde{n}^{*} \colon \Prym(\widetilde{\Gamma}/\Gamma) \ra \Prym(\widetilde{\Gamma}'/\Gamma') \big)$.
\end{prop}

\begin{proof}
    Fix $\ul{M}$ as in the statement of this proposition. Proposition~\ref{statification of Jac} yields that there exists a sheaf $F\in \ol{\Jac}(\widetilde{\Gamma})$ of local type 
    \[\{M_{p_1}, \dots, M_{p_n}, \CO_{q_1}, \dots, \CO_{q_n}\},\] 
    where $\CO_{q_i}\coloneqq\CO_{\widetilde{\Gamma},q_i}$. 
    In particular, $E \coloneq F \ot \iota^{*}F^{\vee}$ lies in $V_{\ul{M}}$. 

    For any sheaf $H \in V_{\ul{M}}$, there exists $L \in \Prym(\widetilde{\Gamma}/\Gamma)$ such that $H = E \ot L$ by \cite[Corollary~4.17]{LSV17}. 
    Thus the action of $\Prym(\widetilde{\Gamma}/\Gamma)$ on $ V_{\ul{M}}$ is transitive. 
    
    By definition $\widetilde{\Gamma}'=\Spec_{\CO_{\widetilde{\Gamma}}}(\endo(E))$. It is easy to check that we have the Cartesian diagram
    \begin{equation*}
        \begin{tikzcd}
            \widetilde{\Gamma}' \arrow["\epsilon"', loop, distance=1.5em, in=205, out=155] \arrow[r, "\pi'"] \arrow[d, "\widetilde{n}"'] & \Gamma' \arrow[d, "n"]\\
        \widetilde{\Gamma} \arrow["\iota"', loop, distance=1.5em, in=205, out=155] \arrow[r, "\pi"] & \Gamma,
        \end{tikzcd}
    \end{equation*}
    where $\epsilon$ and $\iota$ are involution maps associated to the \'etale double covers $\pi'$ and $\pi$, respectively. 
    Alternatively, the involution map $\epsilon$ on $\widetilde{\Gamma}'$ is induced by the canonical isomorphism
    \[\iota^{*}\endo(E) \cong \endo(\iota^{*}E) = \endo(E^{\vee}) \cong \endo(E).\] 
    
    Again, using Proposition \ref{statification of Jac} we see that the stabilizer is 
    \[\Prym(\widetilde{\Gamma}/\Gamma) \cap \ker\big( \Jac(\widetilde{\Gamma}) \ra \Jac(\widetilde{\Gamma}')\big) = \ker\big(\Prym(\widetilde{\Gamma}/\Gamma) \ra \Prym(\widetilde{\Gamma}'/\Gamma')\big).\qedhere\]
\end{proof}

Now we go back to our situation. 
Let $\ol{J}_{\widetilde{\CC}}^{\circ} \subset \ol{J}_{\widetilde{\CC}}$ be the open subset parameterizing relative torsion-free, rank~$1$, degree~$0$ sheaves $F$ that are locally free at least at one of the two points $x$ and $\iota x$ for every closed point $x$ in the singular locus of $\widetilde{\CC} \ra B$. 
Then the group homomorphism \eqref{eq: def of f} extends to a regular map 
\begin{equation}
    \begin{aligned}
        f \colon \ol{J}_{\widetilde{\CC}}^{\circ} & \ra  \ol{J}_{\widetilde{\CC}} \\
        F  & \longmapsto F \ot \iota^{*}F^{\vee}
    \end{aligned}
\end{equation}
by \cite[Lemma 4.12]{LSV17}.

\begin{lem}
    The morphism $f \colon \ol{J}_{\widetilde{\CC}}^{\circ} \ra \ol{P}$ is faithfully flat.
\end{lem}
\begin{proof}
    Surjectivity is given by \cite[Corollary 4.17]{LSV17}. 
    For flatness, we will apply the miracle flatness theorem. Since $\ol{J}_{\widetilde{\CC}} \ra B$ is a local complete intersection morphism over a nonsingular base $B$ \cite[Theorem 9]{AIK77}, the total space $\ol{J}_{\widetilde{\CC}}$ is Cohen--Macaulay, so is the open subvariety $\ol{J}_{\widetilde{\CC}}^{\circ}$. Since $\ol{P}$ is nonsingular, it remains to verify that the morphism $f$ has constant fiber dimension. 

    Let $b\in B$ be a closed point. Set $\widetilde{\Gamma} = \widetilde{\CC}_b$ and $\Gamma = \CC_b$. 
    We adopt the notation of Proposition~\ref{stratification of Prym}. 
    For any $E\in V_{\ul{M}}$, we have 
    \begin{equation}\label{eq: preimage disjoint union}
        f^{-1}(E) = \bigsqcup_{\ul{M'}} (f^{-1}(E) \cap U_{\ul{M'}}) \subset \Jac(\widetilde{\Gamma}),
    \end{equation}
    where $\ul{M'}$ runs over local types satisfying, for each $i$, either $M_{p_i}' \cong M_{p_i}$ and $M_{q_i}' \cong \CO_{\widetilde{\Gamma},q_i}$, or~$M_{p_i}' \cong \CO_{\widetilde{\Gamma},p_i}$ and $M_{q_i}' \cong M_{p_i}^{\vee}$. 
    
    Since the right-hand side of \eqref{eq: preimage disjoint union} is a disjoint union of countably many equidimensional schemes, it suffices to consider the preimages of $E$ lying in $U_{\ul{M'}}$, where 
    \[\ul{M'}=\{M_{p_1},\dots,M_{p_n}, \CO_{\widetilde{\Gamma},q_1},\dots, \CO_{\widetilde{\Gamma},q_n}\}.\] 
    By Proposition~\ref{stratification of Prym}, we can take $F\in U_{\ul{M'}}$ such that $F \ot \iota^{*}F^{\vee} \cong E$. 
    Then we have 
    \begin{align*}
        f^{-1}(E) \cap U_{\ul{M'}} &= 
        \left\{ F \ot L \mid L\in \Jac(\widetilde{\Gamma}),\ (1-\iota^{*})L \text{ fixes } E \right\} \\
        &=  \left\{ F \ot L \mid L\in \Jac(\widetilde{\Gamma}),\ (1-\iota^{*})L\in \ker\big(\Prym(\widetilde{\Gamma}/\Gamma) \ra \Prym(\widetilde{\Gamma}'/\Gamma')\big)\right\},
    \end{align*}
    which is isomorphic to the quotient scheme
    \begin{equation}\label{eq: coset}
        \frac{H_1}{H_2}\coloneq \frac{(1-\iota^{*})^{-1}\big( \ker\big(\Prym(\widetilde{\Gamma}/\Gamma) \ra \Prym(\widetilde{\Gamma}'/\Gamma')\big)  \big)}{\ker\big(\Jac(\widetilde{\Gamma}) \ra \Jac(\widetilde{\Gamma}'')\big)},
    \end{equation}
    where $\widetilde{\Gamma}^{''} = \Spec_{\CO_{\widetilde{\Gamma}}}(\endo(F))$. The affine group action of $H_2$ on $H_1$ is induced by the tensor product of line bundles; thus, it is free. 
    
    The partial normalization map $\widetilde{n}\colon 
    \widetilde{\Gamma}' \ra \widetilde{\Gamma}$ induces the natural short exact sequence of groups
    \begin{equation}\label{exact seq of Jac(tilde C)}
        1 \ra A\times A \ra \Jac(\widetilde{\Gamma}) \xrightarrow{\widetilde{n}^{*}} \Jac(\widetilde{\Gamma}') \ra 1,
    \end{equation}
    where $A \coloneqq H^0(\widetilde{\Gamma}, \bigoplus_{i=1}^n (\endo(E)^{\times} / \CO_{\widetilde{\Gamma}}^{\times})_{p_i}) 
    \cong 
    H^0(\widetilde{C}, \bigoplus_{i=1}^n (\endo(E)^{\times} / \CO_{\widetilde{\Gamma}}^{\times})_{q_i})$ is a commutative affine group. 
    Since $-\iota^{*}$ is compatible with $-\epsilon^{*}$ via $\widetilde{n}^{*}$, we get the exact sequence
    \begin{equation}\label{exact sequence of Prym}
        1 \ra A \ra \Prym(\widetilde{\Gamma}/\Gamma) \xrightarrow{\widetilde{n}^{*}}  \Prym(\widetilde{\Gamma}'/\Gamma') \ra 1.
    \end{equation}
    Similarly, the partial normalization map $n_1\colon \widetilde{\Gamma}'' \ra \widetilde{\Gamma}$ induces the short exact sequence
    \begin{equation}\label{eq: exact sequence of n_1}
        1 \ra B \ra \Jac(\widetilde{\Gamma}) \xrightarrow{n_1^{*}} \Jac(\widetilde{\Gamma}'') \ra 1,
    \end{equation}
    where $B \coloneqq H^0(\widetilde{\Gamma}, \bigoplus_{i=1}^n (\endo(F)^{\times} / \CO_{\widetilde{\Gamma}}^{\times})_{p_i}) \cong A$. 

    Now, since the morphism $1-\iota^{*} \colon \Jac(\widetilde{\Gamma}) \ra \Prym(\widetilde{\Gamma}/\Gamma)$ is faithfully flat of relative dimension~$g$, we obtain from \eqref{eq: coset}, \eqref{exact sequence of Prym}, and \eqref{eq: exact sequence of n_1} that 
    \[\dim f^{-1}(E) = \dim H_1 - \dim H_2 = g+ \dim A - \dim B = g.\]
    Since $E$ is chosen arbitrarily, this implies that every fiber of $f$ is of dimension $g$.
\end{proof}

\subsection{Extension of the Poincar\'e sheaf to $\ol{P} \times_B \ol{P}$}
We denote by $i\colon \ol{P} \ra \ol{J}_{\widetilde{\CC}}$ the closed embedding and $\ol{\CR} \coloneqq (\id_{\ol{J}_{\widetilde{\CC}}} \times i)^{*}\ol{\CP}$ the pullback sheaf on $\ol{J}_{\widetilde{\CC}}\times_B \ol{P}$. Our next goal is to use the following criterion to show that $\ol{\CR}$ is maximal Cohen--Macaulay.
\begin{lem}\label{lem: criterion for MCM}
    Let $X$ and $Y$ be locally Noetherian schemes, and let $f \colon X\ra Y$ be a surjective morphism. Assume that $\CN$ is a sheaf on $X$ that is flat over $Y$. Then $\CN$ is a Cohen--Macaulay sheaf if and only if $Y$ is a Cohen--Macaulay scheme and $\CN|_{X_y}$ is a Cohen--Macaulay sheaf for every point $y\in Y$. 
\end{lem}

\begin{proof}
    Take $x\in X$ and set $y=f(x)$. 
    Since the Cohen--Macaulay property is local, it suffices to show that if $\CN_x$ is a flat $\CO_{Y,y}$-module, then $\CN_x$ is a Cohen--Macaulay $\CO_{X,x}$-module if and only if $\CO_{Y,y}$ is a Cohen--Macaulay ring and $\CN_x\ot_{\CO_{Y,y}}\kappa(y)$ is a Cohen--Macaulay $\CO_x\ot_{\CO_{Y,y}} \kappa(y)$-module. This is a direct consequence of \cite[Corollary 6.3.3]{EGA4.2}.
\end{proof}

\begin{lem}\label{R bar is MCM}
    The sheaf $\ol{\CR}$ on $\ol{J}_{\widetilde{\CC}} \times_B \ol{P}$ is maximal Cohen--Macaulay.
\end{lem}
\begin{proof}    
    According to \cite[Theorem A]{Ari10b}, $\ol{\CP}$ is flat over the second factor of $\ol{J}_{\widetilde{\CC}}\times_B \ol{J}_{\widetilde{\CC}}$, and the restriction of $\ol{\CP}$ to the fiber $\ol{J}_{\widetilde{\CC}_b} \times \{F\}$ over $F$ is maximal Cohen--Macaulay for every $b\in B$ and $F\in \ol{P}_b$. 
    Therefore, the sheaf $\ol{\CR}$ is flat over $\ol{P}$, and the restriction of $\ol{\CR}$ to the fiber over~$F \in \ol{P}$ is maximal Cohen--Macaulay. 
    Then the conclusion follows from Lemma \ref{lem: criterion for MCM} since~$\ol{P}$ is nonsingular.
\end{proof}

Now we are ready to construct the normalized Poincar\'e sheaf on $\ol{P} \times_B \ol{P}$. Let $u\colon \ol{J}_{\widetilde{\CC}}^{\circ} \hookrightarrow \ol{J}_{\widetilde{\CC}}$ be the open embedding. We recall the notation from the diagram below:
\begin{equation*}
    \begin{tikzcd}
       & \ol{\CR} \arrow[d, squiggly]& \ol{\CP}\arrow[d, squiggly]\\
       \ol{J}_{\widetilde{\CC}}^{\circ} \times_B \ol{P} \arrow[d, two heads,"f\times \id_{\ol{P}}"] \arrow[r, hook, "u \times \id_{\ol{P}}"] & \ol{J}_{\widetilde{\CC}}\times_B \ol{P} \arrow[r, hook, "\id_{\ol{J}}\times i"] & \ol{J}_{\widetilde{\CC}}\times_B \ol{J}_{\widetilde{\CC}}\\
       \ol{P} \times_B \ol{P}.& & 
    \end{tikzcd}
\end{equation*}

\begin{thm}\label{thm: def of G bar}
    There exists a sheaf $\ol{\CG}$ on $\ol{P} \times_B \ol{P}$ satisfying the following conditions:
    \begin{enumerate}
        \item $(f\times \id_{\ol{P}})^{*}\ol{\CG} \cong (u\times \id_{\ol{P}})^{*} \ol{\CR}$;
        \item $\ol{\CG}$ is trivialized along both $0\times_B \ol{P}$ and $\ol{P} \times_B 0$.
    \end{enumerate}
\end{thm}
\begin{proof}
    Since $u$ is an open embedding, the sheaf $(u\times \id_{\ol{P}})^{*} \ol{\CR}$ is maximal Cohen--Macaulay by Lemma~\ref{R bar is MCM}.
    The Severi inequality ensures that the general fibers of $\widetilde{\CC} \ra B$ are smooth, so the complement of $J_{\widetilde{\CC}}\times_B P$ in $J_{\widetilde{\CC}}^{\circ}\times_B \ol{P}$ has codimension at least 2. 
    Furthermore, the image of~$J_{\widetilde{\CC}}\times_B P$ under the map $f\times \id_{\ol{P}}$ is contained in $P\times_B P$.
    In Theorem~\ref{thm: def G}, we construct a line bundle $\CG$ on~ $P \times_B P$ such that 
    $(f\times \id_{\ol{P}}|_{J_{\widetilde{\CC}}\times_B P})^{*}\CG \cong \CR,$
    where $\CG$ is maximal Cohen--Macaulay since~$P\times_B P$ is nonsingular. 
    Therefore, we can extend $\CG$ to a sheaf $\ol{\CG}$ satisfying Condition~(1) by \cite[Lemma~4.14]{MRV19b}, and Condition (2) follows from the fact that $\CG$ is normalized. 
\end{proof}

\section{Properties of the Poincar\'e sheaf}\label{sec_properties}
We keep the notation and assumptions as in Section~\ref{subsec1.1: notations}.
\begin{prop}\label{prop: uniqueness, flat}
    The conditions in Theorem \ref{thm: def of G bar} uniquely determine the sheaf $\ol{\CG}$ on $\ol{P}\times_B \ol{P}$. Moreover, it has the following properties:
    \begin{enumerate}
        \item The restriction $\ol{\CG}|_{\ol{P}_b \times \{F\}}$ is a maximal Cohen--Macaulay sheaf for every $b\in B$ and $F \in \ol{P}_b$;
        \item Let $j \colon P\times_B P \hookrightarrow \ol{P} \times_B \ol{P}$ be the open embedding. Then $\ol{\CG} \cong j_{*}\CG$. Consequently, the normalized line bundle $\CG$ constructed in Theorem~\ref{thm: def G} is also unique;
        \item $\ol{\CG}$ is flat over both factors.
    \end{enumerate}
\end{prop}
\begin{proof}
    We first verify the stated properties.  
    Property~(1) follows from the fact that fpqc descent preserves the maximal Cohen--Macaulay property. Since $(u\times_B \id_{\ol{P}})^*\ol{\CR}$ is maximal Cohen--Macaulay, we also obtain that $\ol{\CG}$ is maximal Cohen--Macaulay. 
    Then Property~(2) follows from the extension property of Cohen--Macaulay sheaves (see, for example, \cite[Lemma 2.2]{Ari10b}).
    Finally, property~(3) follows from the miracle flatness theorem (\cite[Proposition (6.1.5)]{EGA4.2}). 
    
    We now turn to the proof of uniqueness. Let $\ol{\CG}$ and $\ol{\CH}$ be sheaves satisfying conditions~(1) and~(2) of Theorem~\ref{thm: def of G bar}. 
    By abuse of notation, we denote by $\CG$ and $\CH$ their restrictions to~$P \times_B \ol{P}$, respectively. They are line bundles satisfying $(f\times\id_{\ol{P}})^{*}\CG \cong (f\times\id_{\ol{P}})^{*}\CH$, and their restriction to $P\times_B 0$ are both trivial. Lemma~\ref{lem: uniqueness of descent} below implies that $\CG \cong \CH$. By the extension of Cohen--Macaulay sheaves, we conclude that $\ol{\CG} \cong j_{*}\CG \cong j_{*}\CH \cong \ol{\CH}$.
\end{proof}

To finish the proof, we first recall a version of the seesaw principle for line bundles. We adapt \cite[Corollary~5.6]{MRV19b} to our setting and the proof carries over without change. 
\begin{lem}\label{lem: seesaw principle}
    Let $S$ be a reduced scheme, locally of finite type over $k$. Let $g \colon Z \ra S$ be a flat and proper morphism with geometrically integral fibers. Assume that $g$ admits a section $s \colon S \ra Z$. 
    Let $\CE$ and $\CF$ be two line bundles on $Z$ such that: 
    \begin{enumerate}
        \item $\CE|_{Z_t} \cong \CF|_{Z_t}$ for every closed point $t\in S$;
        \item there exists an isomorphism $\varphi\colon  s^*\CE \xrightarrow{\sim} s^*\CF$ over $S$.
    \end{enumerate}
    Then there exists a unique isomorphism $\phi\colon \CE \xrightarrow{\sim} \CF$ that lifts $\varphi$.
\end{lem}

\begin{lem}\label{lem: uniqueness of descent}
    Let $S, Z, g,$ and $s$ be as in Lemma~\ref{lem: seesaw principle}. Let $h \colon X \to Y$ be a surjective morphism of reduced $S$-schemes that are locally of finite type over $k$. Let $s_Y \colon Y \to Y \times_S Z$ denote the base change of the section $s$. 
    Then for line bundles $L_1$ and $L_2$ on $Y\times_S Z$, we have $L_1 \cong L_2$ if and only if $(h\times_S \id_Z)^{*}L_1 \cong (h\times_S \id_Z)^{*}L_2$ and $s_Y^{*}L_1 \cong s_Y^{*}L_2$. 
\end{lem}
\begin{proof}
    The forward implication is clear. For the converse, it suffices to set $L \coloneqq L_1 \ot L_2^{-1}$ and show that $L \cong \CO_{Y\times_S Z}$ provided that $(h\times_S \id_Z)^{*}L \cong \CO_{X \times_S Z}$ and $s_Y^* L \cong \CO_Y$. 
    
    Let $p_1\colon Y\times_S Z \ra Y$ be the first projection, which is again flat and proper with geometrically integral fibers. For every closed point $y\in Y$, take a closed point $x\in h^{-1}(y)$. Then $L|_{\{y\}\times Z}$ is isomorphic to $((h\times \id_Z)^{*}L)|_{\{x\}\times Z}$ as sheaves on $Z$, and thus is a trivial line bundle. It then follows from Lemma~\ref{lem: seesaw principle} that $L$ is trivial.
\end{proof}

\begin{rmk}\label{rmk: uniqueness of CG}
    In the proof of uniqueness, we essentially use the assumption that $\ol{P}$ is nonsingular. 
    In this case, though the normalized Poincar\'e sheaf is unique up to isomorphism, the isomorphism is not unique. We actually have
    \[\Aut(\ol{\CG}) \cong \Aut(\CG) \cong H^0(P\times_B \ol{P}, \CO_{P\times_B \ol{P}}^{*}) \cong H^0(P,\CO_P^*),\]
    where the first isomorphism holds because the sheaf $\hom(\ol{\CG},\ol{\CG})$ has property $(S_2)$ (\cite[Tag~0AXQ]{stacks-project}).
    However, if we consider normalized Poincar\'e sheaves with rigidifications, 
    \[(\ol{\CG}, \varphi\colon \ol{\CG}|_{\ol{P} \times_B 0}\xrightarrow{\sim} \CO_{\ol{P}}), \quad (\ol{\CH}, \psi\colon \ol{\CH}|_{\ol{P} \times_B 0}\xrightarrow{\sim} \CO_{\ol{P}}),\] 
    then there uniquely exists an isomorphism $\phi\colon \ol{\CG} \to \ol{\CH}$ that is compatible with the rigidifications. 
\end{rmk}

\subsection{Invariance of $\ol{\CG}$ under swapping map on $\ol{P}\times_B \ol{P}$}
We first recall general facts about compactified Picard schemes. 
Let $X\ra S$ be a flat, finitely presented, projective morphism between $k$-schemes, with integral geometric fibers. Let $\underline{\Pic}_{(X/S)}^{=}$ be the moduli functor of relative torsion-free rank 1 sheaves. Then $\underline{\Pic}_{(X/S), \et}^{=}$, the \'etale sheaf associated to the functor~$\underline{\Pic}_{(X/S)}^{=}$, is representable by a scheme $\Pic_{(X/S)}^{=}$ whose connected components are proper over $S$; see \cite[Theorem 3.1]{AK79b}. We call $\Pic_{(X/S)}^{=}$ the moduli space of relative torsion-free rank 1 sheaves on $X$. 

Let $\Pic_{(X/S)}^0 \subset \Pic_{(X/S)}$ be the union of the connected components of the identity 0 in fibers of $\Pic_{(X/S)} \ra S$. This is an open subset and admits an induced scheme structure. Let $\ol{\Pic}_{(X/S)}^0$ be the scheme-theoretic closure of $\Pic_{(X/S)}^0$ in $\Pic_{(X/S)}^{=}$. 
If $X/S$ admits a section that lies in the smooth locus, then there exists a universal sheaf on $X \times_S \ol{\Pic}_{(X/S)}^0$ by considering the rigidification functor; see \cite[Theorem 3.4 (3)]{AK79b}. 

To simplify notation, we write $\Pic(X)^{=}$, $\Pic^0(X)$, and $\ol{\Pic}^0(X)$ instead when $S = \Spec k$.

The normalized Poincar\'e line bundle $\CG$ on $\ol{P} \times_B P$ defines a $B$-morphism $\rho \colon P \ra \Pic_{(\ol{P}/B)}$. 
Since $\rho(\CO_{\widetilde{\CC}_b}) = \CO_{\ol{P}_b}$ and each ${P}_b$ is connected, the morphism $\rho$ factors through ${\Pic}_{(\ol{P}/B)}^0$ and we write
\begin{equation}\label{eq: def of rho}
    \rho \colon P \ra \Pic_{(\ol{P}/B)}^0.
\end{equation}
This morphism extends to a morphism defined by $\ol{\CG}$
\[\ol{\rho} \colon \ol{P} \ra \ol{\Pic}_{(\ol{P}/B)}^0,\]
which sends $F\in \ol{P}_b$ to $\ol{\CG}|_{\ol{P}_b \times F} \in \Pic(\ol{P}_b)^{=}$. For the same reason, $\ol{\CG}$ also defines a morphism
\[\ol{\rho}'\colon \ol{P} \ra \ol{\Pic}_{(\ol{P}/B)}^0,\]
which sends $F\in \ol{P}_b$ to $\ol{\CG}|_{F \times \ol{P}_b } \in \Pic(\ol{P}_b)^{=}$. 
Since $\ol{\Pic}_{(\ol{P}/B)}^0$ is a closed subscheme of a connected component of $\Pic^{=}_{(\ol{P}/B)}$, it is proper and thus separated over $B$.

\begin{prop}\label{prop: symmetric}
    Let $sw \colon \ol{P} \times_B \ol{P} \ra \ol{P} \times_B \ol{P}$ be the morphism swapping the factors. Then 
    \begin{equation}\label{eq: invariant under swap}
        sw^{*}\ol{\CG} \cong \ol{\CG}.
    \end{equation}
\end{prop}
\begin{proof}
    As before, let $B^{(g)}$ be the open subset of $B$ parameterizing smooth curves. Then the restriction of 
    $\ol{\CG}$ to $P_{B^{(g)}} \times_{B^{(g)}} P_{B^{(g)}}$ is the normalized Poincar\'e line bundle corresponding to the canonical principal polarization $\varphi_{\Xi}$, which is symmetric; see, e.g., \cite[\S8, Proposition 2]{Mumford70}.
    
    Therefore, $\rho$ and $\rho'$ are $B$-morphisms identified over ${B^{(g)}}$. 
    Since $P_{B^{(g)}}$ is a dense open subset of $\ol{P}$, $\ol{P}$ is reduced, and $\ol{\Pic}_{(\ol{P}/B)}^0$ is separated over $B$, it follows that $\ol{\rho} = \ol{\rho}'$. 
    Since $\ol{P} \ra B$ has a section, there exists a universal sheaf $\CI$ on $\ol{P} \times_B \ol{\Pic}_{(\ol{P}/B)}^0$. Therefore,  
    \[sw^{*}\ol{\CG} \cong (\id_{\ol{P}} \times \ol{\rho}')^{*}\CI = (\id_{\ol{P}} \times \ol{\rho})^{*}\CI \cong \ol{\CG}.\qedhere\]
\end{proof}

\subsection{Duality}
As in the case of abelian varieties, we have a duality theorem.
Let $\nu: \ol{P} \ra \ol{P}$ be the morphism sending $F$ to $F^{\vee} = \hom(F ,\CO_{\widetilde{\CC}})$. 
\begin{lem}\label{lem: symmetric}
    $(\nu \times \id_{\ol{P}})^{*}\ol{\CG} \cong (\id_{\ol{P}} \times \nu)^{*} \ol{\CG} \cong \ol{\CG}^{\vee}$.
\end{lem}
\begin{proof}
    Since both sides are maximal Cohen--Macaulay sheaves, we only need to check that they are the same on $P\times_B \ol{P}$. By \cite[Lemma 6.2]{Ari10b}, we have
    \[(f \times \id_{\ol{P}})^{*}\circ(\nu \times \id_{\ol{P}})^{*}{\CG} \cong (\nu \times \id_{\ol{P}})^{*} \CR \cong \CR^{-1} \cong (f \times \id_{\ol{P}})^{*}\CG^{-1}.\]
    The first isomorphism then follows from Lemma \ref{lem: uniqueness of descent}. A similar proof establishes the second isomorphism.
\end{proof}

\subsection{Theorem of the square}
\begin{lem}\label{lem: thm of square}
    Consider the action 
    \[\mu \colon P \times_B \ol{P} \ra \ol{P}, \quad (L,F) \mapsto L\ot F.\]
    \begin{enumerate}
        \item Consider the diagram
            \begin{equation*}
                \begin{tikzcd}
                P\times_B \ol{P} & P \times_B \ol{P} \times_B \ol{P} \arrow[l, "p_{13}"'] \arrow[r, "p_{23}"] \arrow[d, "\mu \times \mathrm{id}_{\ol{P}}"] & \ol{P} \times_B \ol{P} \\
                & \ol{P} \times_B \ol{P}. &
            \end{tikzcd}
            \end{equation*}
            We have $(\mu \times \id_{\ol{P}})^{*}\ol{\CG} \cong p_{13}^{*}(\CG) \ot p_{23}^{*}(\ol{\CG})$.
        \item Consider the diagram
            \begin{equation*}
                \begin{tikzcd}
                    \ol{P}\times_B {P} & \ol{P}\times_B P \times_B \ol{P} \arrow[l, "p_{12}"'] \arrow[r, "p_{13}"] \arrow[d, "\id_{\ol{P}} \times \mu"] & \ol{P} \times_B \ol{P} \\
                    & \ol{P} \times_B \ol{P}. &
                \end{tikzcd}
            \end{equation*}
            We have $(\id_{\ol{P}} \times \mu)^{*}\ol{\CG} \cong p_{12}^{*}(\CG) \ot p_{13}^{*}(\ol{\CG})$.
    \end{enumerate}
    
\end{lem}

\begin{proof}
    To prove (1), notice that $\mu$ equals the composition
    \[P \times_B \ol{P} \xrightarrow{\sim} P \times_B \ol{P} \xrightarrow{p_2} \ol{P}, \quad (L,F) \mapsto (L, L \ot F) \mapsto L\ot F,\]
    which is therefore smooth. Since both $(\mu \times \id_{\ol{P}})^{*}\ol{\CG}$ and $p_{13}^{*}(\CG) \ot p_{23}^{*}(\ol{\CG})$ are maximal Cohen--Macaulay (\cite[Lemma 2.3]{Ari10b}), it suffices to check that their restrictions to $P\times_B P \times_B \ol{P}$ are isomorphic. 
    
    By \cite[Lemma 6.5]{Ari10b}, the pullbacks of these bundles along $f\times f \times \id_{\ol{P}}$ are isomorphic line bundles on $J_{\widetilde{\CC}} \times_B J_{\widetilde{\CC}} \times_B \ol{P}$, so (1) follows from Lemma~\ref{lem: uniqueness of descent}.
    Similarly, we get (2) by pulling back to $J_{\widetilde{\CC}} \times_B P \times_B \ol{P}$ and using Lemma~\ref{lem: uniqueness of descent}.
\end{proof}

\section{Autoequivalence of the derived category}\label{sec_Autoequivalence}
In this section we prove Theorem~\ref{thm B}. As before, we follow the notation and assumptions as in Section~\ref{subsec1.1: notations}. Recall that $B$ is an irreducible scheme, $\CC \to B$ is a flat and projective family of integral curves with planar singularities, and $\pi \colon \widetilde{\CC} \to \CC$ is an \'etale double cover. The relative compactified Prym variety $\ol{P}$ is nonsingular.

We denote by $\ol{P}^{n+1}$ the $(n+1)$-th relative product of $\ol{P}$ over $B$, $n\geq 0$. It has natural projections
\[\pi_{n+1} \colon \ol{P}^{n+1} \ra B, \quad p_i \colon \ol{P}^{n+1} \ra \ol{P}, \quad p_{i,j} \colon \ol{P}^{n+1} \ra \ol{P}^2, \quad \cdots.\]
Let $\Fj$ be the commutative Lie algebra of the abelian group scheme $p \colon P \ra B$, which is a rank~$g-1$ vector bundle on $B$. 

Recall that the relative compactified Prym variety $\ol{P}$ is nonsingular and that the morphism $\pi_1$ is flat; thus, it is a Gorenstein morphism of relative dimension $g-1$. The relative dualizing complex for $\pi_1$ reduces to an invertible sheaf $\omega_{\pi_1}$ on $\ol{P}$, called the \emph{relative dualizing sheaf}, concentrated at degree $-g+1$. In other words, $\omega_{\pi_1}^{\bullet} = \omega_{\pi_1}[g-1]$.

\begin{lem}\label{lem: relative dualizing sheaf of ol p}
    The relative dualizing sheaf $\omega_{\pi_1}$ is isomorphic to $\pi_1^{*}(\det \Fj) ^{-1}$.
\end{lem}
\begin{proof}
    Consider the smooth group scheme $p \colon P \ra B$, we have 
    \[\omega_{\pi_1}|_P \cong \omega_{p} \cong \Omega_{P/B}^{g-1} \cong (\pi_1^{*}(\det \Fj) ^{-1})|_P.\]
    Since $\codim(\ol{P} \setminus P) \geq 2$ by Lemma~\ref{Severi inequality}, and the invertible sheaves $\omega_{\pi_1}$ and $\pi_1^{*}(\det \Fj) ^{-1}$ on $\ol{P}$ are Cohen--Macaulay, the conclusion follows.
\end{proof}

For $n\geq 1$, we set
\begin{equation}\label{eq: Arinkin kernel K_n}
    \CK_n \coloneqq Rp_{2,3,\dots,n+1,*}(p_{1,n+1}^{*}\ol{\CG} \ot^{L} p_{2,n+1}^{*}\ol{\CG} \ot^{L} \cdots \ot^{L} p_{n,n+1}^{*}\ol{\CG}) \ot \pi_{n}^{*}(\det \Fj)^{-1}[g-1] \in D^b(\ol{P}^n).
\end{equation}
\begin{prop}\label{prop: kernel is diagonal}
    $(\nu \times \id_{\ol{P}})^{*}\CK_2 \simeq Rp_{23,*}(p_{12}^{*}\ol{\CG}^{\vee} \ot^{L} p_{13}^{*}\ol{\CG}) \ot \pi_2^{*}(\det \Fj)^{-1}[g-1] \simeq \CO_{\Delta}$.
\end{prop}

We postpone the proof of Proposition~\ref{prop: kernel is diagonal} to Section~\ref{subsec: Proof of Proposition 2.2}. As a consequence of Proposition~\ref{prop: kernel is diagonal}, we deduce the following result.
\begin{thm}\label{thm: equivalence}
    Let $D^b(\ol{P})$ be the bounded derived category of coherent sheaves on $\ol{P}$. The integral functor
    \[\FF \colon D^b(\ol{P}) \xrightarrow{\sim} D^b(\ol{P}), \quad \CF \mapsto Rp_{2,*}(p_1^{*}\CF \ot^{L} \ol{\CG})\]
    is an equivalence of categories. Its quasi-inverse is given by
    \[\FF^{-1} \colon D^b(\ol{P}) \xrightarrow{\sim} D^b(\ol{P}), \quad Rp_{1,*}(p_2^{*}\CF \ot^{L} \ol{\CG}^{\vee} \ot^L \pi_2^{*} (\det \Fj) ^{-1}[g-1]).\] 
\end{thm}

\begin{proof}
    By Proposition~\ref{prop: uniqueness, flat}, Lemma~\ref{lem: relative dualizing sheaf of ol p}, and \cite[Theorem 1.1]{Rizzardo15}, $\FF$ is both the left adjoint functor and the right adjoint functor of $\FF^{-1}$. The composition $\FF \circ \FF^{-1}$ has kernel $(\nu \times \id_{\ol{P}})^{*}\CK_2$. Therefore, $\FF^{-1}$ is fully faithful according to Proposition~\ref{prop: kernel is diagonal}. Since the category $D^b(\ol{P})$ is indecomposable and essentially nonzero, we conclude that $\FF^{-1}$, and thus $\FF$, is an equivalence by a criterion of Bridgeland; see \cite[Proposition 1.54]{Huybrechts06}.
\end{proof}

\subsection{The Abel--Prym map}
Suppose there exists an invertible sheaf $\CL$ of degree $-1$ on $\widetilde{\CC}$ over $B$.
Let $\CI_{\Delta}$ be the ideal sheaf of the diagonal of $\widetilde{\CC} \times_B \widetilde{\CC}$. 
By viewing $\CI_{\Delta}^{\vee} \ot p_1^{*}\CL$ as a flat family of degree $0$ line bundles on $p_2 \colon \widetilde{\CC} \times_B \widetilde{\CC} \to \widetilde{\CC}$, we obtain the \emph{Abel--Jacobi map}
$\AJ_{\CL} \colon \widetilde{\CC} \ra \ol{J}_{\widetilde{\CC}}.$

Notice that $\AJ_{\CL}$ factors through $\ol{J}_{\widetilde{\CC}}^{\circ}$, we set $\AJ_{\CL}^{\circ} \colon \widetilde{\CC} \ra \ol{J}_{\widetilde{\CC}}^{\circ}$. Then the \emph{Abel--Prym map} is defined to be the $B$-morphism
\begin{equation}\label{eq: def of AP}
    \AP_{\CL} \coloneqq f \circ \AJ_{\CL}^{\circ} \colon \widetilde{\CC} \ra \ol{P}.
\end{equation}
The pullback of $\AP_{\CL}$ induces a $B$-morphism
\[\AP_{\CL}^{*} \colon \Pic_{(\ol{P}/B)}^0 \ra J_{\widetilde{\CC}}.\] 

\begin{lem}\label{lem: left inverse} 
    Let $\rho \colon P \ra \Pic_{(\ol{P}/B)}^0$ be as in \eqref{eq: def of rho}. Assume that $\CL$ is an invertible sheaf of degree $-1$ on $\widetilde{\CC}$. Then the image of $\AP_{\CL}^{*} \circ \rho$ lies in $P$, and $\AP_{\CL}^{*} \circ \rho = \id_{P}$.
\end{lem}
\begin{proof}
    It suffices to check the equality after a base change, so we may assume that the smooth locus of $\widetilde{\CC}$ over $B$ admits a section. Then there is a normalized universal sheaf $\CU$ on $\widetilde{\CC}\times_B {J}_{\widetilde{\CC}}$.
    
    Let $\beta \colon J_{\widetilde{\CC}} \ra \Pic_{(\ol{J}_{\widetilde{\CC}}/B)}^0$ be the $B$-morphism induced by $\CP$ on $\ol{J}_{\widetilde{\CC}} \times_B J_{\widetilde{\CC}}$.
    Then $\AJ_{\CL}^{*} \circ \beta = \id_{J_{\widetilde{\CC}}}$ (\cite[Proposition 2.2]{Esteves99}) implies that 
    $(\AJ_{\CL} \times \id_{J_{\widetilde{\CC}}})^{*}\CP$ differs from $\CU$ by tensoring with the pullback of an invertible sheaf on $J_{\widetilde{\CC}}$. 
    
    Let $u \colon \ol{J}_{\widetilde{\CC}}^{\circ} \ra \ol{J}_{\widetilde{\CC}}$ and $i \colon P \ra {J}_{\widetilde{\CC}}$ be the embedding morphisms. 
    Then $\AP_{\CL}^{*} \circ \rho \colon P \ra J$ is characterized by the invertible sheaf
    \begin{align*}
        (\AP_{\CL} \times \id_P)^{*}\CG & \cong (\AP_{\CL}^{\circ} \times \id_P)^{*}(u\times \id_P)^{*}  (\id_{\ol{J}_{\widetilde{\CC}}} \times i)^{*}\CP\\
        & \cong (\AJ_{\CL} \times \id_P)^{*} (\id_{\ol{J}_{\widetilde{\CC}}} \times i)^{*}\CP\\
        & \cong (\id_{\widetilde{\CC}} \times i)^{*} (\AJ_{\CL} \times \id_{{J}_{\widetilde{\CC}}})^{*} \CP,
    \end{align*}
    which is an invertible sheaf that differs from $\CU' = \CU|_{\widetilde{\CC}\times_B P}$ by tensoring with the pullback of an invertible sheaf on $P$. The conclusion then follows.    
\end{proof}

The lemma implies that $\AP_{\CL}^{*}$ is independent of the choice of $\CL$ and is just the inverse morphism of $\rho$. Consequently, even without the existence of a degree $-1$ invertible sheaf $\CL$, we can define the morphism $\AP^{*} \colon \Pic_{(\ol{P}/B)}^0 \ra J_{\widetilde{\CC}}$ to be the inverse of $\rho$.

\subsection{Arinkin's dimension bound}
\begin{prop}[{\cite[Section 7.1]{Ari10b}, \cite[Proposition 3.2]{MSY23}}]\label{prop: Arinkin's dimension bound}
Let $\CK_n$ be the sheaf defined in \eqref{eq: Arinkin kernel K_n}.
Then we have 
\begin{equation}\label{eq: Arinkin's dimension bound 1}
    \codim_{\ol{P}^n}(\Supp(\CK_n)) \geq g-1, \quad \forall n \geq 1.
\end{equation}
\end{prop}
\begin{proof}
    By Lemma~\ref{Severi inequality}, it suffices to show that for every closed point $b\in B$ we have 
    \[\codim_{\ol{P}_b^n}(\Supp(\CK_n) \cap \ol{P}_b^n) \geq \widetilde{g}(\CC_b) - 1.\]
    \medskip
    \noindent {\bf Step 1.}
    Let us fix an \'etale double cover $\widetilde{\CC}_b \ra \CC_b$. Let $\widetilde{\CC}_b^{\mathrm{reg}}$ be the regular locus of the curve~$\widetilde{\CC}_b$. 
    For any $F \in \ol{P}_b$, we denote by $\ol{\CG}_F$ the restriction of $\ol{\CG}$ to $\ol{P}_b \times \{F\}$. We further restrict this sheaf to $P_b$ and denote the line bundle by $\CG_F$.
    We first show that if $(F_1, \dots, F_n) \in \ol{P}_b^n$ lies in $\Supp(\CK_n)$, then $\bigotimes_{i=1}^n(F_i|_{\widetilde{\CC}_b^{\mathrm{reg}}}) \cong \CO_{\widetilde{\CC}_b^{\mathrm{reg}}}$. Here we regard a point $F_i \in \ol{P}_b$ as a sheaf on $\widetilde{\CC}_b$. 
    
    By base change, a point $(F_1, \dots, F_n) \in \ol{P}_b^n$ lies in $\Supp(\CK_n)$ if and only if 
    \[\BH^i(\ol{P}_b, \ol{\CG}_{F_1} \ot^{L} \ol{\CG}_{F_2} \ot^{L} \cdots \ot^{L} \ol{\CG}_{F_n}) \neq 0\]
    for some $i$. 
    Let $T_i \ra P_b$ be the $\BG_m$-torsor corresponding to $\CG_{F_i}$. Then $T_i$ is naturally an abelian group by regarding it as a group extension of $P_b$ by $\BG_m$. We have a $T_i$-action on $\ol{P}_b$ induced by $\mu \colon P_b \times \ol{P}_b \ra \ol{P}_b$. By Lemma~\ref{lem: thm of square} (1) we have 
    \[p_1^{*}\CG_{F_i} \ot p_2^{*}\ol{\CG}_{F_i} \cong \mu^{*}\ol{\CG}_{F_i}\]
    on $P_b \times \ol{P}_b$. This implies that we can lift the action of $P_b$ on $\ol{P}_b$ to an action of $T_i$ on $\ol{\CG}_{F_i}$.
    
    Now let $T \ra P_b$ be the $\BG_m$-torsor corresponding to ${\CG}_{F_1} \ot {\CG}_{F_2} \ot \cdots \ot {\CG}_{F_n}$. If there exists~$i \in \BZ$ such that $\BH^i(\ol{P}_b, \ol{\CG}_{F_1} \ot^{L} \ol{\CG}_{F_2} \ot^{L} \cdots \ot^{L} \ol{\CG}_{F_n}) \neq 0$, then the same argument as in the proof of \cite[Proposition~7.2]{Ari10b} shows that $T$ is a trivial torsor. In other words,
    \[\bigotimes_{i=1}^{n} \CG_{F_i} \cong \CO_{P_b}.\]
    Then the conclusion holds by pulling back the above isomorphism via the restriction of the Abel--Prym map 
    \[\AP \colon \widetilde{\CC}_b^{\mathrm{reg}} \ra P_b.\]
    
    \medskip
    \noindent {\bf Step 2.} Consider the action on the first factor $\mu \times \id_{\ol{P}_b^{n-1}} \colon P_b \times \ol{P}_b^n \ra \ol{P}_b^n$. This morphism is flat; thus, it suffices to show that 
    \[Z \coloneqq (\mu \times \id_{\ol{P}_b^{n-1}})^{-1}(\Supp(\CK_n) \cap \ol{P}_b^n)\]
    satisfies $\codim_{P_b \times \ol{P}_b^n}(Z) \geq \widetilde{g}(\CC_b) -1$.
    
    To prove this, we look at the restriction of the projection map $p_{2,3,\dots,n+1} \colon Z \ra \ol{P}_b^n$ and show that each fiber has codimension at least $\widetilde{g}(\CC_b)-1$ in $P_b$. 
    For a fixed point $(F_1, \dots, F_n) \in \ol{P}_b^n$, if $(L, F_1, \dots, F_n) \in Z$, then by Step~1 we have
    \[(L \ot F_1\ot \cdots \ot F_n)|_{\widetilde{\CC}_b^{\mathrm{reg}}} \cong \CO_{\widetilde{\CC}_b^{\mathrm{reg}}}. \]
    Such $L$ form a countable union of subvarieties of $P_b$ of dimension $\delta(b) = g - \widetilde{g}(\CC_b)$. In particular, 
    \[\codim_{P_b}(Z \cap P_b \times \{F_1\} \times \cdots \times \{F_n\}) \geq \widetilde{g}(\CC_b)-1. \qedhere\]
\end{proof}

\subsection{Proof of Proposition~\ref{prop: kernel is diagonal}}\label{subsec: Proof of Proposition 2.2}
The proof follows an argument analogous to the proof of \cite[Theorem C]{Ari10b}. 
\begin{proof}[Proof of Proposition~\ref{prop: kernel is diagonal}]
    Let $\Phi \coloneqq Rp_{23,*}(p_{12}^{*}\ol{\CG}^{\vee} \ot^L p_{13}^{*}\ol{\CG}) \in D^b(\ol{P}\times_B \ol{P})$. The goal is to prove 
    \begin{equation}\label{eq: Phi=Delta}
        \Phi \simeq \CO_{\Delta} \ot \pi_2^{*}\det \Fj [-g+1].
    \end{equation}
    \medskip
    \noindent {\bf Step 1.} We first show that \eqref{eq: Phi=Delta} holds when restricted on $P \times_B P$. 
    Let $\zeta \colon B \ra P$ be the zero section. We define the action $\mu' \colon P\times_B P \ra P, \quad (L,M) \mapsto L^{-1}\ot M$. Consider the following Cartesian diagrams
    \begin{equation*}
        \begin{tikzcd}
            \ol{P} \times_B P \times_B P \arrow[r, "\id_{\ol{P}} \times \mu'"] \arrow[d,"p_{23}"'] & \ol{P} \times_B P \arrow[d,"p_2"] \\
            P \times_B P \arrow[r, "\mu'"] & P \arrow[d, bend right, "p"']\\
            \Delta_P \arrow[u, hook] \arrow[r] & B. \arrow[u, "\zeta"']
        \end{tikzcd}
    \end{equation*}
    We denote by $\CG$ the restriction of $\ol{\CG}$ to the open subset $\ol{P} \times_B P$. Then we have 
     \begin{align*}
        \Phi|_{P\times_B P}
        & \cong Rp_{23,*}(p_{12}^{*}\CG^{-1} \ot p_{13}^{*}\CG)\\
        & \cong Rp_{23,*} (\id_{\ol{P}} \times \mu')^{*}\CG\\
        & \cong \mu'^{*} Rp_{2,*}\CG,
    \end{align*}
    where the second isomorphism follows from Lemmas~\ref{lem: symmetric}, and \ref{lem: thm of square} (2), and the third from flat base change.
    Thus it suffices to prove
    \begin{equation}\label{eq: p2*CG}
        Rp_{2,*} \CG \cong \zeta_{*}(\det \Fj)[-g+1] \cong \zeta_{*}\CO_B \ot p^{*}(\det \Fj)[-g+1],
    \end{equation}
    where the second isomorphism follows from $\det \Fj \cong \zeta^{*}p^{*}\det \Fj$ and the projection formula.

    By Grothendieck duality and the projection formula, we have
    \[(Rp_{2,*}\CG)^{\vee} = R\hom(Rp_{2,*}\CG, \CO_P) \cong Rp_{2,*} R\hom(\CG, Rp_2^{!}\CO_P) \cong Rp_{2,*}\CG^{\vee} \ot p^{*}(\det \Fj)^{-1}[g-1].\]
    Since $p_2$ is proper flat of relative dimension $g-1$, we have $\CH^i(Rp_{2,*}\CG) = 0$ for $i > g-1$ and~$\CH^i((Rp_{2,*}\CG)^{\vee}) = 0$ for $i > 0$. 
    Taking $n=1$ in Proposition~\ref{prop: Arinkin's dimension bound}, we obtain 
    \[\codim_P(\Supp(Rp_{2,*}\CG)) = \codim_P(\Supp(\CK_1|_P)) \geq g-1.\]
    Then $Rp_{2,*}\CG$ is a Cohen--Macaulay sheaf of codimension $(g-1)$ concentrated at degree $(g-1)$, according to \cite[Lemma 7.6]{Ari10b}. 
    By base change, we have a canonical morphism 
    \[\zeta^{*} R^{g-1}p_{2,*}\CG \ra R^{g-1}\pi_{1,*}\CO_{\ol{P}},\]
    where $R^{g-1}\pi_{1,*}\CO_{\ol{P}} \cong \det \Fj$ by Serre duality. 
    By adjunction, we get morphism
    \[R^{g-1}p_{2,*}\CG \ra \zeta_{*}\det \Fj,\]
    which is an isomorphism on $\zeta(B)$ by definition. Thus it suffices to prove that
    \[\Supp(R^{g-1}p_{2,*}\CG)=\zeta(B)\]
    as schemes. 
    
    We first check it set-theoretically. One side of the inclusion is obvious. For the other side, notice that $L \in P_b$ is in $\Supp(R^{g-1}p_{2,*}\CG) \cap P_b$ if and only if $H^{g-1}(\ol{P}_b, \ol{\CG}_L) \neq 0$. By Serre duality, this is equivalent to $H^0(\ol{P}_b, \ol{\CG}_L^{-1}) \neq 0$. Thus $\ol{\CG}_L^{-1}$ admits a subbundle isomorphic to $\CO_{\ol{P}_b}$. Also, the line bundles  $\ol{\CG}_L^{-1} \cong \ol{\CG}_{L^{-1}}$ and $\CO_{\ol{P}_b} \cong \ol{\CG}_{\CO_{\widetilde{\CC}_b}}$ are {algebraically equivalent}. Therefore, we have $\ol{\CG}_{L^{-1}} \cong \ol{\CG}_{\CO_{\widetilde{\CC}_b}}$. 
    By Lemma~\ref{lem: left inverse}, pulling back via the Abel--Prym map shows that $L^{-1}\cong \CO_{\widetilde{\CC}_b}$, i.e., $L \in \zeta(B)$. 

    We conclude the proof by proving that $\Supp(R^{g-1}p_{2,*}\CG) = \zeta(B)$ as schemes. Since the sheaf $R^{g-1}p_{2,*}\CG$ is Cohen--Macaulay, by the unmixed property we only need to check this generically. By restricting to $B^{(g)}$, over which $\CC_b$ are nonsingular, this is a classical result of abelian schemes proved by Mumford \cite[\S 13]{Mumford70}. 
    
    \medskip
    \noindent {\bf Step 2.}
    We first prove that $\Phi$ is Cohen--Macaulay. By Grothendieck duality,
    \[(\Phi)^{\vee} \cong Rp_{23,*}(p_{12}^{*}\ol{\CG} \ot^L p_{13}^{*}\ol{\CG}^{\vee}) \ot \pi_2^{*}\det \Fj ^{-1}[g-1].\]
    Since $p_{23}$ is proper and flat of relative dimension $g-1$, we have $\CH^i(\Phi) =0$ for $i>g-1$ and~$\CH^i(\Phi^{\vee}) = 0$ for $i>0$.     
    Taking $n=2$ in Proposition~\ref{prop: Arinkin's dimension bound}, we obtain 
    \begin{equation}\label{eq: codim of supp Phi}
        \codim_{\ol{P} \times_B \ol{P}}(\Supp (\Phi)) = \codim_{\ol{P} \times_B \ol{P}}(\Supp (\CK_2)) \geq g-1.
    \end{equation}
    Then $\Phi[g-1] \in \Coh(\ol{P} \times_B \ol{P})$ is a Cohen--Macaulay sheaf of codimension $g-1$. 
    
    By Step 1 we have 
    \begin{equation}\label{eq: Supp Phi in P times P}
        \Supp(\Phi) \cap P\times_B P = \Delta \cap P\times_B P
    \end{equation}
    as schemes. Using this, we can strengthen \eqref{eq: codim of supp Phi} to
    \[\codim_{\ol{P} \times_B \ol{P}}(\Supp(\Phi) \setminus \pi_{2}^{-1}(B^{(g)})) > g-1.\]
    Indeed, in the Step 2 of Proposition~\ref{prop: Arinkin's dimension bound}, the projection map $p_{23} \colon Z \ra \ol{P}^2$ has zero-dimensional fibers over $P \times_B P$. Thus for $b \notin B^{(g)}$, we have
    \[\codim (Z) > \widetilde{g}-1.\]
    This is the same as saying that every maximal-dimensional irreducible component of $\Supp(\Phi)$ intersects $\pi_{12}^{-1}(B^{(g)})$. By the Cohen–Macaulay unmixed theorem, $\Supp(\Phi)$ is of equidimension. Thus, every irreducible component of $\Supp(\Phi)$ intersects $\pi_{12}^{-1}(B^{(g)})$.     
    Since curves over $B^{(g)}$ are all nonsingular, by Mumford's result (\cite{Mumford70}) we have
    \[\Supp(\Phi) \cap \pi_{12}^{-1}(B^{(g)}) = \Delta \cap \pi_{12}^{-1}(B^{(g)}).\]
    Taking the closure of both sides, we see that 
    \[\Supp(\Phi) \subset \Delta.\]
    Thus $\Phi$ is a Cohen--Macaulay sheaf on $\Delta$. 

    Since both sides of \eqref{eq: Phi=Delta} are Cohen--Macaulay sheaves on $\Delta$ that are isomorphic over $\Delta_P$, and~$\codim_{\Delta}(\Delta\setminus \Delta_P) \geq 2$, the equality \eqref{eq: Phi=Delta} follows from \cite[Lemma 2.2]{Ari10b}. 
\end{proof}

\section{Autoduality}\label{sec_autoduality}
In this section we prove Theorem~\ref{thm: C}. Again, we follow the notation and assumptions as in Section~\ref{subsec1.1: notations}.

Recall that $\rho\colon P \ra \Pic_{(\ol{P}/B)}^0$ is the morphism defined by the normalized Poincar\'e line bundle $\CG$ on $\ol{P} \times_B P$. By analogy with \cite[Theorem C]{Arinkin10}, we have:
\begin{lem}\label{lem: autoduality for P}
    $\rho \colon P \xrightarrow{\sim} \Pic_{(\ol{P}/B)}^0$.
\end{lem}
\begin{proof}
   To prove that $\rho$ is an isomorphism, we can change the base by an \'etale cover and assume that the smooth locus of $\widetilde{\CC}/B$ admits a section. Then there exists an invertible sheaf~$\CL$ of degree $-1$ on $\widetilde{\CC}/B$, and Lemma~\ref{lem: left inverse} yields that $\rho$ is a right inverse. Therefore, $\rho$ is a closed embedding. Since $P$ is flat over $B$, it suffices to check that $\rho$ is an isomorphism on each geometric fiber. 
   Thus, we are reduced to proving that $\rho_b \colon P_b \xrightarrow{\sim} \Pic^0(\ol{P}_b)$ is an isomorphism for every closed point $b\in B$. 

    We first notice that the differential of $\rho_b$ at $\zeta \coloneqq [\CO_{\widetilde{\CC}_b}] \in P$,
    \[\dd \rho_b \colon \mathrm{T}_{\zeta}P \xrightarrow{\sim} \mathrm{T}_{0} \Pic^0(\ol{P}_b),\]
    is an isomorphism of $k$-vector spaces, as we now explain. 
    In fact, the injectivity follows from Lemma~\ref{lem: left inverse}. Then the isomorphism follows, since 
    \[\mathrm{T}_{\zeta}P_b \cong \Ext_P^1(\CO_{\zeta},\CO_{\zeta}) \cong \Ext_{\ol{P}}^1(\CO_{\ol{P}_b}, \CO_{\ol{P}_b}) \cong H^1(\ol{P}_b, \CO_{\ol{P}_b}) \cong \mathrm{T}_{0} \Pic^0(\ol{P}_b)\]
    shows that both sides are $(g-1)$-dimensional $k$-vector spaces.
    Here the second isomorphism follows from the fact that $\FF$ is fully faithful. 

    Then the very same argument as in the proof of \cite[Theorem C]{Arinkin10} shows that $\rho_b$ is surjective. We conclude that $\rho_b$ is an isomorphism since the $k$-group scheme $\Pic^0(\ol{P}_b)$ is reduced. 
\end{proof}

From the proof we see that $\rho$ is \'etale locally a homomorphism of group schemes. Hence $\rho$ is an isomorphism of $B$-group schemes.

\begin{thm}\label{thm: autoduality}
    The morphism $\ol{\rho} \colon \ol{P} \xrightarrow{\sim} \ol{\Pic}_{(\ol{P}/B)}^0$ is an isomorphism of $B$-schemes. 
\end{thm}
The theorem obviously holds over $B^{(g)}$. Indeed, the relative Prym variety $P_{B^{(g)}} \ra B^{(g)}$ has a canonically defined principal polarization $\varphi_{\Xi} \colon P_{B^{(g)}} \xrightarrow{\sim} \Pic_{(P_{B^{(g)}}/B^{(g)})}^0$, which equals $\ol{\rho}|_{B^{(g)}}$ as we have seen in the proof of Theorem~\ref{thm: def G}.
\begin{proof}
    As before, we may change the base by an \'etale cover, so that the smooth locus of $\widetilde{\CC}/B$ admits a section. 
    Then there exists an invertible sheaf $\CL$ of degree $-1$ on $\widetilde{\CC}/B$. 

    To simplify notation, set $U = \Pic_{(\ol{P}/B)}^0$ and $\ol{U} = \ol{\Pic}_{(\ol{P}/B)}^0$. 
    Let $\CI$ be the universal sheaf on $\ol{P}\times_B \ol{U}$. Then using Lemma~\ref{lem: autoduality for P}, the same argument as in the proof of \cite[Theorem 2.6]{Esteves04} shows that~$(\AP_{\CL} \times \id_{\ol{U}})^{*} \CI$ is a relative torsion-free rank 1 sheaf on $\widetilde{\CC} \times_B \ol{U}$ over $\ol{U}$. 
    Consequently, we can extend $\AP_{\CL}^{*}$ to 
    \[\ol{\AP}_{\CL}^{*} \colon \ol{\Pic}_{(\ol{P}/B)}^0 \ra \ol{P}.\]

    It follows from $\AP_{\CL}^{*} \circ \rho = \id_{P}$ and the separatedness of $P$ that $\ol{\AP}_{\CL}^{*} \circ \ol{\rho} = \id_{\ol{P}}$. For the same reason, $\ol{\rho} \circ \ol{\AP}_{\CL}^{*} = \id_{\ol{U}}$. Hence $\ol{\rho}$ is an isomorphism. 
\end{proof}

\begin{rmk}
    Fiberwise, the inverse map of $\ol{\rho}$ can be written explicitly as 
    \[\ol{\Pic}^0(\ol{P}_b) \ra \ol{P}_b, \quad M \mapsto \nu(\Supp (\FF(M))).\]
    Indeed, for any $M\in \ol{\Pic}^0(\ol{P}_b)$, there exists $F\in \ol{P}_b$ such that $M = \ol{\rho}(F)$. 
    Since 
    \[\ol{\rho}(F) \cong \ol{\CG}_F \cong \FF^{-1}(\CO_{\nu(F)})[-g+1],\]
    we obtain $F \cong \nu(\Supp (\FF(M)))$. 
    The proof above shows that this map is algebraic.
\end{rmk}

\section{Applications}\label{sec_application}
In this section, we apply our main results to give several examples of dualizable abelian fibrations that satisfy (FV) defined by Maulik--Shen--Yin in \cite{MSY23}. 

\subsection{Dualizable abelian fibrations}\label{sec_app_definition}
A morphism $\pi\colon M \ra B$ is an abelian fibration if the varieties $M$ and $B$ are nonsingular and irreducible, $\pi$ is proper with equidimensional fibers, and~$M$ contains an open subset $P$ that is a smooth commutative $B$-group scheme whose restriction to some open subset $\pi \colon P_U \ra U \subset B$ is an abelian scheme. 

\begin{defn}(\cite[Section~1.4]{MSY23})\label{def: dualizable abelian fibration}
    An abelian fibration $\pi\colon M \ra B$ of relative dimension $n$ is called a \emph{dualizable abelian fibration} if it satisfies the following conditions:
\begin{enumerate}
    \item[(a)] There exists an abelian fibration $\pi^{\vee}\colon M^{\vee} \ra B$ such that $\pi^{\vee}_U$ is the dual abelian scheme of $\pi_U$ over the open subset $U\subset B$;
    \item[(b)] The morphism $\pi$ has full support;
    \item[(c)] There exists a complex $\CP\in D^b\Coh(M^{\vee}\times_B M)$ which extends the \emph{normalized} Poincar\'e line bundle on $M_U^{\vee} \times_U M_U$, and which admits an inverse $\CP^{-1}\in D^b\Coh(M\times_B M^{\vee})$ such that 
    \[\CP^{-1} \circ \CP \simeq \CO_{\Delta_{M^{\vee}/B}},\quad \CP \circ \CP^{-1} \simeq \CO_{\Delta_{M/B}};\]
    \item[(d)] The convolution kernel 
    \[\CK = \CP^{-1} \circ \CO_{\Delta_{M/B}^{\mathrm{small}}} \circ (\CP \boxtimes \CP) \in D^b\Coh(M^{\vee} \times_B M^{\vee} \times_B M^{\vee})\]
    is supported in codimension $n$.
\end{enumerate}
\end{defn}

To give the definition of Fourier vanishing condition, we recall the definition of Chern characters of coherent sheaves on the singular variety $M^{\vee}\times_B M$ \cite[Chapter~18]{Fulton}. 
We fix the closed embedding $i \colon M^{\vee}\times_B M \to M^{\vee}\times M$. We denote by $K_{M^{\vee}\times_B M}(M^{\vee}\times M)$ the Grothendieck group of perfect complexes on $M^{\vee}\times M$ which are supported on the closed subscheme $M^{\vee}\times_B M$, and by $\Chow^k(M^{\vee}\times_B M \to M^{\vee}\times M)$ the degree $k$ bivariant Chow group of $i$. Let $\ch_{M^{\vee}\times_B M}^{M^{\vee}\times M}(-)$ be the localized Chern character operation \cite[Definition 18.1]{Fulton}. Consider the diagram
\begin{equation}\label{eq: diagram of Chern}
    \begin{tikzcd}
    K_{M^{\vee}\times_B M}(M^{\vee}\times M) \arrow[r, "\cap \CO_{M^{\vee}\times M}"] \arrow[d, "{\ch_{M^{\vee}\times_B M,k}^{M^{\vee}\times M}}"'] & K_{*}(M^{\vee}\times_B M) \arrow[d, "\widetilde{\ch}_k"] \\
    \Chow^k(M^{\vee}\times_B M \to M^{\vee}\times M) \arrow[r, "{\cap [M^{\vee}\times M]}"]                                 & \Chow_{2\dim M-k}(M^{\vee}\times_B M),                   
    \end{tikzcd}
\end{equation}
where the horizontal maps are canonical isomorphisms and $\widetilde{\ch}_k$ is defined to make the diagram commute.

Consider the Todd-twisted Chern character
\[\tau \colon K_{*}(M^{\vee} \times_B M) \to \Chow_{*}(M^{\vee} \times_B M), \quad F \mapsto \td(i^{*}T_{M^{\vee}\times M}) \cap \widetilde{\ch}(F)\]
defined in \cite[Theorem~18.3]{Fulton}. 
It follows from \cite[Section 2.4]{MSY23} that the Poincar\'e complex $\CP$ induces the Chow-theoretic Fourier transforms 
\begin{equation}\label{eq: def Chow Fourier}
    \FF \coloneq \td(-T_{M^{\vee} \times_B M}) \cap \tau(\CP) = \sum_i \FF_i, \quad \FF_i \in \Corr_B^{i-n}(M^{\vee},M),
\end{equation}
where $T_{M^{\vee} \times_B M}$ is the virtual tangent bundle of $M^{\vee} \times_B M$.
Similarly, $\CP^{-1}$ induces
\[\FF^{-1} \coloneq  \td(-(\pi \times_B \pi^{\vee})^{*}T_B) \cap \tau(\CP^{-1}) = \sum_i \FF_i^{-1}, \quad \FF_i^{-1} \in \Corr_B^{i-n}(M,M^{\vee}).\]

A dualizable abelian fibration~$\pi\colon M \ra B$ of relative dimension $n$ is said to satisfy the \emph{Fourier vanishing condition} (FV) if 
\begin{equation}\label{eq: def of FV}
    \FF_i^{-1} \circ \FF_j = 0 \in \Corr_B^{i+j-2n}(M^{\vee},M^{\vee}), \quad i+j < 2n.
\end{equation}

\subsection{Relative compactified Prym varieties}
We adapt the same notation and assumptions as in Section~\ref{subsec1.1: notations}. In particular, $\CC \ra B$ is projective and $\ol{P}$ is nonsingular.

\begin{prop}\label{prop: full support}
    The morphism $\pi: \ol{P} \ra B$ is of full support.
\end{prop}
\begin{proof}
    We first check by definition \cite[7.1.1]{Ngo10} that $(P,\ol{P},B)$ forms a $\delta$-regular abelian fibration of relative dimension $g-1$. The condition \cite[7.1.2]{Ngo10} holds since both $P \ra B$ and $\ol{P} \ra B$ are flat of relative dimension $g-1$. 
    Then the affine stabilizer condition (\cite[7.1.3]{Ngo10}) follows from Proposition~\ref{stratification of Prym}. 
    Since $P \ra B$ is a quasi-projective group scheme,
    the Tate module is polarizable (\cite[7.1.4]{Ngo10}) follows from \cite[Theorem 1.2]{Ancona24}. (Alternatively, this is also proved in \cite[Lemma 9.7]{ACLS23}.) 
    That the group scheme $P \to B$ is $\delta$-regular (\cite[7.1.5]{Ngo10}) follows from Lemma~\ref{Severi inequality}.
    Since $\ol{P}$ is nonsingular and the fibers of $\pi$ are irreducible, Ng\^{o}'s Support Theorem (\cite[Theorem 7.2.1]{Ngo10}) implies that $\pi$ is of full support.
\end{proof}

\begin{thm}\label{thm: Prym is dualizable}
    The relative compactified Prym $\ol{P} \ra B$ is a dualizable abelian fibration of relative dimension $g-1$ that satisfies (FV).
\end{thm}
\begin{proof}
    We first check the four conditions in Definition~\ref{def: dualizable abelian fibration}. 
    The relative compactified Prym is a principally polarized abelian scheme over $B^{(g)}$; thus, it is self-dual and Condition (a) holds. 
    Condition~(b) holds by Proposition~\ref{prop: full support}.
    By construction, the restriction of the sheaf~$\ol{\CG}$ to~$\ol{P} \times_{B^{(g)}} \ol{P}$ is the normalized Poincar\'e line bundle. Set $\ol{\CG}^{-1}\coloneq \ol{\CG}^{\vee} \ot \pi_2^{*}(\det \Fj)^{-1}[g-1]$. Then Condition~(c) follows from Theorem~\ref{thm: equivalence}. 
    Condition (d) follows from Arinkin's dimension bound (Proposition~\ref{prop: Arinkin's dimension bound}) since $\CK \cong (\id_{\ol{P} \times_{B} \ol{P}} \times \nu)^{*} \CK_3$.

    It remains to verify (FV); this follows from the same argument as in \cite[Section~3.5]{MSY23}. For the reader's convenience, we include a sketch of the proof. Throughout, all tensor and exterior products are taken in the derived sense.
    
    \medskip
    \noindent {\bf Step 1.} We first prove a variant of Arinkin's dimension bound \eqref{eq: Arinkin's dimension bound 1}.
    Consider the closed embedding
    \[i \colon \ol{P} \times_B \ol{P} \to \ol{P} \times \ol{P}.\]
    For any positive integer $N$, set $\widetilde{\CK}(N) \coloneq \ol{\CG}^{-1} \circ (i_{*}\ol{\CG})^{\ot N} \in D^b(\ol{P} \times \ol{P})$. Now we prove that~$\widetilde{\CK}(N)$ is supported on $\ol{P} \times_B \ol{P}$ and 
    \begin{equation}\label{eq: Arinkin dimension bound 2}
        \codim_{\ol{P} \times_B \ol{P}}(\Supp(\widetilde{\CK}(N))) \geq g-1.
    \end{equation}

    By the flat base change and the projection formula, the object $\widetilde{\CK}(N)$ is isomorphic to the pushforward of an object in $D^b(\ol{P}\times_B \ol{P}\times_B \ol{P})$ via the composition map
    \[\ol{P}\times_B \ol{P}\times_B \ol{P} \xhookrightarrow{i \times id_{\ol{P}}} \ol{P}\times \ol{P}\times_B \ol{P} \xrightarrow{q_{13}} \ol{P}\times \ol{P}.\]
    This map equals the composition map
    \[\ol{P}\times_B \ol{P}\times_B \ol{P} \xrightarrow{p_{13}} \ol{P}\times_B \ol{P} \xhookrightarrow{i} \ol{P}\times \ol{P}.\]
    Thus, $\widetilde{\CK}(N)$ is supported on $\ol{P}\times_B \ol{P}$.

    To prove \eqref{eq: Arinkin dimension bound 2}, it suffices to show that for every closed point $b \in B$ we have
    \begin{equation}\label{eq: dimension bound 3}
        \codim_{\ol{P}_b^2}(\Supp(\widetilde{\CK}(N)) \cap \ol{P}_b^2) \geq \widetilde{g}(\CC_b) - 1.
    \end{equation} 
    By base change, a closed point $(F_1, F_2) \in \ol{P}_b^2$ lies in $\Supp(\widetilde{\CK}(N))$ if and only if 
    \begin{equation}\label{eq: hypercoho nonzero}
        \BH^i(\ol{P}_b, (R\iota^{*}\iota_{*}\ol{\CG}_{F_1})^{\ot N} \ot \ol{\CG}_{F_2}^{\vee}) \neq 0
    \end{equation}
    for some $i$, where $\iota\colon \ol{P_b} \to \ol{P}$ is the closed embedding. 
    Since $\iota$ is a regular embedding, the same argument as in \cite[Step~3 of Proposition~3.3]{MSY23} shows that 
    the increasing filtration on~$R\iota^{*}\iota_{*}\ol{\CG}_{F_1}$ induced by the standard truncation has graded pieces given by finite direct sums of copies of~$\ol{\CG}_{F_1}$. This induces a bounded filtration of the object $(R\iota^{*}\iota_{*}\ol{\CG}_{F_1})^{\ot N} \ot \ol{\CG}_{F_2}^{\vee}$. 
    Each term of the~$E_1$ page of the hypercohomology spectral sequence is given by finite direct sums of copies of 
    \begin{equation}\label{eq: hypercohom nonzero 2}
        \BH^{*}(\ol{P}_b, \ol{\CG}_{F_1}^{\ot N} \ot \ol{\CG}_{F_2}^{\vee}).
    \end{equation}
    In particular, \eqref{eq: hypercoho nonzero} implies that there exists nonzero hypercohomology of the type \eqref{eq: hypercohom nonzero 2}. 
    We can now conclude \eqref{eq: dimension bound 3} from Proposition~\ref{prop: Arinkin's dimension bound}. 

    \medskip
    \noindent {\bf Step 2.} It remains to use Adams operators together with \eqref{eq: Arinkin dimension bound 2} to show that (FV) holds. 
    Recall that the Adams operators 
    \[\psi^N \colon K_{\ol{P} \times_B \ol{P}}(\ol{P} \times\ol{P}) \ra K_{\ol{P} \times_B \ol{P}}(\ol{P} \times\ol{P}), \quad N \geq 1\]
    are constructed by Gillet--Soul\'e \cite[Section~4]{Gillet87_Adams} using the induction formula
    \[\psi^N - \psi^{N-1} \ot \lambda^1 + \cdots + (-1)^{N-1}\psi^1 \ot \lambda^{N-1} + (-1)^{N}N \ot \lambda^{N} = 0,\]
    where $\lambda^k$ is the $k$-th exterior power operation. 
    By the canonical map in the first line of \eqref{eq: diagram of Chern}, they can be regarded as operators on $K_{*}(\ol{P} \times_B \ol{P})$.

    Then the class $\ol{\CG}^{\vee} \circ \psi^N(\ol{\CG}) \in K_{*}(\ol{P} \times_B \ol{P})$ is a linear combination of 
    \[\widetilde{\CK}^{\lambda}(N_1, N_2, \dots, N_k) \coloneq \ol{\CG}^{\vee} \circ \bigotimes_{i=1}^k(\wedge^{N_i}i_{*}\ol{\CG}), \quad \sum_{i=1}^k N_i = N.\]
    Each object above is a direct summand of $\widetilde{\CK}(N)$ and hence is supported on a codimension $g-1$ subset of $\ol{P} \times_B \ol{P}$, as explained in the argument before \cite[Corollary~3.4]{MSY23}. Therefore, $\ol{\CG}^{\vee} \circ \psi^N(\ol{\CG})$ is also supported on a codimension $g-1$ subset of $\ol{P} \times_B \ol{P}$.
    As explained in \cite[Section~3.5.3]{MSY23}, this yields that $\FF^{-1} \circ \widetilde{\ch}(\psi^N(\ol{\CG}))$ vanishes in codimension $< g-1$ for all $N \geq 1$.
    Using the compatibility of Chern characters and Adams operators, namely $\widetilde{\ch}_k(\psi^N(-)) = N^k\widetilde{\ch}_k(-)$ (see \cite[Theorem~3.1]{Kurano00_Adams}), we deduce that
    \[\FF_i^{-1} \circ \widetilde{\ch}_{j+\dim B}(\ol{\CG})=0, \quad i+j < 2(g -1).\]
    Finally, it follows from the definition of $\FF_j$ that 
    \[\FF_i^{-1} \circ \FF_j = \sum_{j'+j'' = j} \td_{j'}(\pi_2^{*}T_B) \cap (\FF_i^{-1} \circ \widetilde{\ch}_{j''+\dim B}(\ol{\CG}))=0, \quad i+j < 2(g-1) \qedhere\]
\end{proof}

\subsection{LSV fibrations}
Let $X \subset \BP_{\BC}^5$ be a general cubic 4-fold. 
We denote by $\CY \ra B \coloneqq (\BP^5)^{\vee}$ the universal family of cubic 3-folds obtained as hyperplane sections of $X$. 
Let $U\subset B$ (resp.~$U_1\subset B$) be the open set that parameterizes hyperplane sections that are nonsingular (resp.~that have a single ordinary node). Denote by $\pi_U \colon \CJ_U \ra U$ the relative intermediate Jacobian fibration associated to $\CY_U \to U$.  
According to \cite[\S 8.5.2]{DM96b}, this fibration admits a flat projective extension $\pi_{U_1} \colon \CJ_{U_1} \ra U_1$, as we now describe. The relative Picard scheme of the relative Fano surface of lines {is a semi-abelian scheme} $\CJ_{U_1}^{\circ} \ra U_1$. For a point $t \in U_1 \setminus U$ in the boundary, the fiber of $\CJ_{U_1}^{\circ}$ is a $\BG_m$-bundle over $J(\widetilde{Y_t})$, where $\widetilde{Y_t}$ is the blow-up of $Y_t$ at the node. Then $\CJ_{U_1}$ is defined to be the Mumford compactification of $\CJ_{U_1}^{\circ}$; its fiber over~$t \in U_1 \setminus U$ is a $\BP^1$-bundle over $J(\widetilde{Y_t})$. 

In \cite{LSV17}, Laza, Sacc\`a, and Voisin constructed a smooth projective compactification $\ol{\CJ}$ of $\CY_{U_1}$. We briefly recall the construction here. For a cubic 3-fold $Y$ and a line $l \subset Y$, let $C_l$ be the curve of the discriminant locus of the conic bundle $\mathrm{Bl}_lY \ra \BP^2$, and let $\widetilde{C}_l$ be the curve of lines in $Y$ that intersect $l$. 
A line $l \subset Y$ is called \emph{very good} if $\widetilde{C_l} \to C_l$ is \'etale and $\widetilde{C}_l$ is irreducible. 
Let $\CF$ be the relative Fano variety of lines of $\CY \to B$. 
We write $\CF^{\circ} \subset \CF$ for the locus of very good lines. 
Then there is a projective family of integral plane curves $C_{\CF^{\circ}} \ra \CF^{\circ}$ together with an \'etale double cover $\widetilde{C}_{\CF^{\circ}} \ra C_{\CF^{\circ}}$ such that the fibers of $\widetilde{C}_{\CF^{\circ}} \to \CF^{\circ}$ are integral curves. 

For general cubic 4-fold $X$, $\CF^{\circ} \ra B$ is smooth and surjective, and the relative compactified Prym variety $\ol{P}_{\CF^{\circ}}$ is nonsingular. The compactification $\ol{\CJ}$ is then constructed by descending~$\ol{P}_{\CF^{\circ}}$ along $\CF^{\circ} \to B$. 

We set $\CF_U^{\circ} = \CF^{\circ} \times_B U$. Let $\widetilde{\CC}_{\CF_U^{\circ}}$ (resp.~$\CC_{\CF_U^{\circ}}$) be the restriction of the family of curves~$\widetilde{\CC}_{\CF^{\circ}}$ (resp.~$\CC_{\CF^{\circ}}$). Then $\Prym(\widetilde{\CC}_{\CF_U^{\circ}}/\CC_{\CF_U^{\circ}}) \ra \CF_U^{\circ}$ is a principally polarized abelian scheme with a canonical principal polarization $\varphi_{\Xi}$. 
In \cite[Section 13]{CG72}, Clemens and Griffiths constructed the theta divisor $\Theta$ on the intermediate Jacobian, which induces a canonically defined principal polarization $\varphi_{\Theta}$ of the abelian scheme $\CJ_U \ra U$. 
Let $\CJ_{\CF_U^{\circ}} = \CJ_U \times_U \CF_U^{\circ}$, and let $q_1 \colon \CJ_{\CF_U^{\circ}} \ra \CJ_U$ be the first projection. Then $\varphi_{q_1^{*}\Theta}$ is a principal polarization of the abelian scheme $\CJ_{\CF_U^{\circ}} \ra \CF_U^{\circ}$. 
Mumford \cite{Mumford74} proved that the relative Abel--Jacobi map induces an isomorphism of principally polarized abelian schemes
\begin{equation}\label{eq: AJ map between Prym and intermediate Jacobian}
    \left( \Prym(\widetilde{\CC}_{\CF_U^{\circ}}/\CC_{\CF_U^{\circ}}) , \varphi_{\Xi} \right) \xrightarrow{\sim}  \left( \CJ_{\CF_U^{\circ}}, \varphi_{q_1^{*}\Theta} \right).
\end{equation}

\begin{prop}\label{prop: Poincare sheaf on OG10}
    Let $q \colon \ol{P}_{\CF^{\circ}} \times_{\CF^{\circ}} \ol{P}_{\CF^{\circ}} \ra \ol{\CJ} \times_B \ol{\CJ}$ be the morphism induced by the projection map $\ol{P}_{\CF^{\circ}} \cong \ol{\CJ}\times_B \CF^{\circ} \ra \ol{\CJ}$. Let $\ol{\CG}$ be the normalized Poincar\'e sheaf on $\ol{P}_{\CF^{\circ}} \times_{\CF^{\circ}} \ol{P}_{\CF^{\circ}}$. Then there exists a maximal Cohen--Macaulay sheaf $\ol{\CQ}$ on $\ol{\CJ} \times_B \ol{\CJ}$ such that 
    \[\ol{\CG} \cong q^{*}\ol{\CQ}.\]
\end{prop}
\begin{proof}
    By abuse of notation, we denote by $q \colon \ol{P}_{\CF_{U_1}^{\circ}} \times_{\CF_{U_1}^{\circ}} {P}_{\CF_{U_1}^{\circ}} \ra \CJ_{U_1} \times_{U_1} \CJ_{U_1}^{\circ}$ the natural morphism and by $\CG$ the restriction $\ol{\CG}|_{\ol{P}_{\CF_{U_1}^{\circ}} \times_{\CF_{U_1}^{\circ}} {P}_{\CF_{U_1}^{\circ}}}$.

    Recall that over $U$ we have an isomorphism of abelian group schemes 
    \[\CJ_{U_1}^{\circ}|_U = \CJ_U \xrightarrow{\varphi_{\Theta}}  {\Pic}_{(\CJ_{U}/U)}^0 = {\Pic}_{(\CJ_{U_1}/U_1)}^0|_U. \]
    Note that $U$ is an open dense subset of $U_1$, which is a Noetherian normal scheme. Moreover, both ${\Pic}_{(\CJ_{U_1}/U_1)}^0$ and $\CJ_{U_1}^{\circ}$ are semi-abelian group schemes over $U_1$.
    It follows from \cite[Proposition~2.7]{FC90} that $\varphi_{\Theta}$ extends uniquely to an isomorphism of semi-abelian group schemes 
    \[\phi\colon \CJ_{U_1}^{\circ} \xrightarrow{\sim} \Pic_{(\CJ_{U_1}/U_1)}^0.\]   
    Let $\CI$ be the universal sheaf on $\CJ_{U_1} \times_{U_1} {\Pic}_{(\CJ_{U_1}/U_1)}^0$ that is normalized along the zero section of $\pi_{U_1}$. 
    We set $\CQ= (\id_{\CJ_{U_1}} \times \phi)^{*}\CI$, which is a sheaf on $\CJ_{U_1} \times_{U_1} \CJ_{U_1}^{\circ}$. 
    
    Since $\Pic_{\CJ_{U_1}/U_1}^0$ is compatible with flat base change, we have a homomorphism over $\CF_{U_1}^{\circ}$,
    \[P_{\CF_{U_1}^{\circ}} \cong \CJ_{U_1}^{\circ} \times_{U_1} \CF_{U_1}^{\circ}  \xrightarrow{\phi \times \id_{\CF_{U_1}^{\circ}}} \Pic_{(\CJ_{U_1}/U_1)}^0 \times_{U_1} \CF_{U_1}^{\circ} \cong \Pic_{(\ol{P}_{\CF_{U_1}^{\circ}}/\CF_{U_1}^{\circ})}^0,\]
    which corresponds to the line bundle $q^{*}Q$ on $\ol{P}_{\CF_{U_1}^{\circ}} \times_{\CF_{U_1}^{\circ}} {P}_{\CF_{U_1}^{\circ}}$. By \eqref{eq: AJ map between Prym and intermediate Jacobian}, $q^{*}Q$ restricted to~$P_{\CF_U^{\circ}} \times_{\CF_U^{\circ}} P_{\CF_U^{\circ}}$ coincides with the normalized Poincar\'e line bundle $\CG$ associated to $\varphi_{\Xi}$. 
    Therefore
    \begin{equation}
        \phi \times \id_{\CF_{U_1}^{\circ}}|_{\pi_2^{-1}(\CF_U^{\circ})} = \varphi_{\Xi} = \rho|_{\pi_2^{-1}(\CF_U^{\circ})},
    \end{equation}
    where $\rho$ is the morphism defined by $\CG$, and $\pi_2 \colon \ol{P}_{\CF^{\circ}} \times_{\CF^{\circ}} \ol{P}_{\CF^{\circ}} \ra \CF^{\circ}$ is the natural projection. 
    According to Lemma~\ref{lem: autoduality for P}, $\rho$ is also a group scheme homomorphism. 
    Since $\CF_{U_1}^{\circ}$ is a Noetherian normal scheme, the uniqueness of extensions of homomorphisms between semi-abelian group schemes (see \cite[Proposition 2.7]{FC90}) implies that
    \[\phi\times\id_{\CF_{U_1}^{\circ}} = \rho,\]
    and hence the normalized sheaves $q^{*}\CQ$ and $\CG$ are isomorphic. 
    
    Finally, let $j \colon \CJ_{U_1} \times_{U_1} \CJ_{U_1}^{\circ} \hookrightarrow \ol{\CJ} \times_B \ol{\CJ}$ be the open embedding. Set $\ol{\CQ} = j_{*}\CQ$. 
    Noting that the complement of $\CJ_{U_1} \times_{U_1} \CJ_{U_1}^{\circ}$ in $\ol{\CJ} \times_B \ol{\CJ}$ has codimension at least 2, and that $\ol{\CG}$ is a maximal Cohen--Macaulay sheaf, we conclude from \cite[Lemma 4.14]{MRV19b} that $\ol{\CQ}$ is the desired sheaf. 
\end{proof}

\begin{thm}\label{thm: LSV is dualizable}
    The LSV fibration $\pi\colon \ol{\CJ} \ra B$ is a dualizable abelian fibration of relative dimension $g-1$ that satisfies (FV).
\end{thm}
\begin{proof}
    We verify the four conditions in Definition~\ref{def: dualizable abelian fibration} as well as (FV).
    Since $\ol{\CJ}\ra B$ is a self-dual abelian fibration of relative dimension $g-1$, Condition (a) holds. 
    Condition (b) follows from \cite[Proposition 9.5]{ACLS23} or \cite[Corollary 1.3]{Ancona24}. 
    For Condition (c), first notice that the restriction of $\ol{Q}$ to~$\ol{\CJ} \times_U \ol{\CJ}$ is the normalized Poincar\'e line bundle by construction. 
    Then the Fourier--Mukai transform induced by $\ol{\CQ}$ is an autoequivalence of $D^b(\ol{\CJ})$ by Theorem~\ref{thm: equivalence} and \cite[Theorem~2.4]{HLS07}. 
    Finally, since $\ol{\CQ}$ descends from $\ol{\CG}$, Arinkin's dimension bound inequalities~\eqref{eq: Arinkin's dimension bound 1} and \eqref{eq: Arinkin dimension bound 2} for~$\ol{\CG}$ also hold for $\ol{\CQ}$. Then the same argument as in the proof of Theorem~\ref{thm: Prym is dualizable} shows that Condition~(d) and (FV) hold.
\end{proof}

\begin{proof}[Proof of Corollary~\ref{cor: motivic multiplicativity}]
    Let $\pi \colon M \to B$ be either the compactified Prym fibration as in Theorem~\ref{thm: compactified Prym is dualizable abelian fibration s.t. FV} or an LSV fibration. By Theorems~\ref{thm: Prym is dualizable} and \ref{thm: LSV is dualizable}, it is a dualizable abelian fibration satisfying (FV). Then (1) follows from \cite[Corollary~2.7]{MSY23}, and (2) follows from \cite[Theorem~2.6]{MSY23}. 
    
    The case where $\pi \colon M \to B$ is a twisted LSV fibration will be
    treated after the proof of Theorem~\ref{thm: equivalence for twisted derived category}.
\end{proof}

\subsection{Twisted Poincar\'e sheaves on twisted LSV fibrations}
For simplicity, from now on let $\pi \colon X \ra B$ be an LSV fibration associated to a general cubic 4-fold, and let $\CP$ denote the Poincar\'e sheaf on $X\times_B X$ constructed in Proposition~\ref{prop: Poincare sheaf on OG10}. Recall that $\pi$ has integral fibers. 

According to \cite[Theorem 2]{AF16}, there exists a canonically defined smooth commutative group scheme $\pi_P \colon P \ra B$ together with an action 
\[\mu \colon P \times_B X \ra X,\]
such that the smooth locus $X^{\mathrm{sm}}$ of $\pi$ is a $P$-torsor. 
The group scheme $P \ra B$ is called the \emph{relative Albanese fibration} of $X \ra B$. 

Sacc\`a proved that $P$ admits a smooth projective hyper-K\"ahler compactification $\ol{P}$, and the two Lagrangian fibrations $\ol{P} \ra B$ and $X \ra B$ are isomorphic in the \'etale topology (\cite[Theorem 6.13]{Sacca24_compactifying}). 
Since $X^{\mathrm{sm}}$ is
a $P$-torsor, it defines a class in $\Sha_{\et} \coloneqq H^1_{\et}(B,P)$, and the fibration $X \to B$ is the corresponding twist of $\ol{P} \to B$. 
We call $\Sha_{\et}$ the \emph{\'etale Shafarevich--Tate group} to distinguish it from the Shafarevich--Tate group $\Sha \coloneq H^1(B,\ul{\Aut}^0_{(X/B)})$ defined in the analytic topology in \cite{AbashevaRogov21}. 

Since $\pi$ has integral fibers, the group $\Sha_{\et}$ is torsion \cite[Theorem 6.29, Proposition 6.32]{Kim25}, and it is canonically isomorphic to the torsion subgroup of $\Sha$ \cite[Proposition 6.34]{Kim25}. 
For any~$t\in \Sha_{\et}$ we denote by $X^t$ the twist of $X$ by the class $t$. Then $X^t$ is a smooth projective hyper-K\"ahler variety \cite[Theorem A]{Abasheva24}, and the induced map $\pi^t \colon X^t \to B$ is a Lagrangian fibration \cite[Corollary~3.7]{AbashevaRogov21}. 
The relative Albanese fibration of $\pi^t$ is still $P \to B$. Moreover, the induced fibration $X^t \to B$ still has full support by Ng\^{o}'s Support Theorem; see, for example, \cite[Remark 9.6]{ACLS23}. 
\begin{exmp}
    Let $\pi^T \colon X^T \ra B$ be the twisted LSV constructed by Voisin~\cite{VoisinTwisted17}. 
    By construction, Lagrangian fibrations $\pi\colon X \to B$ and $\pi^T\colon X^T \to B$ have the same relative Albanese fibration. Therefore, $X^T \ra B$ is an \'etale Shafarevich--Tate twist of $X \ra B$. 
\end{exmp}

For an \'etale Shafarevich--Tate twist $X^t$, we can construct a twisted Poincar\'e sheaf that induces an equivalence of derived categories. The proof is parallel to \cite[Theorem~3.3]{Bottini24} and the main input is the theorem of the square Lemma~\ref{lem: thm of square}. 

Recall that if $\SX \to X$ is a $\mu_n$-gerbe, then we denote by $\Coh(\SX)_{(k)}$ (resp.~$D^b(\SX)_{(k)}$) the isotypic component of $\Coh(\SX)$ (resp.~$D^b(\SX)$) corresponding to the character $\lambda \mapsto \lambda^k$ of $\mu_n$.

\begin{thm}\label{thm: equivalence for twisted derived category}
    Let $t \in \Sha_{\et}$ be $n$-torsion. Then there exists a $\mu_n$-gerbe $\sigma\colon \SX \to X$ and a Cohen--Macaulay sheaf $\CP^t \in \Coh(X^t \times_B \SX)_{(1)}$ such that 
    \begin{enumerate}
        \item There exists an \'etale cover $\{U_i\}_{i\in I}$ of $B$ and a $U_i$-isomorphism $\rho_i\colon X_{U_i}^t \xrightarrow{\sim} X_{U_i}$ for each~$i$, such that over $U_i$ we have
        \[\CP^t_{U_i} \cong (\rho_i \times_{U_i} \sigma)^* \CP_{U_i} \ot p_2^*T_i,\]
        where $T_i$ is a line bundle on $\SX_{U_i}$ whose restriction to each fiber $\SX_b$ over $b\in U_i$ has trivial first Chern class in $H^2(\SX_b,\BQ)$;
        \item The Fourier--Mukai transform
    \[\Phi_{\CP^t}\colon D^b(X^t) \xrightarrow{\sim} D^b(\SX)_{(1)}\]
    is an equivalence.
    \end{enumerate}
\end{thm}
\begin{proof}
    Recall that the natural morphism $\CF^{\circ} \to B$ from the relative Fano variety of very good lines is smooth and surjective. Thus we can take an \'etale cover $U \to B$ such that the base change map $\CF^{\circ}_U \to U$ has a section $s$. Let $\widetilde{\CC}_U \to \CC_U$ be the pullback of $\widetilde{\CC}_{\CF^{\circ}} \to \CC_{\CF^{\circ}}$ along $s$. Then we have a unique isomorphism
    \[\phi_s \colon X_U \xrightarrow{\sim} \ol{\Prym}(\widetilde{\CC}_U/\CC_U),\]
    where $X_U$ is the base change of $X$ along $U \to B$. 

    \medskip
    \noindent {\bf Step 1.} We first show that $\CP$ \'etale locally satisfies the theorem of the square. 
    Let $V \to U$ be an \'etale morphism. We assume that $V$ is irreducible. Then the pullback of $\phi_s$ to $V$ identifies~$X_V$ and $ \ol{\Prym}(\widetilde{\CC}_V/\CC_V)$, which we still denote as $\phi_s$. 
    According to Proposition~\ref{prop: Poincare sheaf on OG10}, we have
    \begin{equation}\label{eq: identify G with Q}
        (\phi_s \times \phi_s)^* \ol{\CG}_V \cong \CP_V,
    \end{equation}
    where $\ol{\CG}_V$ is a normalized Poincar\'e sheaf on $ \ol{\Prym}(\widetilde{\CC}_V/\CC_V)$ and $\CP_V$ is the pullback of $\CP$ to~$X_V \times_V X_V$.  
    Since $X_V \to V$ has a section, the smooth locus of $X_V \to V$ is isomorphic to the base change $P_V\to V$ of the group scheme $P\to B$ \cite[Proposition~8.7]{AF16}. Hence we can regard $P_V$ as an open subvariety of $ X_V$ and the zero section of $P_V \to V$ can be regarded as a section of $X_V \to V$. 
    We have the commutative diagram
    \begin{equation*}
        \begin{tikzcd}
            P_V \times_V X_V \arrow[r, "\mu_V"] \arrow[d, "\phi_s|_{P_V} \times \phi_s"']
            & X_V \arrow[d, "\phi_s"]\\
            {\Prym}(\widetilde{\CC}_V/\CC_V) \times_V  \ol{\Prym}(\widetilde{\CC}_V/\CC_V) \arrow[r, "\mu"] 
            &  \ol{\Prym}(\widetilde{\CC}_V/\CC_V),
        \end{tikzcd}
    \end{equation*}
    where $\mu$ on the second line is the natural action of $\Prym(\widetilde{\CC}_V/\CC_V)$ on $\ol{\Prym}(\widetilde{\CC}_V/\CC_V)$. Then by Lemma~\ref{lem: thm of square} and \eqref{eq: identify G with Q} we have 
    \begin{equation}\label{eq: et local thm of square}
        (\mu_V \times \id_{X_V})^* \CP_V \cong p_{13}^*(\CP_V|_{P_V \times_V X_V}) \ot p_{23}^{*}(\CP_V)
    \end{equation}
    on $P_V \times_V X_V \times_V X_V$. 
    
    Consequently, if $u \colon V \to P_V$ is a section, then it defines an automorphism $a_u$ of $X_V$ over $V$ via the group action $\mu_V$. 
    Pulling back \eqref{eq: et local thm of square} along $u\times \id \times \id$ yields
    \[(a_u \times \id)^*\CP_V \cong  (u \times \id)^* \CP_V \ot \CP_V.\]
    Since $\CP$ is invertible on $P_V \times_V X_V$, $(u \times \id)^* \CP_V$ is a line bundle on $X_V$. 

    \medskip
    \noindent {\bf Step 2.} 
    We now prove property~(1). We apply Step~1 to the transition automorphisms that define the twist $X^t$ and use the resulting line bundles to construct the required $\mu_n$-gerbe. 
    
    For the given $t\in \Sha_{\et}$, we can take an \'etale cover $\{U_i\}_{i\in I}$ of $B$ such that it is a refinement of~$U \to B$ and that $t$ is represented by a \v{C}ech 1-cocycle $\{t_{ij} \in P(U_{ij})\}$. We can also assume that all the $U_i$ are irreducible. By abuse of notation, we also regard $t_{ij}$ as an element in~$\Aut(X_{U_{ij}})$. 
    
    Since $t$ is $n$-torsion, after refining the \'etale cover, there exist sections $c_i \in P(U_i)$ such that 
    \[t_{ij}^n = (c_i|_{U_{ij}} )(c_j^{-1}|_{U_{ij}}).\]
    Since the multiplication-by-$n$ map on $P\to B$ is \'etale locally surjective, after a further refinement of the cover, we may choose sections $r_i \in P(U_i)$ such
    that $r_i^n=c_i$. 
    Replacing $t_{ij}$ by~$t_{ij} (r_i^{-1}|_{U_{ij}}) (r_j|_{U_{ij}})$, we obtain a \v{C}ech $1$-cocycle representing the same class $t$ and satisfying
    \begin{equation}\label{eq: automorphisms are n-torsion}
        t_{ij}^n = 1 \in P(U_{ij}).
    \end{equation}   

    By the construction of $X^t$, there exist isomorphisms of schemes $\rho_i \colon X_{U_i}^t \xrightarrow{\sim} X_{U_i}$ such that over $U_{ij}$ we have
    \[t_{ij} = \rho_j\circ \rho_i^{-1} \colon X_{U_{ij}} \xrightarrow{\sim} X_{U_{ij}}.\]
    By Step~1, there exists an isomorphism
    \[a_{ij} \colon (t_{ij} \times \id)^* \CP_{U_{ij}} \xrightarrow{\sim} \CP_{U_{ij}} \ot p_2^* L_{ij}\]
    for some line bundle $L_{ij}$ on $X_{U_{ij}}$. It is clear that the line bundles $L_{ij}$ may be chosen so that $L_{ii} = \CO_{X_{U_{i}}}$ and $L_{ij} = L_{ji}^{-1}$. 
    From now on, we fix such choices of $\{L_{ij}\}$ and corresponding isomorphisms $\{a_{ij}\}$. 

    Consider the following diagram over $U_{ijk}$, where we omit all
    restriction symbols:
    \begin{equation*}
        \begin{tikzcd}
            (t_{ik} \times \id)^* \CP \arrow[rr, "a_{ik}"] \arrow[d, equal]
            & & \CP \ot p_2^{*}L_{ik} \\ 
            (t_{ij} \times \id)^* (t_{jk}\times \id)^* \CP \arrow[d, "(t_{ij} \times \id)^*a_{jk}"']
            &&\\
            (t_{ij} \times  \id)^* (\CP \ot p_2^* L_{jk}) \arrow[rr, "a_{ij}\ot \id"]
            && \CP \ot p_2^* L_{ij} \ot p_2^* L_{jk}
            \arrow[uu, dashed, "m_{ijk}"']  
        \end{tikzcd}
    \end{equation*}
    Here the left vertical equality is the canonical identification induced by $t_{ik}=t_{jk}\circ t_{ij}$, and $m_{ijk}$ is the unique isomorphism making the diagram commute. Restricting $m_{ijk}$ to $0\times_{U_{ijk}} X_{U_{ijk}}$, where $0$ is the zero section of $P_{U_{ijk}}$ regarded as a section of $X_{U_{ijk}}\to U_{ijk}$, we obtain an isomorphism 
    \[\mu_{ijk} \colon L_{ij} \ot L_{jk} \xrightarrow{\sim} L_{ik}.\]

    Since $t_{ij}^n = 1$, the iterated composition of $a_{ij}^{-1}$ induces an isomorphism
    \[\CP \ot p_2^*L_{ij}^{\ot n} \xrightarrow{(a_{ij}^{-1})^n} (t_{ij}^n \times \id)^* \CP = \CP.\]
    Restricting this isomorphism to $0\times_{U_{ij}}X_{U_{ij}}$ gives an isomorphism
    \[\tau_{ij}\colon L_{ij}^{\ot n} \xrightarrow{\sim} \CO_{X_{U_{ij}}}.\]

    The associativity condition for the isomorphisms of $\mu_{ijk}$ \eqref{eq: associativity of mu} follows from the equality
    \[t_{kl} \circ (t_{jk} \circ t_{ij}) = t_{il} =  (t_{kl} \circ t_{jk}) \circ t_{ij}. \]
    Take $M_i = \CO_{X_{U_i}}$. 
    Since the isomorphisms $\tau_{ij}$ and $\mu_{ijk}$ are both constructed from the same collection of isomorphisms $a_{ij}$, they are compatible, i.e.~the diagram \eqref{diag: tau and mu compatible} commutes. 
    Therefore, with respect to the \'etale cover $\{X_{U_i}\to X\}$, the data $(L_{ij}, M_i, \mu_{ijk}, \tau_{ij})$ give a Hitchin presentation of a cohomology class $\alpha_t \in H^2(X,\mu_n)$ by Appendix~\ref{sec: Hitchin presentation}. 
    
    Let $F_i \coloneq (\rho_i \times_{U_i} \id_{X_{U_i}})^*(\CP_{U_i})$. 
    Then $a_{ij}$ induces an isomorphism over $U_{ij}$
    \[\psi_{ij} \colon F_i \ot p_2^*L_{ij} \xrightarrow{(\rho_i \times \id)^*(a_{ij}^{-1})} (\rho_i \times \id)^* (t_{ij} \times \id)^* \CP_{U_{ij}}  = (\rho_j \times \id)^* \CP_{U_{ij}}  = F_j.\]
    By construction, the isomorphisms $\psi_{ij}$ satisfy the twisted cocycle condition \eqref{eq: cocycle for twisted sheaf}. 
    Therefore, by Sections~\ref{sec: stack of explicit twisted torsors} and \ref{sec: twisted sheaves on gerbes}, there exists a $\mu_n$-gerbe 
    \[\sigma\colon \SX \to X\]
    representing $\alpha_t$ and a weight-1 sheaf $\CP^t$ on $X^t \times_B \SX$ such that, over $U_i$, we have 
    \[\CP^t_{U_i} \cong (\id_{X^t_{U_i}} \times_{U_i} \sigma_{U_i})^* F_i \ot p_2^* T_i \cong (\rho_i\times_{U_i}\sigma_{U_i})^* \CP_{U_i} \ot p_2^*T_i,\]
    where $T_i$ is an $n$-torsion line bundle on $\SX_{U_i}$. This proves property~(1).

    \medskip
    \noindent {\bf Step 3.}
    To prove (2), set 
    \[(\CP^t)^{-1} \coloneq R\hom_{X^t\times_B \SX}(\CP^t, p_1^{*}\omega_{\pi^t})[5],\]
    where $\omega_{\pi^t}$ is the relative canonical bundle of $\pi^t \colon X^t \to B$. Then the Fourier--Mukai transform~$\Phi_{(\CP^t)^{-1}}$ is the right adjoint of $\Phi_{\CP^t}$. Over $U_i$ we have 
    \[(\CP^t)_{U_i}^{-1} \cong (\rho_i \times_{U_i} \sigma)^* \CP_{U_i}^{-1} \ot p_2^*T_i^{-1}.\]
    Notice that we have natural adjunction maps
    \[\id_{D^b(X^t)} \to \Phi_{(\CP^t)^{-1}} \circ \Phi_{\CP^t}, \quad   \Phi_{\CP^t} \circ \Phi_{(\CP^t)^{-1}} \to \id_{D^b(\SX)_{(1)}}.\]
    After pulling back to each $U_i$, these morphisms are isomorphisms by the untwisted local calculation. Since being an isomorphism can be checked \'etale locally, they are isomorphisms globally.
\end{proof}

We can now follow the same argument as in \cite[Section~4.4]{MSY23} to show that Corollary~\ref{cor: motivic multiplicativity} holds for $\pi^t \colon X^t \to B$. We briefly explain why their argument applies in our setting. 

We first verify that $\SX$ is a Deligne--Mumford quotient stack. Since $X$ is projective, Gabber's theorem \cite{Gabber03} identifies the Brauer group with the cohomological Brauer group
\[\Br(X) = H^2_{\et}(X, \BG_m)_{\mathrm{tor}}.\]
Therefore, by \cite[Theorem~3.6]{Edidin_etal_01} the $\mu_n$-gerbe $\SX$ over $X$ is a quotient stack. Moreover, by construction $\SX\to X$ is a $\mu_n$-gerbe and $\mu_n$ is finite \'etale in characteristic zero. By \cite[Tag~06QB]{stacks-project}, $\SX$ is a Deligne--Mumford stack.

Once we know that $\SX$ is a Deligne--Mumford quotient stack, Step~1 of \cite[Section~4.4]{MSY23} applies, using Theorem~\ref{thm: equivalence for twisted derived category}~(2) in place of \cite[Proposition~4.2]{MSY23}. Then Steps~2 and 3 apply identically. Finally, Step~4 carries over using Theorem~\ref{thm: equivalence for twisted derived category}~(1) in place of \cite[Corollary~4.4]{MSY23}.

\appendix
\section{Gerbes and twisted sheaves}
Let $X$ be a scheme over $\BC$. Let $A$ be a sheaf of commutative groups on $X$ with a character 
\[\chi\colon A \to \BG_m.\]\
In this section, we consider the cases where $A=\mu_n$ with the natural character $\chi\colon \mu_n \hookrightarrow \BG_m$, or $A=\BG_m$ with the character $\id \colon \BG_m \to \BG_m$. 
In these cases, equivalence classes of $A$-gerbes are classified by $H^2(X_{\et},A)$. 

Let $\alpha \in H^2(X,A)$ be represented by an \'etale cover $\CU = \{U_i\}_{i\in I}$ of $X$ and a \v{C}ech 2-cocycle 
\[a = (a_{ijk}) \in \Gamma(U_{ijk},A).\] 
An $\alpha$-twisted sheaf defined by C\u{a}ld\u{a}raru is a collection of $\CO_{U_i}$-modules $F_i$, together with transition functions $\psi_{ij} \colon F_i|_{U_{ij}} \xrightarrow{\sim} F_j|_{U_{ij}}$ that satisfy the cocycle condition up to the 2-cocycle~$a$; see \cite[Section~1.2]{Caldararu_thesis}. We refer to such an object as a C\u{a}ld\u{a}raru $\alpha$-twisted sheaf.

Fix an $A$-gerbe $\SX_{\alpha} \to X$ that represents the cohomology class $\alpha$. 
Following Lieblich, one may define an $\alpha$-twisted sheaf to be a sheaf of weight $1$ on $\SX_{\alpha}$, i.e., an $\CO_{\SX_{\alpha}}$-module on which the stabilizer acts through the standard character $\chi$. We call such sheaves Lieblich $\alpha$-twisted sheaves. Lieblich proved that the stack of Lieblich $\alpha$-twisted sheaves is naturally equivalent to the stack of C\u{a}ld\u{a}raru $\alpha$-twisted sheaves; see \cite[Proposition~2.1.3.11]{Lieblich_thesis}.

In the main body of the paper, the cohomology classes we encounter are not given directly by \v{C}ech 2-cocycles. Instead, they arise through Hitchin's equivalent description, where the cocycle is encoded by line bundles on double intersections together with additional data. We call such data a Hitchin presentation. 
In this appendix, we follow \cite[Section~2.1.3]{Lieblich_thesis} and give an explicit description, in Hitchin's notation, of the Lieblich $\alpha$-twisted sheaf corresponding to a given C\u{a}ld\u{a}raru $\alpha$-twisted sheaf.

\subsection{Hitchin presentations of gerbe classes}\label{sec: Hitchin presentation}
We first recall Hitchin's presentation of a class in $H^2(X,\BG_m)$; see \cite[Section~1.2]{Hitchin99_Lecture_Special_Lagrangian}, \cite[Section~1.1]{Caldararu_thesis}. 
Let $\CU = \{U_i\}_{i\in I}$ be an \'etale cover of $X$, and write 
\[U_{ij} = U_i \times_X U_j,\quad U_{ijk} = U_i \times_X U_j \times_X U_k.\] 
For simplicity, we omit restriction symbols when there is no confusion. 

A cohomology class in $H^2(X,\BG_m)$ can be represented, after choosing an \'etale cover, by the following data:
\begin{enumerate}
    \item line bundles $L_{ij}$ on $U_{ij}$ such that $L_{ii} = \CO_{U_i}$ and $L_{ij} = L_{ji}^{-1}$;
    \item isomorphisms $\mu_{ijk} \colon L_{ij}|_{U_{ijk}} \ot L_{jk}|_{U_{ijk}} \xrightarrow{\sim} L_{ik}|_{U_{ijk}}$ satisfying the associativity condition 
    \begin{equation}\label{eq: associativity of mu}
        \mu_{ikl}\circ(\mu_{ijk} \ot \id) = \mu_{ijl}\circ(\id \ot \mu_{jkl}).
    \end{equation}
\end{enumerate}

Similarly, a class in $H^2(X,\mu_n)$ can be represented by the following data:
\begin{enumerate}
    \item Hitchin data $(L_{ij}, \mu_{ijk})$ representing a class in $H^2(X,\BG_m)$;
    \item line bundles $\{M_i\}$ on $U_i$ together with isomorphisms 
    \[\tau_{ij}\colon M_i|_{U_{ij}} \ot L_{ij}^{\ot n} \xrightarrow{\sim} M_j|_{U_{ij}}\] 
    that are compatible with $\mu_{ijk}$ in the sense that the following diagram commutes over $U_{ijk}$:
\begin{equation}\label{diag: tau and mu compatible}
    \begin{tikzcd}
        M_i \ot L_{ij}^{\ot n} \ot L_{jk}^{\ot n} \arrow[d, "\tau_{ij} \ot \id"'] \arrow[r, "\id \ot \mu_{ijk}^{\ot n}"] & M_i \ot L_{ik}^{\ot n} \arrow[d, "\tau_{ik}"] \\
        M_j \ot L_{jk}^{\ot n}  \arrow[r, "\tau_{jk}"'] & M_k.   
    \end{tikzcd}
\end{equation}
\end{enumerate}

Let us explain why this is equivalent to the usual \v{C}ech cocycle representation.  
After a refinement of the cover $\CU$, we can assume that the line bundles $L_{ij}$ and $M_i$ are trivial. After choosing trivializations for $L_{ij}$ and $M_i$, the isomorphisms $\mu_{ijk}$ and $\tau_{ij}$ are given by sections 
\[c_{ijk} \in \Gamma(U_{ijk}, \BG_m), \quad  a_{ij} \in \Gamma(U_{ij},\BG_m).\] 
The associativity condition for $\mu_{ijk}$ implies that $c= (c_{ijk})$ is a 2-cocycle in~$ Z^2(\CU, \BG_m)$, while the compatibility of $\tau_{ij}$ with $\mu_{ijk}$ implies that $a_{ij} a_{jk} = a_{ik} c_{ijk}^n$. 
Refining the cover if necessary, we may take $b_{ij} \in \Gamma(U_{ij}, \BG_m)$ such that $b_{ij}^n=a_{ij}$. Then $c'_{ijk} = c_{ijk} (b_{ij} b_{jk} b_{ik}^{-1})^{-1}$ defines a \v{C}ech 2-cocycle $c' = (c'_{ijk})$ in $ Z^2(\CU, \mu_n)$, which represents an element in $H^2_{\et}(X,\mu_n)$. 
Changing the choices of the $n$-th roots $b_{ij}$ changes $c'$ by a \v{C}ech 2-coboundary in $ B^2(\CU, \mu_n)$ and thus represents the same cohomology class. One can also easily show that the resulting class is independent of the chosen refinement of the \'etale cover. 

In particular, if $\alpha \in H^2(X,\mu_n)$ is a class represented by the data $(L_{ij}, M_i, \mu_{ijk}, \tau_{ij})$, then its image under $i_*\colon H^2(X,\mu_n)\to H^2(X,\BG_m)$ is represented by the Hitchin data $(L_{ij}, \mu_{ijk})$, where $i\colon \mu_n \hookrightarrow \BG_m$ is the natural inclusion.  

\subsection{The stack of explicit twisted torsors}\label{sec: stack of explicit twisted torsors}
From now on, we fix a class $\alpha \in H^2(X,\mu_n)$ and a Hitchin presentation $(L_{ij},M_i,\mu_{ijk},\tau_{ij})$ of $\alpha$ with respect to the \'etale cover 
$\CU=\{U_i\}_{i\in I}$. We denote by $\beta$ the image of $\alpha$ in $H^2(X,\BG_m)$. 

We first recall the definition of C\u{a}ld\u{a}raru $\alpha$-twisted sheaf on $X$ with respect to the given Hitchin presentation. 

\begin{defn}(C\u{a}ld\u{a}raru)\label{def: Caldararu twisted sheaf}
    A C\u{a}ld\u{a}raru $\alpha$-twisted sheaf on $X$ is a collection of $\CO_{U_i}$-modules~$F_i$, together with isomorphisms 
    $\psi_{ij} \colon F_i|_{U_{ij}} \ot L_{ij} \to F_j|_{U_{ij}}$
    that are compatible with~$\mu_{ijk}$ in the sense that the following diagram commutes over $U_{ijk}$:
    \begin{equation}\label{eq: cocycle for twisted sheaf}
    \begin{tikzcd}
        F_i \ot L_{ij} \ot L_{jk} \arrow[d, "\psi_{ij} \ot \id"'] \arrow[r, "\id \ot \mu_{ijk}"] & F_i \ot L_{ik} \arrow[d, "\psi_{ik}"] \\
        F_j \ot L_{jk}  \arrow[r, "\psi_{jk}"'] & F_k.   
    \end{tikzcd}
\end{equation}
\end{defn}
The data $M_i$ and $\tau_{ij}$ do not appear in this definition. Therefore, a C\u{a}ld\u{a}raru $\alpha$-twisted sheaf is the same as a C\u{a}ld\u{a}raru $\beta$-twisted sheaf. 

We now follow \cite[Section~2.1.3]{Lieblich_thesis} to construct a $\mu_n$-gerbe with class $\alpha$. Let $\SX(\CU, \alpha)$ be the stack of explicit twisted right $\mu_n$-torsors over $X$ defined as follows.
For a morphism $S \to X$, we set 
\[S_i = S \times_X U_i, \quad S_{ij} = S \times_X U_{ij}, \quad S_{ijk} = S \times_X U_{ijk}.\] 
For simplicity, we use the same notation for the pullbacks of $L_{ij}$, $M_i$, $\mu_{ijk}$, and $\tau_{ij}$ to these schemes. 
An object of $\SX(\CU, \alpha)(S)$ consists of the following data: 
\begin{enumerate}
    \item a line bundle $N_i$ on $S_i$ for each $i\in I$;
    \item isomorphisms $\phi_{ij} \colon N_i \ot L_{ij}|_{S_{ij}} \to N_j|_{S_{ij}}$ such that the following diagram commutes over $S_{ijk}$:  
        \begin{equation*}
        \begin{tikzcd}
        N_i \ot L_{ij} \ot L_{jk} \arrow[d, "\phi_{ij} \ot \id"'] \arrow[r, "\id \ot \mu_{ijk}"] & N_i \ot L_{ik} \arrow[d, "\phi_{ik}"] \\
        N_j \ot L_{jk}  \arrow[r, "\phi_{jk}"'] & N_k;   
        \end{tikzcd}
        \end{equation*}
        \item $\theta_i \colon N_i^{\ot n} \xrightarrow{\sim} M_i|_{S_i}$ for all $i\in I$ such that the following diagram commutes over $S_{ij}$:
        \begin{equation*}
            \begin{tikzcd}
            N_i^{\ot n} \ot L_{ij}^{\ot n} \arrow[d, "\theta_i \ot \id"'] \arrow[r, "\phi_{ij}^{\ot n}"] & N_j^{\ot n}  \arrow[d, "\theta_j"] \\
            M_i \ot L_{ij}^{\ot n}  \arrow[r, "\tau_{ij}"'] & M_j.   
            \end{tikzcd}
        \end{equation*}
\end{enumerate}
A morphism in $\SX(\CU, \alpha)(S)$ is a collection of isomorphisms $N_i \cong N_i'$ compatible with $\phi_{ij}$ and ~$\theta_i$.

By \cite[Proposition~2.1.3.4]{Lieblich_thesis}, $\SX(\CU,\alpha)$ is a $\mu_n$-gerbe over $X$ with class $\alpha$. 
For simplicity, we denote $\SX_{\alpha} = \SX(\CU,\alpha)$.  
Let $\sigma\colon \SX_{\alpha} \to X$ be the natural map. We also set $\SX_{\alpha,i} = \SX_{\alpha} \times_X U_i$ and $\sigma_i \colon \SX_{\alpha,i} \to U_i$. 

\subsection{Twisted sheaves on gerbes}\label{sec: twisted sheaves on gerbes}
We now specialize to the case where each $M_i$ is trivial. After choosing a trivialization of each $M_i$, the Hitchin presentation of $\alpha$ gives isomorphisms
\[\tau_{ij} \colon L_{ij}^{\ot n} \xrightarrow{\sim} \CO_{U_{ij}}\]
compatible with $\mu_{ijk}$. 

For each $i \in I$, there is an object $s_{i} \in \SX_{\alpha}(U_{i})$ defined by the twisted $\mu_n$-torsor 
\[(L_{ij}, \mu_{ijk},\tau_{ij})_{j,k \in I}\]
over $U_{i}$. 
It gives a section of $\SX_{\alpha, i} \to U_{i}$, and hence identifies $\SX_{\alpha, i}$ with the trivial $\mu_n$-gerbe over $U_i$. 
Using these sections, we define a tautological right $\mu_n$-torsor $P_i$ on $\SX_{\alpha,i}$ as follows. For a scheme $S$ over $U_i$ and an object $\xi \in \SX_{\alpha,i}(S)$, we set 
\[P_i(\xi) = \ul{\Isom}(\xi, s_i|_S).\]
Let $T_i = P_i \times^{\mu_n} \BA^1$ be the associated line bundle on $\SX_{\alpha,i}$. Then $T_i$ is an $n$-torsion line bundle on $\SX_{\alpha,i}$, equipped with a trivialization of $T_i^{\ot n}$, and has weight 1. 

From now on we identify the section $s_i$ with the corresponding right twisted $\mu_n$-torsor over~$U_i$. 
For $i,j \in I$, we denote by $s_i|_{U_{ij}} \in \SX_{\alpha}(U_{ij})$ the pullback of $s_i$ to $U_{ij}$. Then $s_i|_{U_{ij}} \ot L_{ji}$ is naturally an object of $\SX_{\alpha}(U_{ij})$. Also, we have a morphism 
\[s_i|_{U_{ij}} \ot L_{ji} \to s_{j}|_{U_{ij}}\]
in $\SX_{\alpha}(U_{ij})$ given by a collection of isomorphisms
\[\mu_{jik}\colon L_{ik}|_{U_{ijk}} \ot L_{ji}|_{U_{ijk}} \xrightarrow{\sim} L_{jk}|_{U_{ijk}}, \quad \forall k\in I,\]
which induces an isomorphism of weight-1 sheaves on $\SX_{\alpha,ij}$:
\[\alpha_{ij} \colon T_i|_{\SX_{\alpha,ij}} \xrightarrow{\sim}  \sigma_{ij}^{*}L_{ij} \ot T_j|_{\SX_{\alpha,ij}}.\]
By construction we have 
\begin{equation}\label{eq: cocycle for tauto l.b.}
    \alpha_{ik} = (\mu_{ijk} \ot \id_{T_k}) \circ (\id_{L_{ij}} \ot \alpha_{jk}) \circ \alpha_{ij}.
\end{equation}

By \cite[Proposition~2.1.3.11]{Lieblich_thesis}, the stack of Lieblich $\alpha$-twisted sheaves is naturally equivalent to the stack of C\u{a}ld\u{a}raru $\alpha$-twisted sheaves. We now describe this equivalence explicitly in the present notation. 

Let $(F_i, \psi_{ij})$ be a C\u{a}ld\u{a}raru $\alpha$-twisted sheaf. 
Consider the weight-1 sheaf $\sigma_i^{*}F_i \ot T_i$ on $\SX_{\alpha,i}$ and the transition function over $\SX_{\alpha,ij}$:
\[\Phi_{ij} = (\sigma_{ij}^{*}\psi_{ij} \ot \id_{T_j}) \circ(\id_{\sigma_i^{*}F_i} \ot \alpha_{ij}) \colon \sigma_i^{*}F_i \ot T_i \xrightarrow{\sim} \sigma_j^{*}F_j \ot T_j.\]
Set $\widetilde{\psi}_{ij}=\sigma_{ij}^*\psi_{ij}$ and $\widetilde{\mu}_{ijk}=\sigma_{ijk}^*\mu_{ijk}$.
By \eqref{eq: cocycle for twisted sheaf} and \eqref{eq: cocycle for tauto l.b.}, over the triple intersection $\SX_{\alpha,ijk}$ we have
\begin{align*}
    \Phi_{jk} \circ \Phi_{ij} 
    & = \Big[ (\widetilde{\psi}_{jk} \otimes \id_{T_k}) \circ (\widetilde{\psi}_{ij} \otimes \id_{\sigma^{*}L_{jk} \otimes T_k}) \Big] \circ \Big[ (\id_{\sigma^{*}F_i \otimes \sigma^{*}L_{ij}} \otimes \alpha_{jk}) \circ (\id_{\sigma^{*}F_i} \otimes \alpha_{ij}) \Big]\\
    & = \Big[(\widetilde{\psi}_{ik} \ot \id_{T_k}) \circ (\id_{\sigma^{*}F_i} \ot \widetilde{\mu}_{ijk} \ot \id_{T_k}) \Big] \circ \Big[ (\id_{\sigma^{*}F_i} \ot \widetilde{\mu}_{ijk}^{-1} \ot \id_{T_k}) \circ (\id_{\sigma^* F_i} \ot \alpha_{ik}) \Big]\\
    & = \Phi_{ik}.
\end{align*}
Therefore the sheaves $\sigma_i^{*}F_i \ot T_i$ glue to a weight-1 sheaf $\CF_{\alpha}$ on $\SX_{\alpha}$. This is the Lieblich $\alpha$-twisted sheaf corresponding to the C\u{a}ld\u{a}raru $\alpha$-twisted sheaf $(F_i,\psi_{ij})$.

\raggedright
\bibliography{ref}
\bibliographystyle{acm}
\end{document}